\documentclass[11pt]{article}
\usepackage{amsmath, amssymb, theorem, latexsym, epsfig}
 \usepackage{ifpdf}
\numberwithin{equation}{section}

\theoremstyle{plain}
\theorembodyfont{\itshape}
\newtheorem{theorem}{Theorem}[section]
\newtheorem{proposition}[theorem]{Proposition}
\newtheorem{lemma}[theorem]{Lemma}
\newtheorem{corollary}[theorem]{Corollary}

\theorembodyfont{\rmfamily}
\newtheorem{definition}[theorem]{Definition}

\newtheorem{example}[theorem]{Example}
\newtheorem{remark}[theorem]{Remark}

\newtheorem{conjecture}[theorem]{Conjecture}
\newenvironment{proof}{{\noindent \textbf{Proof}\,\,}}{\hspace*{\fill}$\Box$\medskip}

\def\dist{\operatorname{dist}}
\def\cc{\mathbb C}
\def\la{\lambda}

\def\oc{\overline{\cc}}

\def\cp{\mathbb{CP}}
\def\wt#1{\widetilde#1}
\def\rr{\mathbb R}
\def\var{\varepsilon}

\def\mcr{\mathcal R}

\def\nn{\mathbb N}

 \def\zz{\mathbb Z}
 \def\rp{\mathbb{RP}}

\def\La{\Lambda}

\def\mca{\mathcal A}
  
 \def\la{\lambda}
\def\mcc{\mathcal C}
\def\mcs{\mathcal S}

\def\mcq{\mathcal Q}

\def\mcr{\mathcal R}
\def\mcp{\mathcal P}
\def\mcn{\mathcal N}

\def\Xi{\mathcal Z}

\def\gdll{\Gamma_{d,\la}}

\def\gbll{\Gamma_{b1,\la}}

\def\mcq{\mathcal Q}
\def\gblct{\Gamma_{b1,\frac43}}
\def\gblctl{\Gamma_{b1,\frac43,1}}
\def\gblctd{\Gamma_{b1,\frac43,2}}
\def\gbllt{\Gamma_{b1,\frac13}}

\title{On exotic rationally integrable planar dual  billiards I. Complex geometry and type of dynamics}
\author{Alexey Glutsyuk\thanks{Higher School of Modern Mathematics MIPT, Moscow, Russia}
\thanks{HSE University, Moscow, Russia}
\thanks{CNRS, UMR 5669 (UMPA, ENS de Lyon), Lyon, France}\thanks{The results of the project "Symmetry. Information. Chaos" (HSE-BR-2025-84), carried out within the framework of the Basic Research Program at HSE University, are presented in this work.}}
 \begin{document}
 \maketitle
\begin{abstract} A  planar {\it dual billiard} is a planar curve $\gamma$ equipped with a family 
$(\sigma_P)|_{P\in\gamma}$ 
of projective involutions of the projective lines $L_P$ tangent to $\gamma$ at $P$ that fix $P$.  A dual billiard is called {\it rationally integrable,} 
 if there exists a rational function $R(x,y)$ of two variables (called first integral) whose restriction to each tangent line $L_P$ 
 is $\sigma_P$-invariant. In the previous author's paper it was shown that rationally integrable dual billiards 
 exist only on conics punctured at $k$ points, $0\leq k\leq 4$. Their classification given there  includes standard examples with quadratic integrals,   defined by conical pencils, and an infinite family of exotic examples with minimal degree of integrals being any even number greater than two. 
  In the present paper we study a question stated by Dmitry Treschev on dynamics and conservation laws in the exotic examples. We study the  dynamics acting on the two-dimensional phase space: the complex algebraic surface 
 consisting of pairs $(Q,P)$, where $Q\in\cp^2\setminus\gamma$, $P\in\gamma$ and the line $QP$ 
  is tangent to $\gamma$ at $P$. We show that the phase space is fibered by invariant algebraic curves 
  along which $Q$ lies in a level curve of the integral of the dual billiard. For each example we find the type of a generic level curve of the integral and of the corresponding fiber.  We show that a generic level curve is rational  in most of examples and elliptic  in two examples. We present formulas for an invariant area form on the phase space and for invariant holomorphic differentials on invariant curves. This yields a formula for holomorphic differentials on the elliplic level curves of the integral.  \end{abstract}
\tableofcontents
\section{Introduction} 
\subsection{Brief description of main results, historical remarks and plan of the paper}
Consider a strictly convex closed planar curve $C$. The billiard reflection acts on the 
oriented lines intersecting $C$. Namely, let us reflect from $C$ an oriented line $L$ at its last intersection point with $C$ 
with respect to orientation, according to the standard reflection law: the angle of incidence is equal to the angle of reflection; the reflected line is directed inside the convex domain 
bounded by $C$.   Recall that a {\it caustic} of the billiard on $C$ is a curve 
$\alpha$ such that each tangent line to $\alpha$ is reflected from the curve $C$ to 
a tangent line to $\alpha$. The billiard on $C$ is called 
{\it Birkhoff integrable,} if  a topological annulus adjacent to $C$ from the convex side is foliated by closed caustics, and $C$ is a leaf of this foliation. 
 It is well-known that each elliptic billiard is integrable:  ellipses confocal to the boundary are caustics, 
 see \cite[section 4]{tab}. 
 The {\bf Birkhoff Conjecture}  states the converse: {\it the only Birkhoff caustic-integrable 
 convex bounded planar 
 billiards  are  ellipses.}  It was studied by many mathematicians including H.Poritsky, 
 M.Bialy, S.Bolotin, V.Kaloshin, A.Sorrentino,  A.Mironov, I.Koval, the author and others. 
 
 Let us mention some of well-known results on the Birkhoff Conjecture.  In 1950 H.Poritsky \cite{poritsky}, and later 
E.Amiran \cite{amiran} in 1988 proved it under the additional assumption that 
 the billiard in each closed caustic near the boundary has the same closed caustics, as the initial billiard. 
 In 1993 M.Bialy \cite{bialy} proved that if the phase cylinder of the billiard in a domain $\Omega$ is 
 foliated  by non-contractible continuous closed curves which are invariant under the billiard map,  then the boundary 
 $\partial\Omega$ is a circle. See also \cite{wojt} for another proof. V.Kaloshin and A.Sorrentino \cite{kalsor} proved that {\it every  integrable deformation of 
 an ellipse is an ellipse,} under the condition that the family of caustics extends up to a caustic tangent to triangular orbits.  A generalization of the result of Kaloshin and Sorrentino, without extension condition but for a generic ellipse, was recently proved by I.Koval \cite{koval}. M.Bialy and A.E.Mironov \cite{bm6} 
 proved the Birkhoff Conjecture for centrally-symmetric 
 billiards having a family of closed caustics that extends up to a caustic tangent to four-periodic orbits.  D.V.Treschev \cite{treshchev, tres2, tres3}  stated and investigated an analogue  of the Birkhoff Conjecture  for so-called locally integrable local billiards with reflection from pairs of germs of analytic curves and their higher-dimensional analogues.

 S.V.Bolotin \cite{bolotin, bolotin2} suggested the following {\bf polynomial version of the Birkhoff Conjecture.} 
 {\it Let  the billiard flow in the domain bounded by a closed curve $C$ have a non-trivial first integral polynomial in the velocity:} in this case the billiard is called {\it polynomially integrable.} Then {\it $C$ is an ellipse.} The solution of Bolotin's Conjecture is a joint result of M.Bialy, A.Mironov and the author \cite{bm, bm2, gl2}. 
  For other results on the Birkhoff Conjecture see the papers \cite{bialy1, bialy2, bm4, dr, marco}, the survey papers \cite{KS18, bm5}, recent papers \cite{grat, gl4}  and references therein. 
  
   A.P.Veselov \cite{veselov} proved complete integrability of billiards bounded by confocal quadrics in any dimension, with quadratic integrals in involution, and analogous results on billiards in space forms of constant curvature \cite{veselov2}. He proved that their dynamics  corresponds to a shift of the Jacobi variety of an appropriate hyperelliptic curve \cite[p.99, theorems 3, 2]{veselov2}. 
See also \cite{kozlov} for a survey on other results on polynomially integrable billiards.

 S.Tabachnikov introduced projective billiards \cite{tabpr} generalizing billiards on space forms of constant curvature. He 
 suggested a generalized version of the Birkhoff Conjecture for projective billiards and for their projective-dual objects: the dual billiards \cite{tab08}, see the next definition.  

Let $\gamma\subset\rp^2$ be a planar curve. For every point $P\in\gamma$ by $L_P$ we denote the projective tangent line to $\gamma$ at $P$.  

\begin{definition} A  planar {\it dual billiard} is a planar curve $\gamma$ equipped with a family 
$(\sigma_P)|_{P\in\gamma}$ 
of projective involutions of the tangent lines $L_P$ that fix $P$.  
\end{definition}

 \begin{conjecture} \label{conjtd} {\bf (S.Tabachnikov \cite{tab08}).}  Let $\gamma$ be a strictly convex planar curve equipped with an {\bf integrable} dual billiard structure. This means that there exist a topological annulus adjacent to $\gamma$ from the concave side that is foliated by strictly convex closed curves-leaves such that for every $P\in\gamma$ the involution $\sigma_P:L_P\to L_P$ 
 permutes the intersection points of the line $L_P$ with each leaf. See Fig. 1. Then 
the underlying curve $\gamma$ and the leaves  are conics forming a pencil.
\end{conjecture} 
 \begin{figure}[ht]
  \begin{center}
  \vspace{-0.5cm}
   \epsfig{file=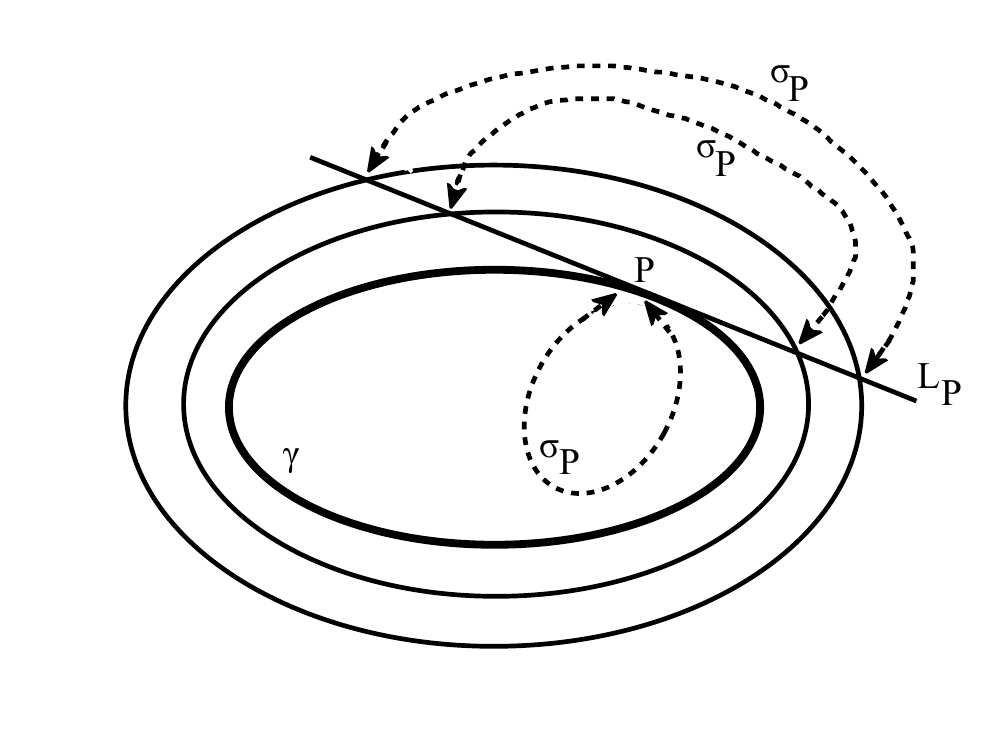, width=20em}
   \vspace{-0.7cm}
   \caption{An integrable dual  billiard structure. Figure taken from the author's paper \cite[Fig. 1]{grat}.}
   \vspace{-0.7cm}
        \label{fig:5}
  \end{center}
\end{figure}

A basic example of a dual billiard is an outer billiard, where the involutions are central symmetries with respect to the tangency points. It was shown in \cite{gs} that if a dual billiard admits a polynomial first integral, i.e., a polynomial whose restriction to 
each tangent line is invariant under the corresponding involution, then the underlying curve is a conic. 

A partial positive answer to Tabachnikov's Conjecture \ref{conjtd} was proved by the author in \cite{grat}, under the condition that 
the foliation admits a non-constant rational first integral. This condition implies in particular that the dual billiard is rationally integrable, see the following definition.

\begin{definition} A dual billiard is called {\it rationally integrable,} 
 if there exists a rational function $R(x,y)$ of two variables, called first integral, whose restriction to each tangent line $L_P$ 
 is $\sigma_P$-invariant. 
 \end{definition}
 \begin{example} \label{exint} see \cite[section 1]{grat}.  Let $\gamma, \alpha\subset\rp^2$ be two distinct conics, 
 $\gamma$ be smooth. Then $\alpha$ defines the dual billiard on $\gamma\setminus\alpha$, with $\sigma_P:L_P\to L_P$ fixing $P$ and permuting the intersection points of the line $L_P$ with $\alpha$. It is called a {\it conical pencil type} dual billiard. It is integrable. Indeed, consider the pencil of conics containing $\gamma$ and $\alpha$. It is defined by two  symmetric $3\times 3$-matrices $A$ (non-degenerate) and $B$ and consists of conics 
$$\mcc_\lambda=\{<(A-\lambda B)x,x>=0\}\subset\rp^2_{[x_1:x_2:x_3]}, \ \ x=(x_1,x_2,x_3), \ \ \mcc_0=\gamma,$$
$$ < , > \text{ is the Euclidean scalar product in } \ \rr^3.$$
For every $P\in\gamma\setminus\alpha$ and every $\lambda$ the involution $\sigma_P:L_P\to L_P$ extended to the complex line $L_P$  permutes the points of  its intersection with the complex conic $\mcc_\lambda$  (Desargues' Theorem, see  \cite{berger87}). Therefore, a quadratic rational function whose level curves are conics of the pencil is an integral of the dual billiard. 
\end{example}

 In the previous author's paper \cite{grat} 
 it was shown that rationally integrable dual billiards 
 exist only on conical arcs and extend to a {\it singular holomorphic dual billiard} on a complex conic $\gamma$:  a holomorphic family of complex projective involutions $\sigma_P$ fixing $P$ of complex projective tangent lines $L_P$ to $\gamma$, defined for all $P\in\gamma$ 
  except for a finite number of points: singularities of dual billiard.  The number of singularities is at most four. The author presented the classification of rationally integrable singular holomorphic dual billiards in the real domain up to projective transformation. 
It includes the above standard conical pencil type examples with quadratic integrals, and an infinite family of exotic examples with minimal degree of integrals being any even number greater than two. 
 
  In the present paper we study a question stated by Dmitry Treschev on dynamics and conservation laws in the exotic examples. We study complex geometry and dynamics of the exotic dual billiards on the conic 
$$\gamma=\{ w=z^2\}\subset\rr^2\subset\cc^2\subset\cp^2.$$
Namely let $\sigma_P$ be the dual billiard involutions.  
We study the dynamics of the map $F:M\to M$ acting as follows on the phase space 
\begin{equation}M:=\{(Q,P) \ | \ Q\in\cp^2, \ P\in\gamma, \ Q\in L_P\}:\label{phspace}
\end{equation}
\begin{equation}F:(Q,P)\mapsto(Q', P'), \ \ Q'=\sigma_P(Q), \  P'\neq P \  \text{ if } Q'\notin\gamma.\label{bmap}\end{equation}
The phase space $M$  is a rational surface fibered over $\gamma$ by projective lines. It is a holomorphically trivial fibration. 
The trivialization is the biholomorphism $M\mapsto\gamma\times\gamma$, $\gamma\simeq\oc$, that sends each pair $(Q,P)$ to $(P,\wt P)$, where 
$\wt P\in\gamma$ is the point of tangency with $\gamma$ of the other line $L\neq QP$ through $Q$ tangent to $\gamma$. 
This map, which is defined outside the diagonal,  clearly extends holomorphically to the diagonal.

We show that the above map $F$ is  birational, see Proposition \ref{prbir}. 

In what follows by $p_M$ we denote the canonical projection to $\cp^2$: 
$$p_M:M\to\cp^2, \ \ \  p_M(Q,P)=Q.$$
This is a double covering of complex surfaces branched along the curve $\gamma$. 

 The surface $M$ is fibered by $F$-invariant algebraic curves: preimages of level curves of the integral of the dual billiard under the projection $p_M$.  They will be called {\it invariant fibers.} 

It is well-known that in the case, when the dual billiard is of conical pencil type and is defined by a pencil of conics passing through given four distinct points, a generic invariant fiber is an elliptic curve, and the dual billiard map $F$ acts on it as a translation of complex torus, see \cite{gh}.

  For each exotic  rationally integrable dual billiard we find the type of a generic  level curve of the integral and of the corresponding invariant fiber.  Birationality of the dual billiard map $F$ implies that a generic fiber is either rational, or elliptic curve. 
  We show that in all the exotic examples except for two ones  generic level curves are rational and find their explicit rational parametrizations. 
  We show that in the  two remaining  examples generic level curves are elliptic. We also show that in all the exotic examples except two each invariant fiber is either elliptic, or a union of two elliptic curves.  These results are stated in Section 2 and proved in Section 3. 
  
  For  the above-mentioned elliptic curves we find the corresponding holomorphic differentials, see Subsection 2.7.  To this end we use the three following theorems and corollary proved in Section 3. To state them in more generality, we deal with a 
  dual billiard structure on an arc of a smooth conic $\gamma\subset\rp^2$, which yields a dual billiard map well-defined on 
  pairs $(Q,P)$, where $Q$ lies in appropriate domain adjacent to $\gamma$ from the concave side. 
  
  \begin{theorem} \label{parea} 
  For every $C^1$-smooth dual billiard structure on an arc of a smooth  conic $\gamma\subset\rp^2$ given by quadratic equation 
  $\gamma=\{ \mcp_2(z,w)=0\}$   the corresponding dual billiard map has a canonical invariant area form 
   \begin{equation} \omega_\gamma:=p_M^*\left((\mcp_2(z,w))^{-\frac32}dz\wedge dw\right).\label{area}\end{equation}
   The latter pull-back has two well-defined global meromorphic branches on $M$ with pole of order 3 along the curve $p_M^{-1}(\gamma)$ that differ by sign. 
   \end{theorem}
   \begin{remark} \label{remforms} 
   The invariant form statement of Theorem \ref{parea} was proved by S.Tabachnikov \cite[theorem E]{tabpr}. 
   In more detail, Tabachnikov proved its dual version for projective billiards on the unit circle. The general case then follows 
   by applying projective duality and projective transformation sending an arbitrary given smooth conic to the unit circle. 
   Namely, the double-valued form $(\mcp_2(z,w))^{-\frac32}dz\wedge dw$ is a canonical 2-form associated to a conic,  holomorphic and non-zero outside the conic and having a pole of order $\frac32$ along it. It is unique up to constant factor. 
   Therefore, a projective transformation sending a conic $\gamma_1$ to a conic $\gamma_2$ sends the canonical 2-form 
   for $\gamma_2$ to that for $\gamma_1$ up to constant factor. 
   For completeness of presentation we give a direct proof of Theorem \ref{parea} in its next version, based on the following 
   formula for the form $\omega_\gamma$ in the case of parabola.  
   \end{remark}
   \begin{theorem} \label{parea2} For  $\gamma=\{ w=z^2\}$ the  corresponding form $\omega_{\gamma}$ on $M$ can be normalized by multiplication by $\pm1$ so that 
   \begin{equation}\omega_{\gamma}=\omega_{par}:=(z(Q)-z(P))^{-3}p_M^*(dz\wedge dw).\label{omm}\end{equation}
   For every $C^1$-smooth dual billiard structure on an arc of  the parabola $\gamma$ 
   the corresponding dual billiard map $F$ preserves the area form $\omega_{par}$. The same holds for a germ of complex holomorphic dual billiard on $\gamma$ at any its point.
   \end{theorem}

\begin{corollary} \label{cinvfib} Consider a singular dual billiard structure on the parabola $\gamma$ with a rational integral $R$. 
Let us write the invariant area form (\ref{omm}) as 
\begin{equation}\omega_{par}=dR\wedge\omega_{fib},\label{oleaves}\end{equation}
where $\omega_{fib}$ is a meromorphic 1-form on $M$. For example, one can take 
\begin{equation}\omega_{fib}=\frac{dw}{(z^2-w)^{\frac32}\frac{\partial R}{\partial z}(z,w)}=\frac{dw}{(z-z(P))^3\frac{\partial R}{\partial z}(z,w)}, \ \ \ \ (z,w)=(z(Q),w(Q)).\label{oleav}\end{equation}

1) While the form $\omega_{fib}$ is multivalued, its restriction to each invariant fiber is a single-valued $F$-invariant meromorphic differential.

2) If a  fiber contains an elliptic irreducible component, then its restriction to this component is a global holomorphic differential.
\end{corollary}
\begin{theorem} \label{twoell} Let a singular holomorphic dual billiard structure on a complex conic have a rational integral $R$. Let a level curve $\Gamma_\la=\{ R=\la\}$ be elliptic. Then the corresponding invariant fiber $p_M^{-1}(\Gamma_\la)$ is a union of two 
elliptic curves, each bijectively projected onto $\Gamma_\la$. Thus, the $p_M$-pushforward of the 
 restriction of the above form $\omega_{fib}$ to each elliptic component of the fiber  is a well-defined global holomorphic differential on the elliptic curve $\Gamma_\la$. 
 \end{theorem}

Rational parametrizations of rational level curves $\Gamma_\la:=\{ R(z,w)=\la\}$ of most of integrals $R$ are constructed as follows. Let $\Sigma$ denote the indeterminacy point set of the integral $R$: each level curve $\Gamma_\la$ passes through all its points.  We construct an auxiliary pencil $C_t$, $t\in\oc$, of algebraic curves, conics in some cases and cubics in other cases. We construct it in such a way that each $C_t$ has a big intersection index with 
each $\Gamma_\la$ at the set $\Sigma$ so that there are only  either one, or two extra points of intersection $D_\la(t)\in\Gamma_\la\cap C_t$. In the case of one extra point $D_\la(t)$ 
the map 
$$\oc\to\Gamma_\la, \ \ \ t\mapsto D_\la(t)$$ 
is a rational parametrization of the curve $\Gamma_\la$ we are looking for. In the case of two extra intersection points we show that their projection  to the value $t$ yields a double covering $\Gamma_\la\to\oc_t$ with two known critical values. This implies that the covering curve $\Gamma_\la$ is rational.

Besides the type and parametrization of generic level curves of integrals, the corresponding fibers and dynamics along them, we also describe critical level curves and their critical points in the general sense of \cite{guss}. Namely,  not only critical points of integral in the usual sense, but also its indeterminacies, which are critical for appropriate level values $\la$ called critical. See the corresponding background material in Subsection 1.3. To this end, we use relation of 
topology of germ of analytic curve to the Newton diagram and quasihomogeneous Taylor polynomials corresponding to its edges: see a background material in Subsection 1.4. 

Real geometry and dynamics of the above-mentioned exotic dual billiards, and more on complex dynamics will be discussed in a next paper. 

\subsection{Previous results: rationally integrable dual billiards and their classification}

  \begin{theorem} \label{tgerm} \cite{grat} Let $\gamma\subset\rr^2\subset\rp^2$ be a $C^4$-smooth non-linear  germ of curve  equipped with a rationally integrable dual   billiard structure. 
 Then $\gamma$ is a conic. The dual   billiard extends to a singular holomorphic dual billiard structure on the ambient complex conic of one of 
 three following types up to real-projective equivalence:
 
  1) The  dual billiard  is of conical pencil type and has a quadratic integral.
    
  2) There exists an affine chart $\rr^2_{z,w}\subset\rp^2$ 
  in which $\gamma=\{ w=z^2\}$ and such that for every $P=(z_0,w_0)\in\gamma$ the 
  involution $\sigma_P:L_P\to L_P$ is given by one of the following formulas: 
  
  a) In the coordinate 
  $$\zeta:=\frac z{z_0}$$
  $$\sigma_P:\zeta\mapsto\eta_\rho(\zeta):=\frac{(\rho-1)\zeta-(\rho-2)}{\rho\zeta-(\rho-1)},$$
\begin{equation} \rho=2-\frac 2{2N+1}, \ \text{ or } \ \rho=2-\frac1{N+1} \  \text{ for some 
  } N\in\nn.\label{rhoval}\end{equation} 
  
  b) In the coordinate 
  $$u:=z-z_0$$
  \begin{equation} \sigma_P: u\mapsto-\frac u{1+f(z_0)u},\label{sigmaef}\end{equation}
   \begin{equation} f=f_{b1}(z):=\frac{5z-3}{2z(z-1)} \text{ (type 2b1))}, \  \text{ or } \ 
   f=f_{b2}(z):=\frac{3z}{z^2+1} \text{ (type 2b2))}.\label{sigma2}\end{equation}

  c) In the above coordinate $u$ the involution $\sigma_P$ takes the form (\ref{sigmaef}) with 
   \begin{equation} f=f_{c1}(z):=\frac{4z^2}{z^3-1} \text{ (type 2c1))}, \  \text{ or } \ 
   f=f_{c2}(z):=\frac{8z-4}{3z(z-1)} \text{ (type 2c2))}.\label{sigma3}\end{equation}
   
   d) In the above coordinate $u$ the involution $\sigma_P$ takes the form (\ref{sigmaef}) with 
   \begin{equation} f=f_d(z)=\frac4{3z}+\frac1{z-1}=\frac{7z-4}{3z(z-1)} \ \ \text{ (type 2d)}.\end{equation}
  \end{theorem}

  \medskip

  {\bf Addendum to Theorem \ref{tgerm} \cite{grat}.} {\it Every dual   billiard structure 
  on $\gamma$ of type 2a) has a rational first integral $R(z,w)$ of the form} 
  \begin{equation}R_{a1}(z,w)=\frac{(w-z^2)^{2N+1}}{\prod_{j=1}^N(w-c_jz^2)^2}, \ \ 
  c_j=-\frac{4j(2N+1-j)}{(2N+1-2j)^2}, \ \text{ for } \rho=2-\frac2{2N+1};\label{exot1}
  \end{equation}
  \begin{equation}R_{a2}(z,w)=\frac{(w-z^2)^{N+1}}{z\prod_{j=1}^N(w-c_jz^2)}, \ \ 
  c_j=-\frac{j(2N+2-j)}{(N+1-j)^2}, \ \text{ for } \rho=2-\frac1{N+1}.\label{exot2}\end{equation}
   {\it The dual billiards of types 2b1) and 2b2) have respectively the integrals}
  \begin{equation}R_{b1}(z,w)=\frac{(w-z^2)^2}{(w+3z^2)(z-1)(z-w)}, 
  \label{exo2bnew} \end{equation}
   \begin{equation}R_{b2}(z,w)=\frac{(w-z^2)^2}{(z^2+w^2+w+1)(z^2+1)}.
  \label{exo2bnew2} \end{equation}
  {\it The dual billiards of types 2c1), 2c2) have respectively the integrals}
  \begin{equation}R_{c1}(z,w)=\frac{(w-z^2)^3}{(1+w^3-2zw)^2},\label{exo2b}\end{equation} 
  \begin{equation}R_{c2}(z,w)=\frac{(w-z^2)^3}{(8z^3-8z^2w-8z^2-w^2-w+10zw)^2}.\label{exoc2}\end{equation}
{\it The dual billiard  of type 2d) has the integral}
 \begin{equation}R_{d}(z,w)=\frac{(w-z^2)^3}{(w+8z^2)(z-1)(w+8z^2+4w^2+5wz^2-14zw-4z^3)}.\label{exodd}\end{equation}

  \medskip

The definition of complex-projective equivalent singular complex dual   billiards 
repeats that of real-projective equivalent ones with change of real projective 
transformations $\rp^2\to\rp^2$ to complex ones acting on $\cp^2$.

 \begin{theorem} \label{tcompl} 
 Every regular  germ of holomorphic curve in $\cp^2$ (different from a straight line) 
 equipped with a 
 rationally integrable  complex dual   billiard structure is a conic. The dual billiard extends to a singular holomorphic dual billiard 
 with at most four singularities on the ambient complex conic. 
 Up to complex-projective equivalence, thus extended dual billiard has one of the types 1) (defined by a complex pencil, with a quadratic integral), 2a), 2b1), 2c1), 2d) listed in Theorem \ref{tgerm}, with 
 a rational integral as in its addendum. Here  $(z,w)$ are complex affine coordinates as in the addendum. 
  \end{theorem}
 {\bf Addendum to Theorem \ref{tcompl}} {\it The billiards of  types 2b1), 2b2), see (\ref{sigma2}), are 
 complex-projectively equivalent, and so are the billiards of types 2c1) and 2c2). There exist complex projective equivalences  between the billiards 2b1), 2b2) sending the integral $R_{b1}$ to $R_{b2}$, and between the billiards   2c1) and 2c2) sending the integral $R_{c1}$ to $-3R_{c2}$.  The singularities of the complex dual billiards 1), 2a)--2d) are the indeterminacies of the corresponding integrals, called the {\bf base points.} The sets of their base points are:
 
 Type 1): the set of base points of the  pencil: the common points of conics.
 
 Type 2a): the set $\Sigma_a:=\{(0,0), E\}$, where $E$ is the infinite point  of  $\gamma$.
 
 Types 2b1), 2c2), 2d): the set $\Sigma_{b1}=\Sigma_{c2}=\Sigma_d=\{(0,0), (1,1),E\}$.
 
 Type 2b2): the set $\Sigma_{b2}=\{(\pm i, -1), E\}$.  
 
 Type 2c1): the set $\Sigma_{c1}=\{(e^{\frac{2\pi i j}3}, e^{\frac{4\pi i j}3}) \ | \ j=0,1,2\}$. } 
 \begin{remark} In \cite{gl3} the author extended the above classification to the so-called dual multibilliards: 
 collections of curves and points equipped with dual billiard structures for appropriate notion of dual billiard structure at a point. 
 He classified rationally integrable dual multibilliards.
 \end{remark}
 
 \subsection{Background material: critical level curves and critical points of rational functions in general sense}
 
 Let $R$ be a meromorphic function on a smooth complex surface $S$. For every $\la\in\oc$ we set 
 $\Gamma_\la=\{ R=\la\}$. Let $A$ be its indeterminacy point and $\la\in\cc$. (The case, when $\la=\infty$ is treated analogously.) Fix a local chart $(z,w)$ centered at $A$. For a $\delta>0$ set  $B_{\delta}(A)$ to be the ball of radius $\delta$ centered at $A$ 
 in the chart.  Recall, see \cite[p. 282]{guss}, that the value $\la$ is called {\it typical} for the point $A$, if for every $\delta>0$ small enough there exists an $\var>0$ such that there is a homeomorphism of the set
 $\cup_{|\mu-\la|<\var}(\Gamma_\mu\cap B_{\delta}(A))$ and the product 
 $\{|\mu-\la|<\var\}\times(\Gamma_\la\cap B_{\delta}(A))$ that commutes with the projection $(z,w)\mapsto \mu=R(z,w)$. 
 In the opposite case the value $\la$ is called {\it atypical,} or {\it critical} for the  point $A$, and the indeterminacy point $A$ is said to be a {\it critical indeterminacy point} for the value $\la$. This motivates the following definition.
 
   \begin{definition} \label{defcr} A value $\la\in\oc$ is called {\it critical} for the function $R$, if at least one of the two following conditions holds:

1) The curve $\Gamma_\la$ contains a critical point of the function $R$ that is different from its indeterminacy points. Then it is a critical point of the map $R$ and is called a {\it (true) critical point,} and $\la$ is called a {\it true critical value.}

2) There exists a critical indeterminacy point for the  value $\la$. 

A true critical point is {\it Morse,} if it is isolated and the corresponding Hessian form of the function $R$ is non-degenerate. Or equivalently, if it is a double point: a transversal self-intersection of the corresponding level curve.
\end{definition}

\begin{remark} It is well-known that in the case, when the ambient surface $S$ is compact, the numbers of indeterminacies and critical values is finite, see \cite{guss} and references therein.
\end{remark}

Consider a critical level curve $\Gamma_\la$. Its {\it critical intersection graph} is the following  graph. Its vertices are its irreducible components. We connect two components by an edge, if they contain the same true critical point, and we mark this edge by 
the triple: the critical point and the pair of components. If a component $v$ contains a true critical point $A$ that is not contained in another component, then we connect $v$ to itself by an edge marked by $A$. 

  \begin{proposition} \label{procritic} Let $R$ be a meromorphic function on a smooth compact complex surface $S$. Let for an infinite set of values $\la$ the  level curve $\Gamma_\la=\{ R=\la\}$ be rational. Then the following statements hold. 
 
 1) Each non-critical level curve is rational.
 
 2) Each critical level curve $\Gamma_\la$ containing a true critical point is either a union of at least two rational curves, or is one rational curve entirely consisting of 
 critical points. In the latter case the function $R-\la$ has a multiple zero along the curve $\Gamma_\la$. 
 
 3) Each isolated true critical point belongs to at least two distinct components of the corresponding critical level curve.
 
 4) The critical intersection graph of  each critical level curve is either a tree, or a disjoint union of trees. 
 
 5) If a critical level curve $\Gamma_\la$ consists of $k\geq2$ components, then it contains at most $k-1$ true critical points. 
 \end{proposition}
 The statements of Proposition \ref{procritic} are well-known and follow from results of paper \cite{guss} and  the material presented in  \cite{eisenbud} and in  \cite[chapter III, section 9]{hartsh}. For completeness of presentation we present its proof below. By $D_\sigma(c)$ 
 we denote the disk in $\cc$ of radius $\sigma$ centered at $c$. 
 
\begin{proof} In the case when $S$ is compact, the number of critical values $\la$ is finite and all the regular level curves $\Gamma_\la$ are homeomorphic. They are all rational, since there is an infinite number of values $\la$ for which $\Gamma_\la$ is rational. In the proof of the other statements we use the following well-known fact given by Picard--Lefschetz Theory, see \cite[chapter IV]{avg}. For every germ of holomorphic function $R$ at its isolated true critical point $A$ 
for every $\delta>0$ small enough there exists an $\var>0$ such that  the intersection $\Gamma_\la\cap B_\delta(A)$ is a finite union of curves homeomorphic to disks, pairwise intersected only at $A$, and the intersections $\Gamma_{\mu,\delta}:=\Gamma_\mu\cap B_\delta(A)$ with $\mu\in D_\var(\la)\setminus\{\la\}$ are all homeomorphic to a connected Riemann surface with a finite number of boundary components: their number is equal to the number of components of the intersection $\Gamma_\mu\cap\partial B_{\delta}(A)$. The rank of the homology group 
$H_1(\Gamma_{\mu,\delta},\zz)$  is positive and equal to the multiplicity of the critical point $A$.  

The above statements together imply that in our situation of a rational function $R$ on a compact complex surface $S$ 
 for every $\mu\in D_\var(\la)\setminus\{\la\}$ the Riemann surface parametrizing 
the curve $\Gamma_\mu$ is obtained from that parametrizing the critical curve $\Gamma_\la$ by pasting at least one handle, which corresponds to a non-trivial element in the above homology group. 
If $A$ is a unique critical point lying in $\Gamma_\la$, then the number of handles is equal to its $\delta$-invariant, see its definition in \cite[section I.3]{GLS} and below in Subsection 3.1. 

Let us prove Statement 2). Let the critical point in question be isolated. Then the curve $\Gamma_\la$ cannot be irreducible, i.e., parametrized by a connected Riemann surface. Indeed, otherwise for $\mu$ close to $\la$ the curve $\Gamma_\mu$ would be parametrized by a Riemann surface of positive genus by the above handle pasting statement, and hence, would not be rational, -- a contradiction. Each component of the curve $\Gamma_\la$ is rational by the same argument. 
If the critical point is not isolated and $\Gamma_\la$ is irreducible, then it is clearly a multiple rational curve entirely consisting of critical points. Statement 2) is proved.

Suppose the contrary to Statement 3): there exists a critical level curve 
$\Gamma_\la$,  and an isolated true critical point $A\in\Gamma_\la$ that lies only in one its irreducible component, denoted $\alpha$. 
Let $\alpha_0$  denote the complement of the component $\alpha$ to the union of balls of radius $\delta$ centered at the critical and indeterminacy points of the function $R$. It is 
a connected Riemann surface. The above argument   
implies that if $\var>0$ is small enough, then for every $\mu\in D_\var(\la)$ the curve 
$\Gamma_\mu$ contains an open subset homeomorphic to a connected Riemann surface $\alpha_0$ with at least one pasted handle. 
This implies that the Riemann surface parametrizing $\Gamma_\mu$ cannot be homeomorphic to the Riemann sphere, and hence, is not rational. Statement 3) is proved.

The proof of Statement 4) is analogous: presence of a cycle in the critical intersection graph again implies that the 
above Riemann surface has positive genus, which leads to contradiction as above. Indeed, in this case the  surface parametrizing $\Gamma_\mu$, $\mu\in D_\var(\la)\setminus\{\la\}$, is obtained by subsequent gluing a cyclical chain 
of $k$ compact orientable surfaces with some topological discs deleted, due to the above argument; $k\in\nn$. The gluing is done along the boundaries of deleted disks. A surface obtained by gluing  two compact surfaces  in this way has genus no less than the genera of the glued surfaces. The equality to both genera takes place only if they are topological spheres, glued along one topological circle. But the last, $k$-th surface is glued to the result of gluing the first $k-1$ surfaces along boundaries of at least two disjoint topological disks, due to cyclicity. Therefore, this yields a surface of positive genus, due to the above inequality. 

Another way to prove Statement 4) would be to consider a resolution of indeterminacies giving a new surface $\wt S$ with a bimeromorphic projection $\pi:\wt S\to S$ 
so that the function $R\circ\pi$ is a holomorphic map $\wt S\to\oc$. It is a rational fibration in the sense that its generic fiber is a 
rational curve. Then to use  a well-known fact stating that the intersection graph of each critical fiber of a rational fibration is 
a tree. In our case each critical fiber is a union of strict transforms of irreducible components of the same fiber before resolution 
and some exceptional curves obtained by blow up. Since the intersection graph of all of them, including the exceptional ones, 
is a tree, erasing the exceptional curves yields either a tree, or a disjoint union of trees.

Statement 5) follows from Statement 4). Proposition \ref{procritic} is proved.
\end{proof}

 \subsection{Background material: germs of analytic curves, Newton diagrams and quasihomogeneous  polynomials} 
 
 Here we deal with a germ at $(0,0)$ of analytic  function $f(z,w)$ in two variables, $f(0,0)=0$. To each  
 monomial $z^mw^n$ entering its Taylor series we put into correspondence 
the quadrant $K_{m,n}:=(m,n)+(\rr_{\geq0})^2$. Let $K(f)$ denote the convex hull of the union of 
the quadrants $K_{m,n}$ through all the Taylor monomials of the function $f$; it is an unbounded polygon 
with a finite number of sides. Recall that the {\it Newton diagram} 
$\bold{N}_f$ is the union of those edges of  the boundary $\partial K(f)$ that do not lie in the coordinate axes: they are called 
the {\it edges of the Newton diagram.} 
 
 Fix a coprime pair of numbers $p,q\in\nn$, $(p,q)=1$. For every monomial $z^kw^m$  
 define its {\it $(p,q)$-quasihomogeneous degree:} 
$$\deg_{p,q}z^kw^m:=kq+mp.$$
 Recall that a polynomial $P(z,w)$ is {\it $(p,q)$-quasihomogeneous,} 
if it contains only monomials $z^kw^m$ of the same $(p,q)$-quasihomogeneous degree, i.e., with $(k,m)$ lying on the same line 
parallel to the segment $[(p,0),(0,q)]$. It is well-known that each $(p,q)$-quasihomogeneous polynomial is a product of a constant and powers of prime factors of the type $z$, $w$, $w^q-cz^p$,  $c\in\cc\setminus\{0\}$, see, e.g., \cite{bermed}, \cite[section 3]{grat}. 

Consider an  edge $\mathcal E$ of the Newton diagram that is not parallel to a coordinate axis, if any.  Then it is parallel to a segment connecting some points $(p,0)$ and $(0,q)$, $p,q\in\nn$, $(p,q)=1$.  To the edge $\mathcal E$ we associate the sum 
$\wt f_{p,q}(z,w)$ of those Taylor monomials of the function $f$ whose bidegrees are contained in $\mathcal E$. This is exactly the {\it lower $(p,q)$-quasihomogeneous part} of the function $f$, i.e., the sum of those Taylor monomials that have the minimal $(p,q)$-quasihomogeneous degree. 

Let us recall the following well-known facts, see, e.g., \cite[section 3]{grat}: 

(i) Each edge $\mathcal E$ of the Newton diagram corresponds to at least one irreducible factor  of the germ $f(z,w)$ in the  ring of germs of analytic functions, with Newton diagram consisting of just one edge, parallel to $\mathcal E$.

(ii) If the Newton diagram consists of just one edge parallel to a segment $[(p,0),(0,q)]$, $p,q\in\nn$, $(p,q)=1$, and the 
lower $(p,q)$-quasihomogeneous part $\wt f_{p,q}(z,w)$ is, up to constant factor,  a product $\prod_{j=1}^k(w^q-c_jz^p)^{r_j}$, $c_j$ being distinct,  then the germ $f(z,w)$ is, up to constant factor,  a product of  $k$  factors $f_j$ with 
$\wt f_{j,(p,q)}=(w^q-c_jz^p)^{r_j}$. 
\begin{proposition} \label{proii} Let $R$ be a non-constant meromorphic function on a connected compact complex smooth analytic surface $S$. 
Let $A\in S$ be its inteterminacy point. Let $R=\frac fg$ be an irreducible fraction presentation of its germ at $A$: 
$f$ and $g$ are coprime germs of holomorphic functions at $A$. 
  Let $(z,w)$ be a local holomorphic chart on $S$ centered at $A$. Consider the Newton diagram $\mcn(f,g)$ of the vector function $(f,g)$,  
  defined in the same way, as above, by its  Taylor monomials having non-zero vector coefficients.   Let $\mcn(f,g)$ consist of just one edge $[(q,0), (0,p)]$, $q,p\in\nn$.  Set $\Phi_{p,q,\la}(z,w)$ to be the lower $(p,q)$-quasihomogeneous part of  the function $f-\la g$ if $\la\in\cc$, and that of the function $g$ if $\la=\infty$. 
  
  1) If for some $\la\in\oc$ the polynomial $\Phi_{p,q,\la}$ is a product of distinct prime factors, then this holds for all but a finite number of values $\la$, and 
  $f-\la g$ (respectively, $g$, if $\la=\infty$) is a product of functions whose lower $(p,q$)-quasihomogeneous parts are the latter 
  prime factors, see (ii). The indeterminacy $A$ is non-critical for the values $\la\in\oc$ satisfying the above statement.
    
  2)  If the point $A$ is a critical indeterminacy point for a value $\la=\mu$, then the quasihomogenous 
  polynomial $\Phi_{p,q,\mu}$ has a multiple  prime factor.
    
  3) If $p=q$ and the homogeneous polynomial $\Phi_{p,q,\mu}$ has no multiple prime factor, i.e., no multiple zero line, 
  then the germ at $A$ of the level curve $\Gamma_\mu=\{ R=\mu\}$ is a union of $p$ pairwise transversal regular germs of 
  analytic curves tangent to the zero lines of the polynomial $\Phi_{p,p,\mu}$. 
  \end{proposition}
  Propisition \ref{proii} follows from Statement (ii).

\section{Main results} 

\subsection{Birationality of billiard map and a priori possible types of generic fibers}

\begin{definition} A singular holomorphic dual billiard structure on a regular algebraic curve $\gamma\subset\cp^2$ is called 
{\it meromorphic,} if it is meromorphic at each its singularity $O$.  This means that in an affine chart $(z,w)$ centered at  $O$ 
such that the $w$-axis is transversal to $\gamma$ the involutions $\sigma_P:L_P\to L_P$ written in the coordinate 
$u=z-z(P)$ on $L_P$ have the form 
\begin{equation}\sigma_P(u)=-\frac{u}{1+\psi(z(P))u},\label{psimer}\end{equation}
$$\psi(z) \text{ is holomorphic on a punctured neighborhood of zero with a pole at } 0.$$
\end{definition}
\begin{example} All the rationally integrable dual billiards given by Theorem \ref{tgerm} are meromorphic, and the corresponding poles are simple.
\end{example}

\begin{proposition} \label{prbir} For every meromorphic singular dual billiard on a regular complex conic the corresponding 
dual billiard map $F:M\to M$ given by (\ref{bmap}) is birational.
\end{proposition}
\begin{proof} 
The maps $F^{\pm1}$ are clearly well-defined on a Zariski open and dense subset of  pairs $(Q,P)\in M$ with $Q$ not lying on a line tangent to $\gamma$ at a singularity of the dual billiard. Let us represent the conic 
$\gamma$ as a parabola $\{ w=z^2\}$ in an affine chart. Then the above function $\psi(z)$ is rational, being meromorphic, 
see (\ref{psimer}). Therefore, the map $F$ is expressed via rational functions, and hence, is rational. This together with 
its invertibility on a Zariski open and dense subset implies its birationality.
\end{proof} 

Consider a rationally integrable singular holomorphic dual billiard from Theorem \ref{tgerm}. Let $R$ be its rational first integral from the Addendum to Theorem \ref{tgerm}.  For every $\la\in\oc$ we set 
$$\Gamma_\la:=\{ R(z,w)=\la\}\subset\cp^2, \ \ \ S_\la:=p_M^{-1}(\Gamma_\la)\subset M.$$
Recall that the singularities of the dual billiard are exactly the indeterminacies of the function $R$. Their number is at most four, and all of them lie in $\gamma$, see the addendum. Let us denote their complement in $\gamma$ by $\gamma^o$:
\begin{equation}\gamma^o:=\gamma\setminus\{\text{indeterminacies of } R\}.\label{gammao}\end{equation}

 \begin{lemma} \label{irred} \cite[lemma 2.8]{gl3} The curve $\Gamma_\la$ is irreducible for all but a finite number of values $\la$.
\end{lemma}
We identify each curve $\Gamma_\la=\{ R=\la\}\subset\cp^2$ with its {\it normalization:} the compact Riemann surface parametrizing $\Gamma_\la$. 
This identification is given by natural projection, which is bijective except for self-intersections. 
The projection $p_M:S_\la\to\Gamma_\la$ of fibers $S_\la$ to level curves $\Gamma_\la$ is a degree two branched covering of compact Riemann surfaces parametrizing $S_\la$ and $\Gamma_\la$. The latter branched covering will be also denoted by $p_M$. 

Everywhere below whenever we say {\it "for generic $\la\in\oc$",} we mean {\it "for all but a finite number of $\la\in\oc$".} 

Recall that a germ of analytic curve $\Gamma$ in $\cc^2$ at a point $A$ is  {\it irreducible,} if it is not a union of two other germs. It is well-known that each germs of analytic curve is a finite union of its irreducible components, parametrized curves  called {\it local branches.} Each  $\Gamma$ is defined by a germ of analytic function $f(z,w)$ vanishing on $\Gamma$ that is minimal, i.e. not a product of a germ vanishing on $\Gamma$ and another germ vanishing at $A$.
Each local branch $\Gamma_j$ corresponds to an irreducible factor of the germ $f$ vanishing on $\Gamma_j$. It  also corresponds  to a point 
$A_j$ of the normalization of the curve $\Gamma$ such that the parametrization by normalization sends a neighborhood 
of the point $A_j$ to $\Gamma_j$. 
  
\begin{proposition} \label{ctype} In every rationally integrable singular holomorphic dual billiard from Theorem \ref{tgerm} the curves $\Gamma_\la$, $S_\la$ satisfy the next statements. 

1)   $\Gamma_\la$ is  either rational for 
generic $\la$, or elliptic for generic $\la$. 

 2) For every $\la\in\oc$ the  preimage $S_\la=p_M^{-1}(\Gamma_\la)\subset M$ is $F$-invariant. 
 It is either a rational, or an elliptic curve, or a union of two rational or two elliptic curves. One of these types arises for  generic $\la\in\oc$.
 
 3) If $\Gamma_\la$ is irreducible and the double covering $p_M:S_\la\to\Gamma_\la$ has at least one branching point, then 
 $\Gamma_\la$ is rational and $S_\la$ is irreducible.
 
 4) Let $\la\neq0$, $\Gamma_\la$ be irreducible, and let $A$ be a point of its normalization. Let $b$ be the corresponding local branch of the 
 curve $\Gamma_\la$ at $A$.  If $A\in\Gamma_\la\cap\gamma$ and $b$ is regular and transversal to $\gamma$, then $A$  is a branching point 
 of the covering $S_\la\to\Gamma_\la$. Conversely, each branching point of the latter covering is a point $A\in\Gamma_\la\cap\gamma$ for which $b$ is either regular and transversal to $\gamma$, or regular and tangent to 
 $\gamma$ with at least cubic contact, or singular.
  
 5) If $\Gamma_\la$ is irreducible and its normalization contains at least three branching points of the above  covering, then $S_\la$ is elliptic, and $\Gamma_\la$ is rational. 
 \end{proposition}
 
 \begin{proposition} \label{elltr} Let $S_\la$ be either an elliptic curve for a generic $\la$, or a union of two elliptic curves for 
 a generic $\la$. Then each its elliptic component is $F$-invariant and $F$ acts there as a translation of complex torus. 
 \end{proposition}
 
 Propositions \ref{ctype} and \ref{elltr} are proved in Subsection 3.2.

  In the next subsections we state our main results on types of curves $\Gamma_{\la}$ and $S_{\la}$  for each 
  exotic rationally integrable dual billiard.
  
 We use the homogeneous coordinates $[z:w:t]$ on $\cp^2$ such that the 
  standard affine chart with the coordinates $(z,w)$ is the chart  $\{ t=1\}$. Thus, the infinite point $E$ of the parabola $\gamma$ is the point $[0:1:0]$.  The projective level curves $\Gamma_\la$ of  integrals and the fibers $S_\la=p_M^{-1}(\Gamma_\la)$ will be denoted by 
  $$\Gamma_{\psi,\la}:=\{ R_{\psi}=\la\}\subset\cp^2 \ \ \text{ and } \ S_{\psi,\la}, \ \ \ \  \psi\in\{ a1, a2, b1,\dots,d\}.$$ 
  The corresponding dual billiard map $F:M\to M$ will be denoted by $F_\psi$.
  
  \subsection{Type of fibers and dynamics in Case a1)}
  
  Consider an arbitrary rational function of the next type, e.g., $R_{a1}$ in (\ref{exot1}):  
  \begin{equation} R(z,w)=\frac{(w-z^2)^{2N+1}}{\prod_{j=1}^N(w-c_jz^2)^2}, \ \ \text{ here } \ c_j\in\cc, \ c_j\neq1 \ \text{ are arbitrary.}\label{rgen}
  \end{equation} 
 \begin{theorem} \label{taga1} 1) For every $\la\neq0,\infty$ the curve $\Gamma_{\la}=\{ R=\la\}$ 
 is rational, the projection 
  $$\Gamma_{\la}\to\oc, \ \ (z,w)\mapsto\frac{z^2}{z^2-w}$$
  is a double covering ramified over $0$ and $\infty$, and the curve  $\Gamma_{\la}$ admits a 
  bijective holomorphic parametrization by the parameter 
  \begin{equation}\tau=\frac z{\sqrt{z^2-w}}\in\oc:\label{tauzw}\end{equation}
$$z(\tau)=i\sqrt\la\tau\prod_{j=1}^N((1-c_j)\tau^2-1),$$
  \begin{equation} 
  w(\tau)=\left(1-\frac1{\tau^2}\right)z^2(\tau)=-\la(\tau^2-1)\prod_{j=1}^N((1-c_j)\tau^2-1)^2.\label{param1}
  \end{equation}
  2) The set of critical points of the function $R$ is the union $\Gamma_0\cup\Gamma_\infty$ of its multiple level curves. In particular, there are no isolated true critical points. Each of the two indeterminacies $(0,0)$, $E$ is critical for the values 
  $\la=0,\infty$.
  \end{theorem}
  \begin{theorem} \label{tas1}
  1) Each fiber $S_{\la}=p_M^{-1}(\Gamma_{\la})\subset M$ with $\la\neq0,\infty$ is a union of two  rational curves $S_{\la\pm}$.
  
  2) Each curve $S_{\la\pm}$ is  parametrized 
  by lifted parametrization (\ref{param1}) so that to each $Q(\tau)=(z(\tau),w(\tau))\in\Gamma_{\la}$ corresponds 
  $$(Q(\tau),P_{\pm}(\tau))\in S_{\la\pm}, \ P_{\pm}=(z_{0\pm},z_{0\pm}^2),$$ 
  \begin{equation} z_{0\pm}(\tau)=\frac{\tau\pm 1}{\tau}z(\tau)=i\sqrt\la(\tau\pm 1)\prod_{j=1}^N((1-c_j)\tau^2-1).\label{parap1s}
  \end{equation}
    \end{theorem}
   
    \begin{theorem} \label{taf1} Let now $R=R_{a1}$, i.e., the $c_j$ in (\ref{rgen}) be  as in (\ref{exot1}). Then 
    
 1) each curve $S_{\la\pm}$ is invariant under the dual billiard map $F=F_{a1}$; 
 
  2) in the  above coordinate $\tau$ on $S_{\la\pm}$, $\la\neq0,\infty$, the map $F:S_{\la\pm}\to S_{\la\pm}$ 
  acts as 
  \begin{equation} F: \tau\mapsto \tau\pm\frac{2}{2N+1}.\label{fa1s}\end{equation} 
  \end{theorem}

  \subsection{Type of fibers and dynamics in Case a2)}
  
  Consider an arbitrary rational function of the next type, e.g., $R_{a2}$ in (\ref{exot2}):
   \begin{equation} R(z,w)=\frac{(w-z^2)^{N+1}}{z\prod_{j=1}^N(w-c_jz^2)}, \ \ \text{ here } \ c_j\in\cc, \ c_j\neq1 \ \text{ are arbitrary.}\label{rgen2}
  \end{equation} 
  
 \begin{theorem} \label{taga2} Let $R$ be as in (\ref{rgen2}). 
 
 A. Then for every $\la\neq0,\infty$ 
 
 1) the curve $\Gamma_{\la}=\{ R=\la\}$ is rational;
 
 2) the restriction to $\Gamma_{\la}$  of the projection 
 \begin{equation}(z,w)\mapsto s:=\frac{z^2}{z^2-w}\label{zws}\end{equation} 
 is a bijection onto the Riemann sphere $\oc_s$, 
  and its inverse is the following parametrization of the curve $\Gamma_{\la}$:
    \begin{equation}z(s)=-\la s\prod_{j=1}^N(1+(c_j-1)s), \ \ w(s)=\la^2s(s-1)\prod_{j=1}^N(1+(c_j-1)s)^2.\label{param2}\end{equation}
    
    B. Let $c_j$ be distinct. Then the  critical point set of the function $R$ is the conic $\Gamma_0=\{ w=z^2\}$. 
    Its indeterminacies $(0,0)$ and $E$ are critical points for the values $\la=0$ and $\la=0,\infty$ respectively. 
    \end{theorem}
\begin{theorem} \label{tas2}  For every $\la\neq0,\infty$ 

1) the curve $S_{\la}=p_M^{-1}(\Gamma_\la)$ is rational; 

2) the projection 
 $p_M:S_{\la}\to\Gamma_{\la}$ has two ramification points in $\Gamma_\la$ that coincide with 
 the origin and infinity in the parabola $\gamma$: they correspond to the parameter values $s=0,\infty$.
 
 3) the curve $S_\la$ is bijectively holomorphically parametrized by 
 $$\tau=\sqrt s: \ \ \ S_\la=\{(Q(\tau),P(\tau))\in M \ | \ \tau\in\oc\}, \ Q(\tau)=(z(\tau^2),w(\tau^2)),$$
 \begin{equation} z(\tau^2)=-\la \tau^2\prod_{j=1}^N(1+(c_j-1)\tau^2), \ \ w(\tau^2)=\la^2\tau^2(\tau^2-1)\prod_{j=1}^N(1+(c_j-1)\tau^2)^2,\label{zwa2}\end{equation}
 $$P(\tau)=(z_0(\tau),z_0^2(\tau)), \ \ z_0(\tau)=\frac{\tau+1}{\tau}z(\tau^2).$$
 \end{theorem}
 
 \begin{theorem} \label{taf2} Let now $R=R_{a2}$, i.e., the $c_j$ be as in (\ref{exot2}). 
 Then 
 
 1) each curve $S_{\la}$ is invariant under the dual billiard map $F=F_{a2}$; 
 
 2) in the  above coordinate $\tau$ on $S_{\la}$ the map $F:S_{\la}\to S_{\la}$ 
  acts as 
  \begin{equation} F: \tau\mapsto \tau+\frac{1}{N+1};\label{fa2s}\end{equation}
  \end{theorem}

   \subsection{Type of fibers and dynamics in Cases b1), b2)}
   
      The dual billiard of type b2) is obtained from that of type b1) by a projective transformation $\Psi$ so that the integral 
   $R_{b2}$ is the $\Psi$-pushforward of the integral $R_{b1}$, see Addendum to Theorem \ref{tcompl}. In the homogeneous coordinates $[z:w:t]$ 
   the matrix of the transformation $\Psi$ is 
$$M=\left(\begin{matrix} 0 & \frac i2 & -\frac i2\\
1 & \frac12 & \frac12\\ 1 & -\frac12 & -\frac12\end{matrix}\right),$$
 see \cite[formula (8.1)]{grat}.   Thus, in the affine chart $(z,w)$  \  $\Psi$ takes the form 
 \begin{equation}\Psi(z,w)=\left(\frac{i(w-1)}{2z-w-1}, \frac{2z+w+1}{2z-w-1}\right), \ \ R_{b1}(z,w)=R_{b2}\circ\Psi(z,w).\label{psipr}\end{equation}
    \begin{theorem} \label{tbg1} 1) For  $\la\in\cc^*\setminus\{0,1,\frac43,\infty\}$ the level curve $\gbll=\{ R_{b1}=\la\}$ is rational of degree four and parametrized by
    \begin{equation} z=\frac{t((1-\la)t+\la)}{\la(t-1)(4-t)}, \ \ \  \ w=\frac{t^2((1-\la)t+\la)(3\la-t)}{\la^2(t-1)(4-t)^2}.\label{wb1}\end{equation}
Its only singular points are the indeterminacies $(0,0)$, $(1,1)$, $E$ of the function $R_{b1}$, which are double points of transversal self-intersection. 

2) The critical values of  $R_{b1}$, see Definition \ref{defcr}, are $0$,  $1$, $\frac43$, $\infty$.

3) The indeterminacies $(0,0)$, $E$ of the function $R_{b1}$ are critical for the value $\la=0$, and $(1,1)$ is critical for the values $\la=0,1$. 

4) Its critical points  lying outside its double zero  $\gamma$ are Morse. They are:  

4.1) the points $(1,-3)$, $(-\frac13, -\frac13)$ corresponding to the value $R_{b1}=\infty$;

4.2) the point $(0,-1)$ corresponding to the value $R_{b1}=1$;

4.3) the point $(\frac13,1)$ corresponding to the value $R_{b1}=\frac43$.

5) The critical fiber $\Gamma_{b1,1}$ is the union 
$\Gamma_{b1,1}=\{ z=0\}\cup \Gamma_{b1,1,*}$, 
$$\Gamma_{b1,1,*}:=\{2z^3-3z^2w-3z^2+6wz-w^2-w=0\}$$
\begin{equation}=\{(\frac t{(t-1)(4-t)}, \frac{t^2(3-t)}{(t-1)(4-t)^2}) \ | \ t\in\oc\}.
\label{gb1*}\end{equation}
The intersection $\{ z=0\}\cap \Gamma_{b1,1,*}$ is transverse, it is  $\{(0,0),  (0,-1), E\}$.  
The only singular point of the curve $\Gamma_{b1,1,*}$ is the point  $(1,1)$, which is a cusp. 

6) The critical fiber $\Gamma_{b1,\frac43}$ is a union of conics:  $\Gamma_{b1,\frac43}=\Gamma_{b1,\frac43,1}\cup\gblctd$,
 \begin{equation}\gblctl:=\{ w=\frac{3z^2}{4z-1}\}, \ \ \gblctd:=\{ w=z(4-3z)\}.\label{gb43}\end{equation}
 These conics are intersected transversally at  $(0,0)$, $(1,1)$, $E$ and  $(\frac13,1)$.
 \end{theorem}  
 
  \begin{theorem} \label{tbs1} 1) For every  $\la\neq0,1,\frac43,\infty$ the  projection $p_M:S_{b1,\la}\to\Gamma_{b1,\la}$ has four ramification points at the parameter values 
  \begin{equation}t=\frac{\la}{\la-1}, \ \infty  \ \text{ and roots of the polynomial } \ t^2-4\la t+4\la.\label{brb}\end{equation}
  The curve $S_{b1,\la}$ is elliptic and conformally equivalent to the cubic
  \begin{equation}y^2=((1-\la)t+\la)(t^2-4\la t+4\la)\label{eltb}\end{equation}
  which has holomorphic differential
  \begin{equation}\omega_b:=\frac{dt}{\sqrt{((1-\la)t+\la)(t^2-4\la t+4\la)}}.\label{diffb}\end{equation}
  
  2) The curve $S_{b1,1}$ consists of three rational  components $S_{b1,10}$, $S_{b1,2\pm}$:  
   the component $S_{b1,10}$ is  a double covering over the axis $\{ z=0\}$ branched over $(0,0)$ and $E$; 
   each  component $S_{b1,2\pm}$  is bijectively projected onto $\Gamma_{b1,1,*}$.

  3) The curve $S_{b1,\frac43}$ 
  corresponding to $\la=\frac43$ is singular and consists of two rational irreducible components, each being a double covering over a conical component of the curve $\Gamma_{b1,\frac43}$. 
    
  4) For $\la\neq0,1,\frac43,\infty$ the map $F:S_{b1,\la}\to S_{b1,\la}$ is a translation of the complex torus parametrizing $S_{b1,\la}$: the quotient of the line $\cc$ by the period lattice of the differential (\ref{eltb}). 
    \end{theorem}
   \begin{theorem} \label{tbg2} The statements of Theorems \ref{tbg1}, \ref{tbs1} hold for the dual billiard of type b2) with the following 
 modifications: 

- the parametrizations  (\ref{wb1}), (\ref{gb1*})  of level curves are replaced by their postcompositions with the transformation $\Psi$ given by (\ref{psipr}); 

- the above  points and curves in the ambient plane $\cp^2$ of the dual billiard are replaced by their $\Psi$-images.

For every $\la\neq0,1,\frac43,\infty$ the elliptic curve $S_{b2,\la}$ has the same conformal type and holomorphic differential, as above.
\end{theorem}
Theorem \ref{tbg2} follows from Theorems \ref{tbg1}, \ref{tbs1} and the discussion before Theorem \ref{tbg1}.

  \subsection{Type of fibers and dynamics  in Cases c1), c2)}
  Recall that the dual billiard of type 2c1) is obtained from that of type 2c2) by complex projective transformation 
  $M_c: \cp^2_{[z:w:t]}\to\cp^2_{[z:w:t]}$ with the matrix 
  $${\LARGE M_c=\left(\begin{matrix}-\frac12 & \frac12 & \frac12\\ 1 & \frac{\bar\var}2 & \frac\var2\\ 1 & \frac\var2 & 
 \frac{\bar\var}2\end{matrix}\right),} \ \ \ \var=e^{-\frac{2\pi i}3},$$
  which sends $R_{c1}$ to $-3R_{c2}$,   see  \cite[formula (8.3)]{grat}:  
 \begin{equation}-3R_{c2}=R_{c1}\circ M_c,\label{rc1c2m}\end{equation}
  \begin{theorem} \label{tc}  1) The critical values of the function $R_{c1}$ in the sense of Definition \ref{defcr} are 
 $0,\infty, \frac{27}{64}$.
 
 2) For every regular value $\la$ the curve $\Gamma_{c1,\la}=\{ R_{c1}(z,w)=\la\}$ 
 is elliptic. 
 
 3) The curve $\Gamma_{c1,\infty}=\{ R_{c1}=\infty\}=\{1+w^3-2zw=0\}$ is a rational cubic, regular everywhere except for 
 a unique double point, namely,  its only infinite point $V:=[1:0:0]\in\cp^2_{[z:w:t]}$. 
 
 4) The level curve $\Gamma_{c1,\frac{27}{64}}$  is the union of the three following conics, obtained one from the other 
 by the order 3 cyclic group action generated by the transformation $(z,w)\mapsto (\var z,\var^2 w)$, $\var=e^{-\frac{2\pi i}3}$:  
 $$\mcc_1:=\{ 4z^2+3(w^2+1)+6z(w+1)+5w=0\},$$
 $$\mcc_\var:=\{ 4\bar\var z^2+3(\var w^2+1)+6z(w+\var)+5\bar\var w=0\},$$
 \begin{equation}\mcc_{\bar\var}:=\{ 4\var z^2+3(\bar\var w^2+1)+6z(w+\bar\var)+5\var w=0\}.\label{3con}\end{equation}
 The conics $\mcc_1$, $\mcc_{\bar\var}$ are tangent to each other at the point $(\bar\var,\var)$ with cubic contact and intersect transversally at the point $(\frac{\bar\var}2,\var)$. The other tangencies and transversal intersections of conics are obtained from the latter ones by the above group action. The transversal intersections are the only critical points of the function $R_{c1}$  in the  curve $\Gamma_{c1,\frac{27}{64}}$ in the sense of Definition \ref{defcr}.  
 
 5) The indeterminacies $(1,1)$, $(\var,\bar\var)$, $(\bar\var,\var)$ of the integral are critical only for the critical values $\la=0,\infty$. For the other values $\la$ the germ at each of them of the corresponding level curve is a union of two regular germs of parametrized curves tangent with cubic contact.
 \end{theorem}
 
 \begin{theorem} \label{tc'} 1) The critical values of the function $R_{c2}$ in the sense of Definition \ref{defcr} are 
 $0,\infty, -\frac{9}{64}$.
 
 2) The Statements 2)--5) of Theorem \ref{tc} hold for the function $R_{c2}$, with order three symmetry group generated by the projective transformation with the matrix 
 \begin{equation}Sym_{c2}=\left(\begin{matrix}-1 & 1 & 0\\ -2 & 1 & 1\\ 0 & 1 & 0\end{matrix}\right),\label{symc2}
 \end{equation} 
 and the points and the conics in question being the 
 $M_c$-preimages of those corresponding to the function $R_{c1}$.  Namely, 
 
 - the double point of the cubic $\Gamma_{c2,\infty}$ is $(\frac12,1)$; 
 
 - the conics forming the critical level curve $\Gamma_{c2,-\frac9{64}}$ are 
 \begin{equation}\{ w+8z^2=0\}, \ \ \{ 9(w-z)^2+w-z^2=0\}, \ \ \{ 9(z-1)^2+w-z^2=0\};\label{conic2}\end{equation} 
  
  - their cubic tangency points are  $(0,0)$, $(1,1)$ and $E=[0:1:0]$;  
 
 - their transversal intersection points are $(\frac54,1)$, $(-\frac14,-\frac12)$, $(\frac12,-2)$;
 
 - the indeterminacies of the function $R_{c2}$ are $(0,0)$, $(1,1)$, $E$; they are critical points only for the values $\la=0,\infty$. 
 \end{theorem}

 \begin{theorem} \label{tc2} 1) For every $\la\neq0,\infty, \frac{27}{64}$ 
 the curve $S_{c1,\la}$ consists of two elliptic  components, each bijectively projected onto $\Gamma_{c1,\la}$ by 
$p_M$. Each component is $F$-invariant, and $F$ acts there as a complex torus translation. 
 
 2) The curve $S_{c1,\infty}$ consists of two rational components, each bijectively covering the curve 
 $\Gamma_{c1,\infty}$.
 
 3) The curve $S_{c1,\frac{27}{64}}$ consists of six rational component, each being bijectively projected onto 
 a conic from (\ref{3con}): two components above each conic. 
 
 4) The above statements also hold for the curves $S_{c2,\la}$ corresponding to the integral $R_{c2}$, where in Statements 1), 3) the value $\frac{27}{64}$ and conics 
 (\ref{3con}) are replaced respectively by $-\frac9{64}$ and conics (\ref{conic2}). 
  \end{theorem}
  
 \subsection{Type of fibers and dynamics in Case d)}
 
 \begin{theorem} \label{tdg} 1) For every $\la\neq0,-\frac13, -\frac9{32},\infty$ the level curve $\Gamma_{d,\la}=\{ R_d=\la\}$ is 
 a rational sextic  parametrized as $\Gamma_{d,\la}(t)=(z(t),w(t))$, 
      \begin{equation}z(t)=-\frac{(\la t-1)(\la t^2+12\la t-9)}{\la t(3+t)((8\la+1)t-5)},\label{forztd}\end{equation}
      \begin{equation}w(t)=-\frac{(\la t^2+12\la t-9)^2(\la t^2+4\la t-1)}{\la^2t^2(3+t)^2((8\la+1)t-5)(t+4)}.\label{wformd}\end{equation}
      
      2) For $\la=-\frac13$ the curve $\Gamma_{d,\la}$ is a union of two rational cubics: 
     \begin{equation}\Gamma_{d,-\frac13}=\Gamma_{d,-\frac13,1}\cup\Gamma_{d,-\frac13,2}, \ \ \Gamma_{d,-\frac13,1}=\{ w=-\frac{z^2(5z-8)}{4z-1}\},\label{gd131}
      \end{equation}
      \begin{equation} \Gamma_{d,-\frac13,2}=\{ \mcp=0\}, \ \mcp(z,w):=w+8z^2-7z^3+8z^2w+w^2-11zw=0, \label{gd213}\end{equation}
      $$\Gamma_{d,-\frac13,2} \ \ 
 \text{ is parametrized as } \ \Gamma_{d,-\frac13,2}(\tau)=(z_2(\tau),w_2(\tau)),$$
      \begin{equation} z_2(\tau)=\frac{(\tau+9)(2\tau+27)}{\tau(5\tau+72)}, \ \ w_2(\tau)=\frac{(2\tau+27)^2(2\tau-9)}{\tau^2(5\tau+72)}.\label{gd132}\end{equation}
      
     3) For $\la=-\frac9{32}$ the curve $\Gamma_{d,\la}$ is a union of rational quartic and conic:
 $$ \Gamma_{d,-\frac9{32}}=\Gamma_{d,-\frac9{32},1}\cup\Gamma_{d,-\frac9{32},2},$$
   $$\Gamma_{d,-\frac9{32},1}=\Gamma_{d,-\frac9{32},1}(t)=(z_3(t),w_3(t)), \ \ z_3(t)=-\frac{(9t+32)(t+8)}{40t(3+t)},$$
      \begin{equation} w_3(t)=-\frac{(t+8)^2(3t+8)(3t+4)}{40t^2(t+3)^2}=-z_3^2(t)\frac{40(3t+8)(3t+4)}{(9t+32)^2},
      \label{gd91}\end{equation}
      \begin{equation}\Gamma_{d,-\frac9{32},2}=\{ w+8z^2-9zw=0\}.\label{starcon}\end{equation}
      4) One has $\Gamma_{d,\infty}=C\cup\{ z=1\}\cup\mathcal S$,  $S$ is a rational cubic: 
      \begin{equation} C=\{ w+8z^2=0\}, \ \mathcal S:=\{ w+8z^2+4w^2+5wz^2-14zw-4z^3=0\}.\label{cands}\end{equation}
      \end{theorem}
      \begin{theorem} \label{tdcr} 1) The critical values of the function $R_d$ in the general sense of Definition \ref{defcr} 
      are $\la=0,-\frac13, -\frac14, -\frac9{32},\infty$. There are no true critical points with value $\la=-\frac14$. 
      
       \noindent 2) The only true critical point  with  value $\la=-\frac13$ is  $(\frac25, \frac85)\in\Gamma_{d,-\frac13,1}\cap\Gamma_{d,-\frac13,2}$. 
      
     \noindent 3) The only true critical point  with  $\la=-\frac9{32}$ is  
      $(\frac1{10}, -\frac45)\in\Gamma_{d,-\frac9{32},1}\cap\Gamma_{d,-\frac9{32},2}$. 
      
      \noindent 4) The only true critical points  with  $\la=\infty$ are:
      
    $$(1,-8)\in C\cap\{ z=1\}, \ \  \ \ (-\frac12,-2)\in C\cap\mathcal S;$$
    the line $\{ z=1\}$ and  $\mcs$ intersect only at the indeterminacies $(1,1)$ and $E$. 
       
      \noindent      5) The true critical points are Morse, i.e., transversal intersections. 
    
     \noindent  6) The indeterminacies $(0,0)$, $(1,1)$, $E$ are critical points for the values $\la=0, -\frac14$, $\la=0, -\frac9{32}, \infty$ and 
      $\la=0$ respectively.  
      \end{theorem} 
      \begin{theorem} \label{locind} 1) For every $\la\neq0,-\frac13, -\frac9{32},\infty$ the germs of the curve $\Gamma_{d,\la}$ at the indeterminacies $E$, $(1,1)$ have  regular local branches, namely, 
      
           at $E$: one  branch transversal to the infinity line and two  branches tangent respectively to the conics $C$ and  
      $\{ w=-\frac54z^2\}$ with at least cubic contact;
      
       at $(1,1)$: three pairwise transversal local branches.
       
       2) For $\la\neq0, -\frac14$ its germ at $(0,0)$ consists of two regular local branches tangent to $C$ with at least cubic contact and having cubic contact 
 between themselves. For $\la=-\frac14$ its germ at $(0,0)$ is singular and irreducible: a degenerate cusp. 
       
       3) The quartic $\Gamma_{d,-\frac9{32},1}$ and the cubic $\mcs$ both have  cusps at $(1,1)$. The  tangent lines to $\Gamma_{d,-\frac9{32},1}$ and $\Gamma_{d,-\frac9{32},2}$ at $(1,1)$ are transversal. The tangent lines to 
    $\Gamma_{d,-\frac9{32},1}$,  and to $\mcs$ at $(1,1)$ are transversal to the line $\{ z=1\}$ and to the conic $\gamma$. 
      \end{theorem}
 \begin{theorem} \label{td} 1) For every $\la\neq0, -\frac13, -\frac9{32}, \infty$ the projection $p_M:S_{d,\la}\to\Gamma_{d,\la}$ 
 has four ramification points corresponding to the parameter values 
 \begin{equation} t=-4, \ \text{  the roots of the polynomial } \ (9\la^2+\la)t^3+(36\la^2-3\la)t^2-36\la t+9.\label{brd}
 \end{equation}
The curve $S_{d,\la}$ is elliptic and conformally equivalent to the curve 
\begin{equation}y^2=(t+4)((9\la^2+\la)t^3+(36\la^2-3\la)t^2-36\la t+9).\label{ctyped}\end{equation}
with holomorphic differential
\begin{equation}\omega_d=\frac{dt}{\sqrt{(t+4)((9\la^2+\la)t^3+(36\la^2-3\la)t^2-36\la t+9)}}.\label{diffd}\end{equation}
 The map $F:S_{d,\la}\to S_{d,\la}$ is a translation of the complex torus parametrizing $S_{d,\la}$: the quotient of the line $\cc$ 
 by the period lattice of the latter differential. 
 
 2) For $\la=\infty$ the curve $S_{d,\infty}$ consists of 5 rational components: 
 
 - two components, each bijectively projected by $p_M$ onto the parabola $C$;
 
 - one rational component that is a double covering over the line $\{ z=1\}$;
 
 - two components, each bijectively projected onto the rational cubic $\mcs$.
 
 3) For $\la=-\frac13$ the curve $S_{d,-\frac13}$ consists of two rational components, which are double coverings 
 over the cubics $\Gamma_{d,-\frac13,1}$ and $\Gamma_{d,-\frac13,2}$. 
 
 4) For $\la=-\frac9{32}$ the curve $S_{d,-\frac9{32}}$ consists of three rational components: 
 
 - two components, each bijectively projected onto   the quartic $\Gamma_{d,-\frac9{32},1}$; 
 
 - one component being a double covering over the conic $\Gamma_{d,-\frac9{32},2}$.
 \end{theorem}

\subsection{Invariant holomorphic 1-forms on elliptic level curves}
 
 \begin{theorem} \label{telld}  1) For the dual billiards of types 2b), 2c), 2d) the $p_M$-pullbacks to $M$ of the differential 1-forms 
 \begin{equation}\omega_{\alpha}=g_{\alpha}(z,w)dw, \ \ g_{\alpha}=(z^2-w)^{-\frac32}\left(\frac{\partial R_{\alpha}}{\partial z}(z,w)\right)^{-1}, 
  \ \ \alpha=b1,\dots,d,\label{oma}\end{equation}
are holomorphic differentials on those fibers $S_\la$ that are either elliptic in Cases b), d), or  unions of two elliptic curves in Case c). The above differentials on (each component of) the above $S_\la$ are equal to  (\ref{diffb}) in Case b) and (\ref{diffd}) in Case d) up to constant factor depending on $\la$. In Case 2c)  forms (\ref{oma}) are holomorphic differentials on non-critical  level curves $\Gamma_\la$.

2) The above functions $R_{\alpha}$ are as in the Addendum to Theorem \ref{tgerm}:
\begin{equation*}R_{b1}(z,w)=\frac{(w-z^2)^2}{(w+3z^2)(z-1)(z-w)}, 
 \end{equation*}
   \begin{equation*}R_{b2}(z,w)=\frac{(w-z^2)^2}{(z^2+w^2+w+1)(z^2+1)}, \end{equation*}
  \begin{equation*}R_{c1}(z,w)=\frac{(w-z^2)^3}{(1+w^3-2zw)^2},\end{equation*} 
  \begin{equation*}R_{c2}(z,w)=\frac{(w-z^2)^3}{(8z^3-8z^2w-8z^2-w^2-w+10zw)^2},\end{equation*}
 \begin{equation*}R_{d}(z,w)=\frac{(w-z^2)^3}{(w+8z^2)(z-1)(w+8z^2+4w^2+5wz^2-14zw-4z^3)}.\end{equation*}
 \end{theorem}
 Theorem \ref{telld} follows from Corollary \ref{cinvfib}.

 \section{Proofs of main results}
 
 \subsection{Background material: $\delta$-invariant of plane curves}
 In the proof of some of main results  we  use the following background material contained in \cite{GLS}. Let $S$ be a compact two-dimensional complex  manifold. Let $\Gamma\subset S$ be a compact holomorphic curve, i.e., a purely one-dimensional complex analytic subset. Let us recall the definition of $\delta$-invariant $\delta(Q)$ of a point $Q\in\Gamma$, see, e.g.,  \cite[section I.3]{GLS}.  Fix a local biholomorphic chart $(z,w)$ on $S$ centered at $Q$. The intersection of the curve $\Gamma$ with a small ball $B$ centered at $Q$ is a union of topological disks identified with 
 the local branches of the germ at $Q$ of the curve $\Gamma$. Each two of them intersect each other only at $Q$. Let $f(z,w)$ be a holomorphic function on $B$ defining $\Gamma$:  $\Gamma\cap B=\{ f=0\}$ and 
$f$ has no  curves of multiple zeros. The intersection $\Gamma\cap\partial B$ is a finite disjoint union of topological circles: the intersections of  $\partial B$ with the  local branches of the germ of the curve $\Gamma$ at $Q$. 
Their number, denoted by $k\in\nn$, is equal to the number of local branches. 
  For every $\var\neq0$ small enough the level curve $\gamma_\var:=\{ f=\var\}\cap B$ is regular and 
intersects the boundary $\partial B$ at $k$ disjoint topological circles close to the above ones. Let us take a sphere with 
$k$ disks deleted and paste it to $\gamma_\var$. 

\begin{definition} \cite[section I.3]{GLS}. The genus of the compact surface thus obtained is called the {\it $\delta$-invariant} of the singular point $Q$ of the curve $\Gamma$.
\end{definition}
\begin{remark} \label{rdelta} The $\delta$-invariant of a singular point  is known to be always positive. The $\delta$-invariants of a double point of transversal self-intersection of regular branches and of  a cusp  are equal to one. Conversely, if the $\delta$-invariant is equal to one, then the singular point is either a double point, or a cusp. The $\delta$-invariant is superadditive, or in other terms, upper semicontinuous, in the following sense:  if a singular germ of analytic curve at a point $p$  is analytically deformed to a curve with several singularities close to $p$, then the sum of their $\delta$-invariants is no greater than $\delta(p)$. 
The equality takes place under the condition that the deformation is  {\it equinormalizable}: this means that the curves of the family can be analytically simultaneously parametrized, bijectively up to self-intersections, by a family of Riemann surfaces of constant topological type \cite[subsection 3.3]{teiss}.\end{remark}
Let $\Gamma$ be an irreducible algebraic curve  in 
$\cp^2$, $d$ be its degree, $g$ be the genus of its normalization, i.e., parametrizing compact Riemann surface. Let  $Sing(\Gamma)$ be the set of its singular points.  Then 
\begin{equation} g=\frac{(d-1)(d-2)}2-\sum_{p\in Sing(\Gamma)}\delta(p), \ \ \ \ \ \delta(p) \text{ is the } \delta-\text{invariant},
\label{hiron}\end{equation} 
by Hironaka Genus Formula \cite{hir}.

 \subsection{A priori possible types of curves and dynamics. Proof of Propositions \ref{ctype}, \ref{elltr}} 
 
 The proof of Proposition \ref{ctype} is based on the following lemma.

\begin{lemma} \label{palt} 1) Each curve $S_\la$ is $F$-invariant. 

2) If $S_\la$ consists of two components for a generic $\la$, then each its component is $F$-invariant for a generic $\la$. 

3)  For every $N\in\nn$  there exists a  $\delta>0$ such that for every $\la\neq0$ with  $|\la|<\delta$  the restriction of the iterate $F^k$ to each invariant component  of the fiber $S_\la$ is not the identity for all $k=1,\dots,N$. 
\end{lemma}

\begin{proof} Invariance of the curve $S_\la$ follows from invariance of the integral. Let us prove Statement 3). The foliation $R=const$ is regular outside 
indeterminacies of the function $R$ and its critical points lying outside its multiple zero curve $\gamma=\{ R=0\}$. 
In particular  $\gamma^o$, see (\ref{gammao}),  is its regular leaf. Fix an arbitrary point $P_0\in\gamma^o\cap\cc^2$ and affine coordinates 
$(x,y)$ centered at $P_0$ so that the $x$-axis be tangent to $\gamma$ at $P_0$. Fix an arbitrary small $\var>0$. 
Let $D_\var$ denote the $\var$-disk in $L_{P_0}$ centered at $P_0$ in the coordinate $x$.  
 Fix a $\delta>0$ small enough so that for every  $\la\in\cc$, $|\la|<\delta$, 
the level curve $\Gamma_\la$ is irreducible and intersects $D_\var$. Take an intersection point 
$Q_0\in\Gamma_\la\cap D_\var$. Set 
$$(Q_1,P_1)=F(Q_0,P_0), \ \ (Q_2,P_2)=F(Q_1,P_1), \ \dots$$
The differences $x(Q_1)-x(P_0)$ and $x(P_0)-x(Q_0)$ are nonzero and asymptotically equivalent, as $\var\to0$, since 
$Q_1=\sigma_{P_0}(Q_0)$ and $\sigma_{P_0}$ is a projective involution of the line $L_{P_0}$ fixing $P_0$. 
The differences $x(P_1)-x(Q_1)$ and $x(Q_1)-x(P_0)$ are also asymptotically equivalent. Indeed, the two lines 
through $Q_1$ tangent to $\gamma$ touch it at $P_0$ and $P_1$. Hence, as $\var\to0$, the point 
$Q_1$ tends to $P_0$, and then the ratio of the above differences tends to one, see \cite[proposition 2.10]{alg}. Finally, 
\begin{equation} x(P_1)-x(Q_1)\simeq x(P_0)-x(Q_0)=o(1),\label{aseq01}\end{equation}
$x(Q_1)-x(Q_0)\simeq 2(x(P_0)-x(Q_0))$. Similarly we get 
$$x(Q_2)-x(Q_1)\simeq 2(x(P_1)-x(Q_1))\simeq 2(x(P_0)-x(Q_0))\simeq x(Q_1)-x(Q_0).$$  
Thus, the $x$-coordinates of the points $Q_j$ form an asymptotic arithmetic progression, as $\var\to0$. Hence, $N+1$ successive points $Q_0,\dots, Q_N$ are distinct, whenever $\delta$ is small enough. 

Let now $S_\la$ consist of two components $S_{\la\pm}$ for a generic $\la$. Then each of them is bijectively projected onto 
$\Gamma_\la$. It suffices to prove Statement 2), i.e., invariance of each 
component,  for an open subset of values $\la\in\oc$. We will prove it for $\la\neq0$, $|\la|<\delta$, where $\delta>0$ is small enough. In the above arguments let $\Gamma_\la$ be irreducible. We claim that the above pairs $(Q_0,P_0)$ and  $(Q_1,P_1)$ lie in the same component of the fiber $S_\la$. Indeed,  consider the path $\beta$ in $\Gamma_\la$ going from $P_0$ to $P_1$ 
whose projection to the $x$-axis is the straightline segment $[x(P_0),x(P_1)]=[0,x(P_1)]$. Let us show that it lifts 
to a path in $S_\la$ connecting $(Q_0,P_0)$ and $(Q_1,P_1)$. Let the integral $R$ be normalized to vanish on $\gamma$, as in 
the Addendum to Theorem \ref{tgerm}, by taking its post-composition with appropriate M\"obius transformation. Let $k$ denote its order of zero along the curve $\gamma$. Fix a holomorphic branch of the function $\wt R=R^{\frac1k}$ on a neighborhood 
of the point $P_0$ in $\cc^2$. It is locally constant on level curves of the function $R$. Set $\wt\la:=\wt R(Q_0)$. For every 
$P\in\gamma$ close enough to $P_0$ the restriction $\wt R|_{L_P}$ written in the coordinate $\wt x:=x-x(P)$ on $L_P$ 
has the asymptotics $\xi(P)\wt x^2(1+o(1))$, as $\wt x\to0$; here $\xi(P)$ is a non-zero analytic function. Thus, if $\wt\la$ 
is small enough, then the equation $\wt R|_{L_P}=\wt\la$ has two asymptotically opposite solutions $\wt x_P=\sqrt{\frac{\wt\la}{\xi(P)}}(1+o(1))$ analytic in $P$ on a neighborhood of the point $P_0$. Fix an analytic solution $\wt x_P$, $P\in\beta$, 
such that $\wt x_{P_0}=\wt x(Q_0)$. It corresponds to an analytic family of intersection points $Q(P)\in L_P\cap\Gamma_\la$ with 
$Q(P_0)=Q_0$. The ratios of the values $\wt x_P$ of the same holomorphic branch are close to one, since $\xi(P)\to\xi(P_0)\neq0$, as $P\to P_0$. One has $Q(P_1)=Q_1$, by (\ref{aseq01}) and since $\wt x_{P_0}\simeq\wt x_{P_1}$, as $\wt\la\to0$, 
by the above asymptotics. The path 
$(Q(P),P)_{P\in\beta}\subset S_\la$ connecting $(Q_0,P_0)$ to $(Q_1,P_1)$ is constructed. 
This  proves the lemma.
\end{proof} 

\begin{proof} {\bf of Proposition \ref{ctype}.} Its Statement 2)  implies Statement 1). Indeed, for every branched covering 
of compact surfaces the genus of covering surface is no less than that of the base. This follows by Riemann--Hurwitz Theorem. 
Therefore if we prove  that some component of the curve $S_\la$ is either elliptic, or rational, this will imply that  $\Gamma_\la$ 
has one of these types as well. Note that  if $S_\la$ is elliptic, then a priori the curve 
$\Gamma_\la$ may be rational. 

Recall that each $S_\la$ is $F$-invariant. Let us suppose that $S_\la$ is irreducible for a generic $\la$. 
Let $g$ denote its genus. It cannot be greater than one. Indeed, the cardinality of 
the group 
of conformal automorphisms of a compact Riemann surface of a genus $g\geq2$ is bounded by a constant depending only on $g$. On the other hand,  for small $\la$ 
the cyclic group of conformal automorphisms $S_\la\to S_\la$ generated by $F$ may contain arbitrarily big number of elements, 
by Lemma \ref{palt}, Statement 3). Hence, $g\in\{0,1\}$.  The curve $\Gamma_\la$ being irreducible for a generic $\la$, 
see Lemma \ref{irred}, its double covering curve 
$S_\la$ consists of at most two components. If it consists of two components, each of them is bijectively projected onto 
 $\Gamma_\la$. Therefore, the conformal type of components is the same.  Each component being 
$F$-invariant for generic $\la$ (Lemma \ref{palt}), this implies that it is either rational, or  elliptic, as in the above argument. One of these two cases takes place for generic $\la$. This proves Statement 2) of Proposition \ref{ctype}.

Let us prove Statement 3). For every double covering over a connected Riemann surface that has at least one branching point 
the covering surface is obviously connected. 
If $\Gamma_\la$ is irreducible, $S_\la$ is elliptic and the covering $p_M:S_\la\to\Gamma_\la$ has at least one branching point, 
then $\Gamma_\la$ is rational, by Riemann--Hurwitz Theorem. This proves Statement 3). 

Let us prove Statement 4). Let us first prove that each branching point $A$ of the covering $p_M:S_\la\to\Gamma_\la$ 
lies in the intersection $\Gamma_\la\cap\gamma$.  
Given any point $A\in\cp^2\setminus\gamma$, there are two tangent lines to $\gamma$ through $A$. The corresponding tangency points depend locally holomorphically on $A$. This follows from quadraticity of  tangency order by the Implicit Function Theorem. This implies that there are no branching points in the complement $\Gamma_\la\setminus\gamma$. 
Let now $A$ be a point of the normalization of the curve $\Gamma_\la$ that lies in the intersection $\Gamma_\la\cap\gamma$. 
Let $b$ denote the corresponding local branch of the curve $\Gamma_\la$,  locally biholomorphically parametrized by a parameter $\tau$ as $\tau\mapsto b(\tau)$, $b(0)=A$. Consider the double-valued function $b\to\gamma$ sending 
$\tau$ to a point $P(\tau)\in\gamma$ such that the line $P(\tau)b(\tau)$ is tangent to $\gamma$ at $P(\tau)$. 
One has $p_M(b(\tau),P(\tau))=b(\tau)$. If $b$ is regular and transversal to $\gamma$, then $\dist(b(\tau),\gamma)\simeq c_1|\tau|$, and then $\dist(P(\tau),A)\simeq c_2\sqrt{|\tau|}$, as 
$\tau\to0$; here $c_1,c_2\neq0$. The latter asymptotics follows from \cite[proposition 3.8]{gs}. Therefore, $A$ is a branching point.
 
Conversely, let a point $A$ of the normalization of the curve $\Gamma_\la$ be a branching point; thus, $A\in\Gamma_\la\cap\gamma$. Suppose the contrary to the second part of of Statement 4): the corresponding 
  local branch $b$ is regular and quadratically tangent to $\gamma$. Then in the above notations the function $P(\tau)$ is a 
  multivalued holomorphic function on $b\setminus\{ A\}$ that has two holomorphic branches with asymptotics
  $P_j(\tau)=(c_j+o(1))\tau$, as $\tau\to0$; here $c_1\neq c_2$. This follows from quadraticity of tangence and 
\cite[proposition 2.50]{alg}. Therefore, $P_{1,2}(\tau)$ are not analytic extensions one of the other along a circuit around the origin $A$. Hence, they are holomorphic single-valued and $A$ is not a branching point. Statement 4) of Proposition \ref{ctype} is proved. 

Every double covering $\oc\to\oc$ has two branching points, by Riemann--Hurwitz Formula. This together with Statement 3) implies that if the curve $\Gamma_\la$ is irreducible and the covering $S_\la\to\Gamma_\la$ has at least three branching points, 
then $\Gamma_\la$ is rational, $S_\la$ is irreducible and cannot be rational. Hence, in this case it is elliptic. Proposition \ref{ctype} is proved. 
\end{proof}

\begin{proof} {\bf of Proposition \ref{elltr}.} Let now $S_\la$ be either elliptic, or a union of two elliptic curves for generic $\la$. 
Then each its component is $F$-invariant for a generic $\la$, and the restriction $F|_{S_\la}$ has 
non-trivial iterates $F^k|_{S_\la}$, $k=1,\dots,N$, for arbitrarily big $N$, as $\la$ is small enough, by Lemma \ref{palt}. A conformal automorphism of complex torus that is not a 
translation, if it exists, it is an automorphism of order at most six. Therefore, for  $\la\neq0$ small enough, and hence, for generic $\la$ the restriction of the map $F$ to each component of the curve $S_\la$ is a translation. Proposition \ref{elltr} is proved.
\end{proof}

   \subsection{Case a1). Proofs of Theorems \ref{taga1}--\ref{taf1}}

\begin{proof} {\bf of Theorem \ref{taga1}.} 
 Consider the family of parabolas  
 $$X_u:=\{ w=uz^2\}, \ \ u\in\cc^*,$$ 
  equipped with the coordinate $z$. For every $u$ the restriction $R|_{X_u}$ is the quadratic monomial 
 \begin{equation}R_{u}(z):=\frac{(u-1)^{2N+1}}{\prod_{j=1}^N(u-c_j)^2}z^2.\label{ra1}\end{equation}
 It takes a given value $\la\neq0,\infty$ at  two points, which we will denote by $z_{\pm}(u)$. Therefore, the projection 
 $$\psi:\Gamma_\la\to\oc_u, \ \ (z,w)\mapsto u=\frac w{z^2}$$ 
 is a double covering.  
It is branched exactly over $0$ and  $\infty$. Indeed, for a given $\la\neq0,\infty$ 
one has $z_{+}(u)=z_{-}(u)$, if and only if
 the coefficient at $z^2$ in (\ref{ra1}) is equal to either zero, or infinity. The latter coefficient, as a function of $u$, 
 has zero of odd degree $2N+1$ at $u=1$, pole of odd degree one at $u=\infty$ 
 and  poles of even degree 2 at $c_j$, or poles of even degrees bigger than two at some $c_j$, if some $c_j$ coincide. The latter poles of even degree do not contribute to branching, since extending 
 the root functions $z_{\pm}(u)$ along a circuit around an even degree pole does not permute them. On the other hand, extension along a circuit around a zero or pole of odd degree permutes $z_{\pm}(u)$. 
 Therefore, $\Gamma_\la$ is 
 the Riemann sphere bijectively parametrized by the parameter $\sqrt{1-u}=\frac1{\tau}$.  Parametrization (\ref{param1}) is found immediately by solving equation $R_{u}(z)=\la$ in $z$. 
 
 Let us show that $R_u$ has no isolated true critical points. Indeed, a priori such points may exist only for values 
 $\la\neq0,\infty$, since the curves $\Gamma_0$, $\Gamma_{\infty}$  entirely consist of critical points. 
 The curve $\Gamma_\la$  with $\la\neq0,\infty$ contains no true critical points, by Proposition \ref{procritic}, Statement 3) and since it is not a multiple level curve, being a parametrized curve of degree $4N+2=\deg R$. The statement saying that 
 the indeterminacies $(0,0)$ and $E$ are critical exactly for the values $\la=0,\infty$ follows directly from the 
 parametrization (\ref{param1}) of each curve $\Gamma_\la$, $\la\neq0,\infty$. Indeed, its germ at infinity is one germ of parametrized curve taken once,  
  the germ of its parametrization at $\tau=\infty$, while this is not the case for $\la=0,\infty$. In the case, when $c_j$ are distinct, its germ at $(0,0)$ consists of 
 $2N$ regular germs depending locally analytically on $\la$, while this is obviously not the case for $\la=0,\infty$. The latter germs 
  are given by germs of parametrization (\ref{param1}) of the curve $\Gamma_\la$ at roots of the polynomials $(1-c_j)\tau^2-1$. The case, when some of $c_j$ coincide, is treated analogously. 
 Theorem \ref{taga1} is proved.
 \end{proof}
 
 \begin{proof} {\bf of Theorem \ref{tas1}.} Fix a $Q=(z,w)=(z(\tau),w(\tau))\in\Gamma_\la$. A point $P=(z_0,z_0^2)\in\gamma$ 
 such that $Q\in L_P$ is found from the equation
\begin{equation}w=2zz_0-z_0^2: \ \ \ z_{0\pm}= z\pm\sqrt{z^2-w}=\frac{\tau\pm 1}{\tau}z(\tau),\label{zpm0}\end{equation}
see (\ref{tauzw}). This implies the statements of Theorem \ref{tas1}.
\end{proof}

\begin{proof} {\bf of Theorem \ref{taf1}.} Let now $R=R_{a1}$. Then each fiber $S_{\la}$ is $F$-invariant, since $R$ is 
a rational integral of the dual billiard. The fact that for every $\la\neq0$ small enough 
each its component $S_{\la\pm}$ is $F$-invariant follows from formula 
(\ref{parap1s}) and  formula (\ref{aseq01}) from the proof of Lemma \ref{palt}. This implies invariance of $S_{\la\pm}$ for all 
$\la\in\cc^*$, by uniqueness of analytic extension. Let us now express the restriction $F:S_{\la\pm}\to S_{\la\pm}$ in the coordinate $\tau$. In our case $\rho=2-\frac2{2N+1}$. In the coordinate $\zeta=\frac z{z_0}$ on $L_P$ one has 
\begin{equation}\sigma_P(\zeta)=\eta_{\rho}(\zeta)=\frac{(\rho-1)\zeta-(\rho-2)}{\rho\zeta-(\rho-1)}=\frac{(2N-1)\zeta+2}{4N\zeta-(2N-1)}.\label{sigma1}\end{equation}
Set 
$$(z_{*\pm}(\tau),w_{*\pm}(\tau)):=F(z(\tau),w(\tau)), \ \ \zeta_{\pm}(\tau):=\frac{z(\tau)}{z_{0\pm}(\tau)}=\frac{\tau}{\tau\pm 1},$$
$$\zeta_{*\pm}(\tau):=\frac{z_{*\pm}(\tau)}{z_{0\pm}(\tau)}; \ \ \ \zeta_{*\pm}(\tau)=\eta_{\rho}(\zeta_{\pm}(\tau))=
\frac{(2N+1)\tau\pm2}{(2N+1)\tau\mp(2N-1)},$$
by (\ref{zpm0}) and (\ref{sigma1}). Let us now find the point $(z_{*\pm},w_{*\pm})$ and its $\tau$-coordinate on the curve 
$S_\la$, which will be denoted by $\tau_{*\pm}$. One has 
$$z_{*\pm}(\tau)=z_{0\pm}(\tau)\zeta_{*\pm}(\tau)=z_{0\pm}(\tau)\frac{(2N+1)\tau\pm2}{(2N+1)\tau\mp(2N-1)},$$
$$w_{*\pm}=2z_{*\pm}z_{0\pm}-z_{0\pm}^2,  \ \ \ w_{*\pm}-z_{*\pm}^2=-z_{0\pm}^2(\zeta_{*\pm}-1)^2,$$
$$\zeta_{*\pm}(\tau)-1=\pm\frac{(2N+1)}{(2N+1)\tau\mp(2N-1)},$$
$$\tau_{*\pm}=\frac{z_{*\pm}}{\sqrt{z_{*\pm}^2-w_{*\pm}}}=(\pm) \frac{\zeta_{*\pm}}{(\zeta_{*\pm}-1)},$$
$$\tau_{*\pm}(\tau)=(\pm) \frac{(2N+1)\tau\pm2}{(2N+1)}.$$
Here the sign $(\pm)$ a priori may be different from the other sign $\pm$. In fact $(\pm)=+$. Indeed, otherwise 
the map $F:S_{\la\pm}\to S_{\la\pm}$, $\tau\mapsto\tau_{*\pm}$  would be an involution. This is impossible, since 
for small $\la\neq0$ this is not an involution, see Lemma \ref{palt}. Hence, $\tau_{*\pm}(\tau)=\tau\pm\frac{2}{2N+1}$. 
 This proves Theorem \ref{taf1}.
\end{proof}

 \subsection{Case a2). Proofs of Theorems \ref{taga2}--\ref{taf2}}
  
\begin{proof} {\bf of Theorem \ref{taga2}.} 
The restriction of the function $R(z,w)$ to the parabola  
$$C_s=\{\frac{z^2}{z^2-w}=s\}=\{ w-z^2=-s^{-1}z^2\}$$
is equal to 
  \begin{equation}\mcr_{s}(z):=-s^{-1}\prod_{j=1}^N(1+(c_j-1)s)^{-1}z.\label{ra1}\end{equation}
  Therefore, the equation $\mcr_s(z)=\la$ has unique solution $z=z(s)$ given by (\ref{param2}), and then 
  $w=\frac{s-1}sz^2$ is also given by (\ref{param2}). Bijectivity of parametrization (\ref{param2}) follows from (\ref{zws}). 
  Statement A is proved. 
  
  Let us prove Statement B. Let $c_j$ be distinct. 
  The fact that the critical point set coincides with the conic $\Gamma_0$ is proved analogously to the proof of 
  Statement 2) of Theorem \ref{taga1}; now the curve $\Gamma_\infty$ is a union of $N$ distinct parabolas and one line, each 
  with multiplicity one. The indeterminacy $(0,0)$ is 
  critical only for the value $\la=0$, since for all $\la\neq0$, including 
  $\la=\infty$, the germ at $(0,0)$ of the curve $\Gamma_\la$ consists of $N+1$ regular local branches;  
  $N$ of them are quadratically tangent to each other and transversal to the remaining $N+1$-st branch. This follows immediately from (\ref{param2}), since $c_j$ are distinct. 
  The infinite point $E$ is critical for the values $\la=0,\infty$, as in the previous subsection. Indeed, for 
  $\la\neq0,\infty$ the germ at $E$ of the parametrized curve $\Gamma_\la$ is an irreducible singular germ 
  taken with multiplicity one: it corresponds to the unique 
  parameter value $s=\infty$. And this is obviously not the case for $\la=0,\infty$. Theorem \ref{taga2} is proved.
  \end{proof}
   
  \begin{proof} {\bf of Theorem \ref{tas2}.} Fix a point $(z(s),w(s))\in\Gamma_\la$. A 
  point $P=(z_0,z_0^2)\in\gamma$ such that  
  $(z(s),w(s))\in L_P$ is found from the quadratic equation 
  $$2z(s)z_0-z_0^2=w(s)=z^2(s)(1-s^{-1}).$$
  It has solutions $z_0(\tau)=\frac{\tau+1}{\tau}z(\tau^2)$ with  two different root values $\tau=\sqrt s$.
\end{proof}
  
  \begin{proof} {\bf of Theorem \ref{taf2}.} Each curve $S_\la$ is $F$-invariant, since $R_{a2}$ is an integral of the dual billiard 
  in question. In the coordinate $\zeta=\frac z{z_0}$ the involution $\sigma_P$, $P=(z_0,z_0^2)$, defining the dual billiard 
  takes the form $\sigma_P(\zeta)=\eta_\rho(\zeta)=\frac{(\rho-1)\zeta-(\rho-2)}{\rho\zeta-(\rho-1)}$. In our case $\rho=2-\frac1{N+1}$. Therefore,
  $$\sigma_P(\zeta)=\frac{N\zeta+1}{(2N+1)\zeta-N}.$$
  The $\zeta$-coordinate of the point $Q(\tau)=(z(\tau^2),w(\tau^2))\in L_{P(\tau)}$ is equal to 
  $$\zeta(\tau):=\frac{z(\tau^2)}{z_0(\tau)}= \frac{\tau}{\tau+1}.$$
  Therefore, the $\zeta$-coordinate of its image under the involution $\sigma_{P(\tau)}$ is 
  $$\zeta_*(\tau):=\frac{(N+1)\tau+1}{(N+1)\tau-N},$$
  $$\sigma_P(Q(\tau))=Q_*(\tau)\in S_{\la}, \ \ Q_*(\tau)=(z_*(\tau),w_*(\tau)), \ z_*(\tau)=\zeta_*(\tau)z_0(\tau),$$
  $$ w_*(\tau)=2z_*(\tau)z_0(\tau)-z_0^2(\tau)=\frac{2\zeta_*(\tau)-1}{\zeta_*^2(\tau)}z_*^2(\tau).$$
  Let us now find the $\tau$-parameter value of  $F(Q(\tau),P(\tau))=(Q_*(\tau),P_*(\tau))$, 
  which will be denoted by $\tau_*$. One has 
  $$\tau_*^2=\frac{z_*^2(\tau)}{z_*^2(\tau)-w_*(\tau)}=\left(\frac{\zeta_*(\tau)}{\zeta_*(\tau)-1}\right)^2=\left(\tau+\frac1{N+1}\right)^2.$$
  Therefore, $\tau_*=\pm(\tau+\frac1{N+1})$. Here the sign is "$+$". Indeed,  otherwise, the map $F:S_\la\to S_\la$ would be an involution, and we get a contradiction, as in the proof of Theorem \ref{taf1}.  This proves Theorem \ref{taf2}.
  \end{proof}
  
  \subsection{Case b). Proof of Theorems \ref{tbg1}, \ref{tbs1}}
  \begin{proof} {\bf of Theorem \ref{tbg1}.} 
  Recall, see (\ref{exo2bnew}), that 
$$R_{b1}(z,w)=\frac{(w-z^2)^2}{(w+3z^2)(z-1)(z-w)}, \ \ \Sigma_{b1}=\{(0,0), (1,1), E\}.$$
Set 
\begin{equation}C_t:=\{ G_t(z,w)=0\}, \ \ G_t(z,w):=w-z^2+tz(z-1).\label{ctb1}
\end{equation}
The conics $C_t$ form a pencil of those conics passing through the points of the set $\Sigma_{b1}$ that are tangent to the infinity line at $E$. This implies that $C_t$ intersects a generic fiber $\gbll$ at points of $\Sigma_{b1}$ with indices 2, 2, 3 respectively. Thus, 
the intersection $C_t\cap\gdll$ contains only one extra point $D_{b1}(t)$, by B\'ezout Theorem. 
Instead of proving the latter intersection index statement we prove uniqueness of the extra point directly and find it explicitly. All the points of intersection $C_t\cap\gbll$ are solutions of the system of equations 
\begin{equation}\begin{cases} (w-z^2)^2-\la(w+3z^2)(z-1)(z-w)=0\\
w-z^2+tz(z-1)=0.\end{cases}\label{sysb1}\end{equation}
Substituting the second equation to the first one and cancelling $z-1$ transforms the first equation to
\begin{equation} t^2z^2(z-1)-\la(w+3z^2)(z-w)=0.\label{t2z}\end{equation}
We express $w$ and terms in (\ref{t2z}) via $z$ from the second equation in (\ref{sysb1}):  
$$w=z((1-t)z+t), \ \ w+3z^2=z((4-t)z+t), \ \ z-w=z(t-1)(z-1).$$
Substituting these formulas to (\ref{t2z}) and cancelling $z^2(z-1)$ yields (\ref{wb1}):  
$$t^2-\la((4-t)z+t)(t-1)=0,$$
$$ z=z(t)=\left(\frac{t^2}{\la(t-1)}-t\right)\frac1{4-t}=\frac{t((1-\la)t+\la)}{\la(t-1)(4-t)},$$
$$w=w(t)=z((1-t)z+t)=\frac{t((1-\la)t+\la)}{\la(t-1)(4-t)}\left(-\frac{t((1-\la)t+\la)}{\la(4-t)}+t\right)$$
$$=\frac{t^2((1-\la)t+\la)(3\la-t)}{\la^2(t-1)(4-t)^2}.$$
In the case, when $\la\neq0,\infty$ and  there is no cancellation in the above fractional expressions for $z$ and $w$, 
the above formulas yield a parametrized rational quartic $\Gamma_{b1,\la}$. 
It contains no true critical points by Proposition 
\ref{procritic}.  The indeterminacies $(0,0)$, $(1,1)$ and $E$  are its double points of transversal self-intersection, which follows from the parametrization: they correspond respectively to $t=0,\frac\la{\la-1}$, to two roots of the quadratic polynomial 
$t^2-4\la t+4\la$, and to $t=4,\infty$. The value $\la=1$ corresponds to a cancellation, see the argument below, and hence, 
is excluded, and the latter quadratic polynomial indeed has two distinct roots. 

Let us now describe those critical fibers that contain at least one true critical point. Each of them contains either a multiple irreducible component, or  at least two  components, see Proposition \ref{procritic}. Therefore, they correspond to those $\la$ for which some of the fractions in the above 
formulas for $z(t)$ and $w(t)$ become fractions of smaller degrees, i.e., some multipliers there cancel out, so that the resulting 
parametrized curve be of degree less than four. Let us treat the a priori possible cancellation cases. 

Case 1):  the multiplier $t-1$ cancels out with  $(1-\la)t+\la$. This is impossible, since $\la\neq\la-1$. 

Case 2):  the multiplier $t-1$ cancels out with $3\la-t$ in $w(t)$. This happens exactly when $\la=\frac13$. For this $\la$ after cancellation we get 
\begin{equation}(z(t),w(t))= (\frac{t(2t+1)}{(t-1)(4-t)}, -\frac{3t^2(2t+1)}{(4-t)^2}).\label{gb13f}\end{equation}
The curve thus parametrized is contained in $\gbllt$. We claim that it coincides with $\gbllt$. To this end it suffices to show that it is a quartic. Indeed, its intersection points with the infinity line correspond to $t=1, 4,\infty$. The points $t=1, \infty$ 
correspond to transversal intersections, and $t=4$ to a quadratic tangency point. Indeed, as $t\to4$ one has $w\simeq cz^2$, $c\neq0$, by (\ref{gb13f}). Thus, the total intersection index with the 
infinity line is equal to four.  The curve $\Gamma_{b1,\frac13}$ contains no true critical point, by Proposition \ref{procritic}, and hence, no singularities except maybe for the indeterminacies  $(0,0)$, $(1,1)$, $E$. They are transversal self-intersections, as in the above argument. Thus, they are not critical for  $\la=\frac13$ and $\Gamma_{b1,\frac13}$ is a non-critical level curve.
 
Case 3): the degree 1 polynomial $(1-\la)t+\la$ becomes a constant. This happens exactly when $\la=1$. Then 
$$\Gamma_{b1,1}=\{(w-z^2)^2-(w+3z^2)(z-1)(z-w)=0\} =\{ zH(z,w)=0\},$$ 
$$H(z,w)=2z^3-3z^2w-3z^2+6wz-w^2-w.$$
For $\la=1$  parametrization (\ref{wb1}) defines an irreducible cubic 
\begin{equation}\Gamma_{b1,1,*}=\{(\frac t{(t-1)(4-t)}, \frac{t^2(3-t)}{(t-1)(4-t)^2}) \ | \ t\in\oc\}.\label{parcub}\end{equation} 
This together with the above formula implies that $\Gamma_{b1,1,*}=\{H=0\}$. The intersection $\{ z=0\}\cap \Gamma_{b1,1,*}$ 
consists of the points $(0,0)$, $(0,-1)$, $E$, which correspond to $t=0,\infty,-4$ respectively. They are their transversal intersections, and they are regular points of the cubic 
$\Gamma_{b1,1,*}$.  Therefore, $(0,-1)$  is a Morse 
critical point of the function $R_{b1}$. The parameter  value $t$ of the curve $\Gamma_{b1,1,*}$ corresponding to the 
point $(1,1)$ is found from the equation $\frac wz=\frac{t(3-t)}{4-t}=1$: it is its  double root $t=2$. Therefore, $(1,1)$ is a singularity of the curve $\Gamma_{b1,1,*}$ and not a self-intersection. 
The genus   of the rational cubic $\Gamma_{b1,1,*}$ is equal to  $\frac{(3-1)(3-2)}2=1$ minus the sum of $\delta$-invariants of its singularities, and it should vanish. Therefore, $(1,1)$ is the unique singularity, not self-intersection, and its $\delta$-invariant is equal to one. Hence, it is a cusp.

Case 4): the multiplier $4-t$ cancels with either $(1-\la)t+\la$, or $3\la-t$. Each of the latter cancellations 
happens exactly when $\la=\frac43$. Then it cancels out in both $z(t)$ and $w(t)$. After cancellations we get the parametrized curve 
\begin{equation}\gblctl:=\{(z_1(t),w_1(t))=(\frac{t}{4(t-1)}, \frac{3t^2}{16(t-1)}) \ | \ t\in\oc\}\subset\gbll.\label{gb11}\end{equation}
It is a regular conic, and  direct calculation shows that 
\begin{equation}\gblctl=\{ w=\frac{3z^2}{4z-1}\}.\label{gb11f}\end{equation}
The ambient level curve $\gbll$ being a quartic, it contains an additional conic, which we will denote $\gblctd$. Let us find it. 
To this end, we take a small $\delta\in\cc$, set 
$$\la=\la_\delta:=\frac43+\delta,$$
 and for this $\la$ we pass to a new variable $\tau$ so that 
 $$4-t=\delta\tau.$$
 Writing the parametrization of the curve $\gbll$ with $\la=\la_\delta$ by the new parameter $\tau$ and passing to limit, as 
 $\delta\to0$ we get 
 $$(1-\la)t+\la=(1-\frac43-\delta)t+\frac43+\delta=\frac{4-t}3+\delta(1-t)=\delta\frac{\tau-9}3+O(\delta^2),$$
 $$z=\frac{t((1-\la)t+\la)}{\la(t-1)(4-t)}\to z_2(\tau):=\frac{\tau-9}{3\tau},$$
 $$w=zt\frac{3\la-t}{\la(4-t)}\simeq 4z\frac{4-t+3\delta}{\frac43(4-t)}\to w_2(\tau):=3\frac{\tau+3}{\tau}z_2(\tau)=\frac{(\tau-9)(\tau+3)}{\tau^2}.$$
 The above limit curve parametrized by $(z(\tau),w(\tau))$ will be denoted by $\gblctd$. It is contained in $\gblct$, by construction. Let us represent it as a graph. Inverting the function $z=z_2(\tau)$ yields 
 $$\tau=-\frac{9}{3z-1}, \ \ w_2(\tau)=3z\frac{-9+9z-3}{-9}=z(4-3z), \ \ \gblctd=\{ w=z(4-3z)\}.$$
 The union of two distinct conics $\gblctl$ and $\gblctd$ coincides with their ambient quartic $\gblct$, by degree coincidence. 
 Their intersection points are the base points $(0,0)$, $(1,1)$, $E$ and the point $(\frac13, 1)$, which is found from the equation $\frac{3z^2}{4z-1}=z(4-3z)$. All these intersections are transversal, by B\'ezout Theorem. 
 
 To finish the proof of Statement 1) of Theorem \ref{tbg1}, it remains to show that for $\la\neq0,1,\frac43,\infty$ the indeterminacies 
 $(0,0)$, $(1,1)$, $E$ are double points of transverse self-intersection of the curve $\Gamma_\la$. They are self-intersections:  the point $(0,0)$ corresponds to $t=0,\frac\la{\la-1}$, $E$ corresponds to $t=4,\infty$, and $(1,1)$ to roots of the quadratic 
 polynomial $t^2-4\la t+4\la$, and the latter roots are distinct for $\la\neq0,1$. The latter self-intersections are transverse, by 
 $\delta$-invariant argument, as in Case 3): in our case the sum of their $\delta$-invariants is equal to 
 $3=\frac{(4-1)(4-2)}2$, and hence, each $\delta$-invariant is equal to one and the corresponding self-intersection is 
 transversal. Statement 1) is proved. 
 
 Statements 5) and 6) are also proved. Together they imply Statement 4)  for finite critical values $\la$. For $\la=\infty$ 
 Statement 4) follows immediately from (\ref{exo2bnew}) by direct calculation of intersection points of the conic $\{ w+3z^2=0\}$ with 
 the lines $\{z=1\}$, $\{ z=w\}$.

Now for the proof of Theorem \ref{tbg1} it remains to prove its critical indeterminacy point statement 3). Together with Statement 4), it will imply Statement 2). The indeterminacies 
$(0,0)$ and $E$ are both double points of transversal self-intersection for each level curve $\Gamma_\la$ with $\la\neq0,\infty$, which was shown above. For $\la=\infty$ this statement is obvious. The point $(1,1)$ is a double point of transversal self-intersection of each level curve with $\la\neq0,1,\frac43,\infty$, by already proved Statement 1), and hence, is not critical for these $\la$. It is clearly critical for the values $\la=0,1$: it is not a double point of 
transversal self-intersection for each one of the corresponding level curves. For $\la=0$ this is obvious. For $\la=1$ 
this follows from already proved Statement 5): it is a cusp of the component $\Gamma_{b1,1,*}$. It remains to exclude the other a priori potential critical values $\la=\frac43,\infty$ and prove that $(1,1)$ is a double point of transversal self-intersection for the critical corresponding level curves.   For $\la=\infty$ this is obvious. For $\la=\frac43$ this was proved above. 
Theorem \ref{tbg1} is proved.
\end{proof}

\begin{proof} {\bf of Theorem \ref{tbs1}.} For $\la\neq0,1,\frac43,\infty$ the curve $\Gamma_\la$ has two regular local branches at each indeterminacy $(0,0)$, $(1,1)$, $E$, see Theorem \ref{tbg1}, Statement 1); six branches in total. Two branches, namely one branch at $(0,0)$ and one at $E$, with base points $t=0,4$ in (\ref{wb1}), are both tangent to $\gamma$. The other four  local branches are transversal to $\gamma$, by B\'ezout Theorem. Their base points are the values $t$ given by 
(\ref{brb}).   Therefore, $S_\la$ is elliptic and branched over $\Gamma_\la$ 
at  points (\ref{brb}), by Proposition \ref{ctype}, Statement 5). Hence, (\ref{diffb}) is a holomorphic differential on 
$S_\la$. Statement 4) of Theorem \ref{tbs1} follows 
by ellipticity and Proposition \ref{elltr}. 

The preimage $S_{b1,10}:=p_M^{-1}(\{ z=0\})$ is projected to the axis $\{ z=0\}$ as a double covering with 
ramification points $(0,0)$ and $E$, by Proposition \ref{ctype}, Statement 4). Set $S=p_M^{-1}(\Gamma_{b1,1,*})$. 
The only potential branching points of the Riemann surface double covering $p_M:S\to\Gamma_{b1,1,*}$ are the points of 
intersection $\Gamma_{b1,1,*}\cap\gamma$, by the same proposition. These points are $(0,0)$, $E$ and $(1,1)$. 
The two former ones are regular points of the curve $\Gamma_{b1,1,*}$ where it is quadratically tangent to $\gamma$: 
they correspond to the parameter values $t=0,4$, and the latter statements follow immediately by elementary calculation. Therefore, $(0,0)$ and $E$ are  not branching points, by Proposition \ref{ctype},   
Statement 4). The point $(1,1)$ is a cusp of the curve $\Gamma_{b1,1,*}$, thus,  not a self-intersection. Hence, the 
projection $p_M:S\to\Gamma_{b1,1,*}$ has at most one branching point. But if a double covering 
over a Riemann sphere is branched, then it has at least two distinct branching points. Therefore, the covering in question is not branched. The covering surface $S$ is thus a union of two rational curves, since the base $\Gamma_{b1,1,*}$ is rational. 
Statement 2) of Theorem \ref{tbs1} is proved.

Each conical component of the curve $\Gamma_{b1,\frac43}$ is tangent to 
$\gamma$ at $(0,0)$, respectively $E$, and intersects $\gamma$ transversally at $E$, $(1,1)$, 
respectively at $(0,0)$, $(1,1)$. Therefore, its preimage in $M$ is its double covering with two 
branching points being the above transversal intersection points, see Statement 4) of Proposition \ref{ctype}. Statement 3) of 
Theorem \ref{tbs1} is proved. 
\end{proof}

\subsection{Case c). Proof of Theorems \ref{tc}, \ref{tc'}, \ref{tc2}} 
We use the next proposition.
\begin{proposition} \label{locbr} Each curve $\Gamma_{c1,\la}$ with $\la\neq0,\infty$ has two regular local branches at each 
base point, and each of them is tangent to the cubic $\Gamma_{c1,\infty}=\{ 1+w^3-2zw=0\}$ with cubic contact. 
The local branches also have cubic contact between them. The curve $\Gamma_{c1,\la}$ considered as  zero 
divisor of the function $R_{c1}(z,w)-\la$ has no multiple component. 
\end{proposition}
\begin{proof} It suffices to prove the first statement of the proposition only at the base point $(1,1)$, due to order three symmetry. 
Take the local coordinates $u=1+w^3-2zw$, $v=z-1$ centered at $(1,1)$. Recall that the cubic  $\Gamma_{c1,\infty}$ 
is quadratically tangent to $\gamma$ at the base points, by regularity (being a graph $\{ z=\frac{1+w^3}{2w}$), order three symmetry and B\'ezout Theorem. 
Therefore, 
$$w-z^2=a(u+bv^2)+\text{h.o.t.}, \ \ \ a,b\neq0,$$
$$(w-z^2)^3-\la(1+w^3-2zw)^2=a^3(u+bv^2)^3-\la u^2+\text{h.o.t.}$$
\begin{equation}=-Q(u,v)+ \text{h.o.t.,} \ \ Q(u,v)=\la u^2-(ab)^3v^6.\label{quasih}
\end{equation}
Here the polynomial $Q(u,v)$ is (3,1)-quasihomogeneous of quasihomogeneous degree $2\times 3=1\times 6=6$, 
and  "h.o.t." means a sum of monomials lying above the Newton diagram of the rest of the expression in question. 
The polynomial $Q$ is a product of two 
prime $(3,1)$-quasihomogeneous polynomials $\sqrt{\la} u\pm(ab)^{\frac32}v^3$ for $\la\neq0,\infty$. Each 
factor corresponds to a regular local branch of the curve $\Gamma_{c1,\la}$, by Statement (ii) in Subsection 1.4. 
The local branches 
in question are clearly tangent to the line $\{ u=0\}$, i.e., to the cubic $\Gamma_{c1,\infty}$,  with cubic contact. The first statement of  Proposition \ref{locbr} is proved. Its second statement follows  by the same argument, since the 
above prime polynomials differ by cubic terms. 
Each irreducible component $C$ of the curve $\Gamma_{c1,\la}$ intersects 
$\Gamma_{c1,\infty}$. Each intersection point is a base point, and hence, one of the above local branches is contained in $C$. The local branches are not multiple, and hence,  $C$ also isn't. The proposition is proved.
\end{proof}

\begin{proof} {\bf of Theorems \ref{tc} and \ref{tc'}.} Theorem \ref{tc} implies Theorem \ref{tc'}, by (\ref{rc1c2m}) and straightforward calculations to find the $M_c$- preimages of the conics from  Theorem \ref{tc} and of their intersections. 
Thus, it suffices to prove Theorem \ref{tc}. 
 The curve $\Gamma_{c1,\la}$ is a sextic. For a generic $\la$ it is  irreducible  (Lemma \ref{irred}) and its only singular points  (if any) are the base points  (Bertini Theorem \cite[vol. 1, chapter 1, section 1]{grh}). Its genus $g$ is equal to 
 $\frac{(6-1)(6-2)}2=10$  minus the sum of $\delta$-invariants  of its germs at the base points, by (\ref{hiron}). The 
 $\delta$-invariants are equal to one and the same number denoted $\delta$, by 
order three symmetry $(z,w)\mapsto(\var z,\bar\var w)$, which fixes $R_{c1}$ and cyclically permutes the base points. 
The genus $g=10-3\delta$ should be equal to either one, or zero, by Proposition \ref{ctype}. Hence, $g=1$, and $\delta=3$, since
   $10$ is not divisible by three. The fact that $\delta=3$ also follows from Proposition \ref{locbr}. 
    Thus, $\Gamma_{c1,\la}$ is an elliptic curve. 
   
\begin{proposition} \label{pifany} A critical level curve $\Gamma_{c1,\nu}$ with $\nu\neq0,\infty$, if any, is a union of three conics permuted by the order three symmetry.  Each of them is tangent to 
the cubic $\Gamma_{c1,\infty}$  at some two base points with cubic contact.
\end{proposition}
\begin{proof} Case 1): the curve $\Gamma_{c1,\nu}$ is irreducible. It cannot be multiple, by Proposition \ref{locbr}. 
Thus, it is an irreducible sextic. The base points are not critical for  values $\la\neq0,\infty$, by the same proposition. 
Hence, if  $\nu$ is critical, then $\Gamma_{c1,\nu}$ contains some  true critical points. Their number cannot be greater than one: otherwise the genus $g(\Gamma_{c1,\nu})$ found by (\ref{hiron}) would  be negative, since the  $\delta$-invariants at the three base points are equal to three, while $\frac{(d-1)(d-2)}2=10$. Thus, the critical point is unique and fixed by the symmetry. Therefore, in homogeneous coordinates $[z:w:t]$ it is one of the points $[0:0:1]$, $[0:1:0]$, $[1:0:0]$. But the function $R_{c1}$ takes values 
 0, 0, $\infty$ respectively at the latter points. Hence, $\nu\in\{0,\infty\}$ -- a contradiction. 
 
 Case 2): the curve $\Gamma_{c1,\nu}$ contains at least one line $L$. Then $L$ 
 intersects the cubic $\Gamma_{c1,\infty}$, and only at some base points. At each base point  $p\in L$ it   should contain some of 
 the corresponding local branches in $\Gamma_{c1,\nu}$ from Proposition \ref{locbr}. Hence, $L$ is tangent to $\Gamma_{c1,\infty}$ at $p$ with cubic contact.  But their tangency order is quadratic. Indeed,  it suffices to check this at the base point $(1,1)$, where 
 $L=\{ w=2z-1\}$.  One has $(1+w^3-2zw)_L=1+w^3-w(w+1)=(w-1)^2(w+1)\simeq2(w-1)^2$, as $w\to1$. Hence, the tangency 
 is quadratic, -- a contradiction. 
  
  Case 3):  the sextic $\Gamma_{c1,\nu}$ contains at least two irreducible components, and each of them is non-linear. Hence, 
  their number is at most three. 
  
  Subcase 3.1): some component $C$ is fixed by the order three symmetry, and hence, intersects $\gamma$ by the three base points. Then at each base point $C$ contains exactly one 
  local branch from Proposition \ref{locbr}: presence in $C$ of two local branches at some base point would imply that 
  $\Gamma_{c1,\nu}=C$, and hence, $\Gamma_{c1,\nu}$ is irreducible. The union of the other  components is  symmetry-invariant. Their number, which is equal to either one, or two, should be a divisor of 3. Hence, it is one, invariant component, which we will denote by $C'$. The components $C$ and $C'$ form 
  a sextic. They are cubics intersected only at the three base points, by B\'ezout Theorem and since their intersection index at each base point is equal to three, by Proposition \ref{locbr}. Neither $C$, nor $C'$ contains true critical points, by the same $\delta$-invariant argument, as in Case 1), now applied to $C$, $C'$. In more detail, each component 
  $C$, $C'$ has at most one singular point, by (\ref{hiron}), and hence, it is invariant under the symmetry if it exists. 
  In this case $\nu\in\{0,\infty\}$, as in Case 1), which is not the case by assumption. 
  Therefore, the curve $\Gamma_{c1,\nu}=C\cup C'$ is non-critical, as in Case 1), -- a contradiction.  
  
  Subcase 3.2): the sextic $\Gamma_{c1,\nu}$ is a union of three conics  permuted by the symmetry. Then each conic   
  is tangent to the cubic $\Gamma_{c1,\infty}$ at two base points with cubic contact, by Proposition \ref{locbr},  
  B\'ezout Theorem and since there are no intersections outside the base points.  Proposition  \ref{pifany} is proved.
  \end{proof}
  
  Let us now find the above conics for the functions $R_{c1}$, $R_{c2}$. 
  To simplify calculations, we first find them for the function $R_{c2}$. We use the fact that $R_{c2}$ has order three symmetry given by transformation (\ref{symc2}). The latter transformation is the conjugate of the order three symmetry of the 
  function $R_{c1}$ by the projective transformation $M_c$. It can be also found geometrically from the conditions that 
  it permutes cyclically the base points $(0,0)=[0:0:1]$, $E=[0:1:0]$, $(1,1)=[1:1:1]$ and the intersection points of the tangent lines to $\gamma$ at them. That is, it cyclically permutes the points $[1:0:0]$, $[1:2:0]$, $[\frac12:0:1]$. 
  
  The cubic polynomial in the denominator of the function $R_{c2}$  is equal to 
  $-(w+8z^2)(1+w)+10zw+8z^3$. Its lower $(2,1)$-quasihomogeneous part at $(0,0)$ is 
  $-(w+8z^2)$. Hence, the conic $\mcc_1':=\{ w+8z^2=0\}$ is tangent to $\Gamma_{c2,\infty}$ 
  at $(0,0)$ with at least cubic contact. The same statement holds at the infinite point $E=[0:1:0]$ 
   and is proved analogously by passing to the affine coordinates $(\wt z,\wt w)=(\frac zw, \frac1w)$. 
   The tangency contacts are exactly cubic, by B\'ezout Theorem. The other conics tangent to 
   $\Gamma_{c2,\infty}$ with cubic contact at the other base point pairs are obtained from the conic $\mcc_1'$ by order three symmetry, and 
   we get formulas (\ref{conic2}) for them. 
   
   Let us now check that the union of conics (\ref{conic2}) is a level curve $\Gamma_{c2,\la}$ and find the corresponding 
   value $\la$. It suffices to prove that $R_{c2}\equiv const$ on $\mcc_1'$, by symmetry. Indeed, substituting $w=-8z^2$ to the equation $R_{c2}=\la$ yields 
   $$(w-z^2)^3=-9^3z^6=\la(10zw+8z^3)^2=\la(-80+8)^2z^6=72^2\la z^6,$$
   which holds exactly for $\la=-\frac9{64}$. Thus, the conic $\mcc_1'$ lies in the curve $\Gamma_{c2,-\frac9{64}}$. 
   The intersection of any two conics (\ref{conic2}) consists of one their common tangency base point, where their contact is 
   cubic, and one more point of transversal intersection. The latter transversal intersection points are found by straightforward 
   calculation, and we get formulas for them given by Theorem \ref{tc'}. 
   
   The critical value  of the function $R_{c1}$ different from $0,\infty$ is $\frac{27}{64}$: it is obtained from the above critical value by multiplication by $-3$, see (\ref{rc1c2m}). The equations for the corresponding conics are obtained from 
   the quadratic polynomials defining conics (\ref{conic2}) by the inverse to the projective transformation $M_c$, and we get 
   (\ref{3con}). The indeterminacy point statements of Theorems \ref{tc} and \ref{tc'} follow from Proposition \ref{locbr}. Theorems \ref{tc} and \ref{tc'} are proved.
   \end{proof}
   
   \begin{proof} {\bf of Theorem \ref{tc2}.} Fix a small ball $B\subset\cp^2$ centered at $V=[1:0:0]$, the double point of the curve 
   $\Gamma_{c1,\infty}$,  in the  affine chart $(\wt z,\wt w)=(\frac1z,\frac wz)$ centered at $V$.   We consider that it is disjoint from the conic $\gamma$ and has radius  less than $\frac12$.  The inverse function $p_M^{-1}$ has two holomorphic branches on $B$, 
   since $B$ is disjoint from $\gamma$ and simply connected. Let us denote the latter holomorphic branches by 
   $g_1$ and $g_2$. The normalization, i.e., the parametrizing Riemann surface of the curve $\Gamma_{c1,\infty}$ is 
   the Riemann sphere. The complement $\Gamma^o_{\infty}:=\Gamma_{c1,\infty}\setminus B$ in the normalization 
   is a cylinder with two boundary components lying in $\partial B$, let us denote them by $X_1$, $X_2$. 
    A holomorphic branch of the function $p_M^{-1}$ on $X_1$ extends analytically to the whole normalization of the curve $\Gamma^o_{\infty}$ along paths 
    avoiding the base points. The extension is well-defined and single-valued, including the base points. Indeed, they are  quadratic tangency points of the curve $\Gamma_{\infty}$ with $\gamma$ (B\'ezout Theorem and symmetry) and hence, $p_M^{-1}|_{\Gamma_{\infty}^o}$ has trivial monodromy along circuits around each base point, 
    by Proposition \ref{ctype}, Statement 4). This proves Statement 2) of Theorem \ref{tc2}.

    \begin{remark} \label{rperm} The above extension from $X_1$ to $X_2$ along the cylinder $\Gamma^o_{\infty}$ permutes the holomorphic branches  $g_{1,2}$. Indeed, the cubic $\Gamma_{c1,\infty}$ is the graph of the function $z=\frac{1+w^3}{2w}$. 
     The equation $w=2z_0z-z_0^2$ on a point $P=(z_0,z_0^2)\in\gamma$ for which the tangent line $L_P$ passes 
     through a given point $(z,w)$ of the above graph yields 
     $$z_0^2-\frac{z_0(1+w^3)}{w}+w=0, \ \ z_0\in\{ w^2, \ \frac1w\}.$$
     As $w$ goes along a path from zero to infinity, the point $(z,w)$ of the graph goes from $V$ to $V$. The corresponding 
     values $z_0$ at the boundary points of the path are distinct and equal to $0$ and $\infty$. Therefore, the extension permutes the branches $g_{1,2}$. 
     \end{remark}
     
     {\bf Claim 1.} {\it For every $\la$ big enough the intersection $\Gamma_{c1,\la}\cap B$ consists of two topological cylinders 
     $C_{\pm,\la}$ with boundary curves being closed curves $X_{j,\pm,\la}\subset\partial B$, $j=1,2$, $X_{j,\pm,\la}\to X_{j}$, 
     as $\la\to\infty$. We numerate the boundary components so that $C_{\pm,\la}$ is adjacent to the curves $X_{1,\pm,\la}$ and 
     $X_{2,\mp,\la}$.}
     
     \begin{proof} In the affine chart  $(\wt z,\wt w)$  one has
          $$\Gamma_{c1,\la}=\{(\wt w\wt z-1)^{-\frac32}(\wt z^3+\wt w^3-2\wt w\wt z)=\la^{-\frac12}\}.$$
   The left-hand side of the latter equation has two holomorphic branches on $B$ that differ by sign. Each of them has Morse critical point at $V=(0,0)$. 
   Therefore the preimage of a small number $\la^{-\frac12}$ under each branch is a topological cylinder, by Morse Lemma. 
   This proves the claim. 
   \end{proof}
    
    {\bf Claim 2.} {\it Let $\overline B$ contain no base points. Then for every $\la$ big enough the complement 
    $\Gamma^o_\la:=\Gamma_{c1,\la}\setminus B$ is a union of two topological cylinders $Y_{\pm,\la}$ 
    touching each other at the three base points. The boundary components of each cylinder $Y_{\pm,\la}$ are the above curves $X_{j,\pm,\la}$, $j=1,2$.} 

     \begin{proof} Let $B'\Subset B$ be a smaller ball centered at $V$, let $\Gamma'_{c1,\infty}=\Gamma_{c1,\infty}\setminus B'$: this is a topological cylinder taken with boundary.  Set 
     $$u:=-\frac{1+w^3-2zw}{2w}=z-\frac{1+w^3}{2w}: \ \ \ \Gamma_{c1,\infty}=\{ u=0\}.$$
     The coordinates $(u,w)$ are global holomorphic coordinates on a neighborhood  of the curve  
     $\Gamma_{c1,\infty}'$, and the projection $p_w:(u,w)\mapsto w$ sends $\Gamma_{c1,\infty}'$ 
 bijectively onto a closed bounded topological annulus  $\mca\subset\cc\setminus\{0\}$ as the graph of function $z=\frac{1+w^3}{2w}$. For every $\la$ big enough the intersection $\Gamma_{\la,\mca}:=\Gamma_{c1,\la}\cap p_w^{-1}(\mca)$ is a  double covering over $\mca=\Gamma'_{c1,\infty}$ under the projection $p_w$. This follows from construction, since the equation $R_{c1}(z,w)=\la$ 
 can be rewritten as 
 $$u^2=\frac{f(u,w)}\la, \ \ \ f(u,w)=\frac{(w-z^2)^3}{4w^2}.$$
 We claim that the normalization of the curve $\Gamma_{\la,\mca}$ consists of two components, each bijectively projected onto 
 $\mca$. Indeed, the only a priori possible branching points are the base points in $\mca$, i.e., points where $u=w-z^2=0$. 
  But 
 the covering is not branched there, which follows from Proposition \ref{locbr}: it implies that 
 there are two local branches of the covering over a small neighborhood of each base point.  Finally, the above double covering projection sends $\Gamma_{\la,\mca}$ onto the topological cylinder $\Gamma'_{c1,\infty}=\mca$ without ramification, and the preimage 
 of each boundary component consists of two closed curves, by Claim 1. Therefore, the covering surface $\Gamma_{\la,\mca}$ 
 is a disjoint union of two cylinders, each 
 bijectively projected onto the base. For every $\la$ big enough the complement $\Gamma_{c1,\la}\setminus B$ is obtained from 
  the curve $\Gamma_{\la,\mca}$ by cutting off four small topological cylinders  adjacent to its four boundary curves and to the curves $X_{j,\pm,\la}$.  This implies that the complement is itself also a disjoint union of two topological cylinders with boundary 
  components $X_{j,\pm,\la}$. The claim is proved. 
  \end{proof}

     For the proof of the statement of Theorem \ref{tc2} saying that for each regular $\la$ the curve 
     $S_\la$ consists of two components bijectively projected to $\Gamma_{c1,\la}$, we have to show that the inverse $p_M^{-1}$ has two holomorphic branches on the normalization of the curve $\Gamma_{c1,\la}$. Let us construct the latter branches. 
     Take an initial holomorphic branch $g$ of the inverse $p_M^{-1}$ on $\overline B$, say $g=g_1$, and its restriction 
      to the cylinder $C_{+,\la}\subset B$. It has analytic extension through $X_{1,+,\la}$ to the cylinder $Y_{+,\la}$ from Claim 2, since its only a priori possible singularities are the base points. They are points of quadratic tangency of the curve 
      $\Gamma_{c1,\la}$ with $\gamma$, since this is true for $\la=\infty$ and each its local branch  has cubic contact 
      with $\Gamma_{c1,\infty}$ there. Hence, they are not singularities, as in the above proof of Statement 2) of Theorem \ref{tc2}. The restriction to $X_{2,+}$ of thus extended holomorphic branch of $p_M^{-1}|_{Y_{+,\la}}$ is $g_2\neq g_1$, as in Remark \ref{rperm}.  Let us  extend the new branch $g_2$ through $X_{2,+}$ to the cylinder $C_{-,\la}$, 
      which is adjacent to $X_{2,+}$ and $X_{1,-}$. Then we extend it to  $Y_{-,\la}$, and the restriction to 
      $X_{2,-}$ of thus extended branch becomes $g_1$. Then we extend it to the initial adjacent cylinder $C_{+,\la}$, Thus, we returned back and get the initial branch $g_1$. This yields a global holomorphic branch of the inverse $p_M^{-1}$ on $\Gamma_{c1,\la}$. 
      Starting from the other local branch $g_2$ yields another holomorphic branch. Thus, $S_\la$ consists of two elliptic components biholomorphically projected to $\Gamma_\la$. Each of them is $F$-invariant, and $F$ acts as a torus translation there, by Proposition \ref{elltr}. Statement 1) of Theorem \ref{tc2} is proved.

     Each conic from (\ref{3con}) is quadratically tangent to $\gamma$ at two base points. It intersects 
     $\gamma$ at no other point, by B\'ezout Theorem. Therefore, the inverse 
    $p_M^{-1}$ has two global holomorphic branches on the conic in question, by triviality of monodromy of its extension around a 
    quadratic tangency point. This implies Statement 3) of Theorem \ref{tc2}. Statement 4) then follows immediately 
    by (\ref{rc1c2m}). 
    \end{proof}

   \subsection{Case d). Proof of Theorems \ref{tdg}, \ref{tdcr} and \ref{td}}
   
   \begin{proof} {\bf of Theorem \ref{tdg}.} 
   Recall that 
$$R_{d}(z,w)=\frac{(w-z^2)^3}{(w+8z^2)(z-1)(w+8z^2+4w^2+5wz^2-14zw-4z^3)},$$
$$\gdll=\{ R_d=\lambda\}=\{ \mcp_\la(z,w)=0\},$$
 \begin{equation} \mcp_\la(z,w):= (w+8z^2)(z-1)(w+8z^2+4w^2+5wz^2-14zw-4z^3)-\la^{-1}(w-z^2)^3.\label{mcpl}\end{equation}
 {\bf Motivation of construction of rational parametrization of the curve $\gdll$.} Consider the set $\Sigma$ of three common points of the curves $\gdll$: 
 $$\Sigma=\{(0,0), (1,1), E\}, \ \ E=[0:1:0] \text{ is the infinite point of   } \gamma=\{ w=z^2\}.$$ 
 Recall that the curves $\gdll$ have degree 6. 
 Our goal is to construct a one-parametric family of algebraic curves $\mcq_t$ of a given degree, also containing $\Sigma$, and having big enough intersection index with $\gdll$ at $\Sigma$ so that there exists exactly one extra intersection point $D(t)\in \mcq_t\cap\gdll$ lying outside $\Sigma$. Then the correspondence $t\mapsto D(t)$ will be a rational parametrization of the curve $\gdll$.  The curves $\mcq_t$ will be the cubics 
$$\mcq_t:=\{ G_t(z,w)=0\},$$ 
\begin{equation} G_t(z,w)=w+8z^2+4w^2+5wz^2-14zw-4z^3+t(w-z^2)(w-z).\label{cubp}\end{equation}
 Let us describe how the above cubics were found. Using Newton diagrams one can show that for a generic $\la\in\cc^*$ 
 the germs of the  curve $\Gamma_{d,\la}$ at the base points $(0,0)$, $(1,1)$, $E$ satisfy the statements of Theorem 
 \ref{locind}:  
 
 - the germ  at $(0,0)$ consists of two regular local branches tangent 
 to the conic $C:=\{ w+8z^2=0\}$ with at least cubic contact, see Proposition \ref{2-cubic};
  
 - the germ at  $E$  is a union of three 
 regular local branches: one transversal to the infinity line; two tangent to the 
 conics $C$, respectively $\{4w+5z^2=0\}$, with at least cubic contact, see Proposition \ref{45inf} below;
 
 - the germ at $(1,1)$ is a union of three regular pairwise transversal  branches, see Proposition \ref{lgen} below. 
 
 We construct a pencil of cubics $\mcq_t\subset\cp^2$, $t\in\oc$, satisfying the following conditions; here by $G_t(z,w)$ we denote their defining polynomials:
 
(i) Each cubic $\mcq_t$ passes through $(0,0)$ and has at least cubic contact with the conic $C$ in the following sense: the restriction  $G_t|_C$ of its defining polynomial has trivial 2-jet, i.e., 
$(0,0)$ is its zero of order at least three. 

(ii) Each $\mcq_t$ passes through $(1,1)$ and the corresponding polynomial $G_t(z,w)$ has trivial 1-jet there.

(iii) Each $\mcq_t$ intersects the infinity line at $E$ with index at least two.

Each of conditions (i), (ii) is given by three linear equations on the coefficients of a homogeneous polynomial defining $\mcq_t$, and (iii) is given by two equations. The homogeneization of the cubic polynomial in the denominator of the integral $R_d$, which defines the cubic $\mcs$, see (\ref{cands}), 
satisfies them. It appears that the homogeneization of the polynomial $(w-z^2)(w-z)$ also satisfies them. Therefore, so do all the cubics from the pencil 
(\ref{cubp}) generated by them. Statements (i)--(iii) and the above properties of germs of the curves $\gdll$ at $\Sigma$ together imply that the intersection indices of generic curves $\gdll$ and $\mcq_t$ are equal to at least 6 at $(0,0)$,  6 at $(1,1)$ and 
5 at $E$. Thus, their intersection index at $\Sigma$ is equal to at least 17.  On the other hand, their total intersection index is equal to $3\times6=18$, by B\'ezout Theorem, and 18-17=1. Therefore, there is at most one extra intersection point $D(t)\in (\mcq_t\cap\gdll)\setminus\Sigma$ 
for generic $\la$ and $t$.  Let  us find it explicitly as a function 
of $t\in\oc$. This will yield a  rational parametrization of the curve $\gdll$. 

Let $(z,w)$ denote coordinates of a point of intersection $\mcq_t\cap\gdll$. They satisfy the system of equations 
\begin{equation} \begin{cases} w+8z^2+4w^2+5wz^2-14zw-4z^3+t(w-z^2)(w-z)=0\\
(w-z^2)^2+\la t(w+8z^2)(z-1)(w-z)=0.\end{cases}\label{cases1}\end{equation}
The first equation in (\ref{cases1}) is the equation $G_t(z,w)=0$. The second one is the equation obtained from the equation $R_d(z,w)-\la=0$ as follows. First we replace the cubic polynomial in the denominator of $R_d$ by $-t(w-z^2)(w-z)$, 
see the first equation. Then we cancel $w-z^2$. Multiplying the equality thus obtained by its denominator we get the second equation in (\ref{cases1}). 

We now solve system (\ref{cases1}) and get a rational parametrization of the curve $\gdll$. 
To do this, we rewrite system (\ref{cases1}) in the new variables 
$$(z,u), \ \ \ u:=w-z^2:$$
\begin{equation}\begin{cases} u+9z^2+4(u+z^2)^2+5(u+z^2)z^2-14z(u+z^2)\\ 
-4z^3+tu(u+z(z-1))=0\\ \\ 
u^2+\la t(u+9z^2)(z-1)(u+z(z-1))=0.\end{cases}\label{cases2}\end{equation}
Rewriting the left-hand sides in (\ref{cases2}) as quadratic polynomials in $u$ yields
\begin{equation} \begin{cases}(4+t)u^2+(z-1)((13+t)z-1)u+9z^2(z-1)^2=0\\
(1+\la t(z-1))u^2+\la tz(z-1)(10z-1)u+9\la t z^3(z-1)^2=0.\end{cases}\label{cases3}\end{equation}
Subtracting the first equation in (\ref{cases3}) multiplied by $\la t z$ from the second one cancels the $u$-free term and 
after dividing by $u$ yields 
$$(-\la t(3+t)z+1-\la t)u-\la t(3+t)z^2(z-1)=0,$$
\begin{equation} u=\frac{\la t(3+t)z^2(z-1)}{-\la t(3+t)z+1-\la t}.\label{uform}\end{equation}
Substituting (\ref{uform}) to the second equation in (\ref{cases3}) yields
$$(1+\la t(z-1))\frac{\la^2 t^2(3+t)^2z^4(z-1)^2}{(-\la t(3+t)z+1-\la t)^2}+\frac{\la^2t^2(3+t)z^3(z-1)^2(10z-1)}{-\la t(3+t)z+1-\la t}$$
$$+9\la t z^3(z-1)^2=0.$$
Cancelling $\la t z^3(z-1)^2$ and multiplying by $(-\la t(3+t)z+1-\la t)^2$ yields 
$$\la t(1+\la t(z-1))(3+t)^2z+\la t(3+t)(-\la t(3+t)z+1-\la t)(10z-1)$$
$$+9(-\la t(3+t)z+1-\la t)^2=0.$$
The $z^2$-term in the above equality cancels out. The remaining terms are
$$z\left(-(3+t)^2\la^2t^2+\la t(3+t)^2+\la t(3+t)(\la t(3+t)-10(\la t-1))+18\la t(3+t)(\la t-1)\right)$$
$$+\la t(3+t)(\la t-1)+9(\la t-1)^2=0.$$
Simplifying this equality yields
$$z(\la t(3+t)^2-10\la t(3+t)(\la t-1)+18\la t(3+t)(\la t-1))+(\la t-1)(\la t^2+12\la t-9)=0.$$
This yields  formula (\ref{forztd}) for $z$:
\begin{equation}z=-\frac{(\la t-1)(\la t^2+12\la t-9)}{\la t(3+t)((8\la+1)t-5)}.\label{forztdbis}\end{equation}
Now let us find $u=u(t,\la)$ and $w$ by substituting (\ref{forztdbis}) to (\ref{uform}). One has 
$$z-1=-\frac{(\la t-1)(\la t^2+12\la t-9)+\la t(3+t)((8\la+1)t-5)}{\la t(3+t)((8\la+1)t-5)}$$
\begin{equation}=-\frac{(9\la^2+\la)t^3+(36\la^2-3\la)t^2-36\la t+9}{\la t(3+t)((8\la+1)t-5)},\label{z-1form}\end{equation}
$$-\la t(3+t)z+1-\la t=(\la t-1)\left(\frac{\la t^2+12\la t-9}{(8\la+1)t-5}-1\right)$$
\begin{equation}=\frac{(\la t-1)(\la t^2+(4\la-1)t-4)}{(8\la+1)t-5}=\frac{(\la t-1)^2(t+4)}{(8\la+1)t-5}.\label{fordenu}\end{equation}
Substituting formulas (\ref{forztdbis})--(\ref{fordenu}) to (\ref{uform}) we get (\ref{wformd}): 
$$u=-\frac{(\la t-1)^2(\la t^2+12\la t-9)^2((9\la^2+\la)t^3+(36\la^2-3\la)t^2-36\la t+9)}{\la^2t^2(3+t)^2((8\la+1)t-5)^2(\la t-1)^2(t+4)}$$
\begin{equation}=-\frac{(\la t^2+12\la t-9)^2((9\la^2+\la)t^3+(36\la^2-3\la)t^2-36\la t+9)}{\la^2t^2(3+t)^2((8\la+1)t-5)^2(t+4)},\label{forut}\end{equation}
$$w=u+z^2=\frac{(\la t^2+12\la t-9)^2}{\la^2t^2(3+t)^2((8\la+1)t-5)^2(t+4)}\left((\la t-1)^2(t+4)\right.$$
$$\left.-((9\la^2+\la)t^3+(36\la^2-3\la)t^2-36\la t+9)\right)$$
$$=-\frac{(\la t^2+12\la t-9)^2((8\la^2+\la)t^3+(32\la^2-\la)t^2-(28\la+1)t+5)}{\la^2t^2(3+t)^2((8\la+1)t-5)^2(t+4)}$$
$$=-\frac{(\la t^2+12\la t-9)^2(\la t^2+4\la t-1)}{\la^2t^2(3+t)^2((8\la+1)t-5)(t+4)}.$$
The parametrization (\ref{forztd}), (\ref{wformd}) has degree six in the sense that the intersection of its image with 
a complex line corresponds to six values of the parameter $t$ taken with multiplicities; this holds if and only if there are no cancellations in fractions (\ref{forztd}), (\ref{wformd}) reducing the degree. If this holds and in addition the sextic $\Gamma_{d,\la}$ is irreducible, this yields a one-to-one parametrization of its normalization. Below we show that this is always the case if there are no cancellations. 
 
 In the opposite case, when there are cancellations reducing the degree,  the curve $\Gamma_{d,\la}$ is reducible and the  image of the above parametrization is its irreducible component.  Let us describe the cancellation cases and prove  Statements 2)--4) of Theorem 
 \ref{tdg}. 
 
 For $\la=0$ the  curve $\Gamma_{d,\la}$ is the triple conic $\gamma$. 
For $\la=\infty$ it is the zero locus of the denominator in $R_d$, which consists of three irreducible components: 
the conic $\{ w+8z^2=0\}$, the line $\{ z=1\}$ and the rational cubic $\{ w+8z^2+4w^2+5wz^2-14zw-4z^3=0\}$.  
Its rational parametrization is given in \cite[formula (7.14)]{grat}.  This proves Statement 4) of Theorem \ref{tdg}.

Case 1): the factor $t+3$ in the denominator in (\ref{forztd}), (\ref{wformd}) cancels with a factor in the numerator. 

Subcase 1.1): it cancels with $\la t-1$. This happens exactly for $\la=-\frac13$. 
Let us calculate all the factors for this $\la$. One has 
$$\la t-1=-\frac13(t+3), \ \la t^2+12\la t-9=-\frac13(t^2+12t+27)=-\frac13(t+3)(t+9),$$ 
$$(8\la+1)t-5=-\frac53(t+3), \ \la t^2+4\la t-1=-\frac13(t^2+4t+3)=-\frac13(t+3)(t+1).$$
Substituting the above formulas to (\ref{forztd}), (\ref{wformd}) yields the parametrization
\begin{equation}z=z_1(t)=-\frac{t+9}{5t}, \ \ w=w_1(t)=-\frac{(t+9)^2(t+1)}{5t^2(t+4)}=-5z_1^2(t)\frac{t+1}{t+4}\label{pargd2}
\end{equation}
of an irreducible component $\Gamma_{d,-\frac13,1}$ of the curve $\Gamma_{d,-\frac13}$. Substituting the inverse function $t=-\frac{9}{5z+1}$ to 
the above formula for $w_1$ yields the equation $w=-\frac{z^2(5z-8)}{4z-1}\}$ of the latter component.

Let us show that the curve $\Gamma_{d,-\frac13}$ consists of the rational cubic $\Gamma_{d,-\frac13,1}$ and another 
rational cubic $\Gamma_{d,-\frac13,2}$ and find the second cubic. This is done analogously to the rescaling arguments in Case b). 
Let us choose a small $\delta\in\cc$ and introduce a rescaled parameter $\tau$, set 
$$\la=-\frac13+\delta, \ \ t+3=\delta\tau.$$
As $\delta\to0$, one has $t=-3+\delta\tau\to -3$, $\la t\to1$, $t+4\to1$, 
$$\la t-1=-\frac{\delta}3(\tau+9+O(\delta)), \ \ \la t^2+12\la t-9=-\delta(2\tau+27+O(\delta)),$$
$$(8\la+1)t-5=-\frac53(t+3)+8\delta t=-\frac{\delta}3(5\tau+72).$$
Substituting the latter formulas to (\ref{forztd}) and (\ref{wformd}) and taking limit yields a curve denoted $\Gamma_{d,-\frac13,2}$ rationally parametrized by (\ref{gd132}): 
$$z\to z_2(\tau):=\frac{(\tau+9)(2\tau+27)}{\tau(5\tau+72)},$$
\begin{equation}w\to w_2(\tau):=\frac{(2\tau+27)^2(2\tau-9)}{\tau^2(5\tau+72)}=z_2^2(\tau)\frac{(2\tau-9)(5\tau+72)}{(\tau+9)^2}.\label{wstar}\end{equation}
Thus parametrized limit curve is clearly a cubic. Let $\mathcal P(z,w)$ denote a cubic polynomial vanishing on $\Gamma_{d,-\frac13,2}$. Let us show that it takes the form (\ref{gd213}) up to constant factor. To this end let us analyze the germs of 
the cubic $\Gamma_{d,-\frac13,2}$ at $(0,0)$, $(1,1)$ and the infinite point $E=[0:1:0]$. The only parameter value corresponding to $(0,0)$ is $\tau=\tau_0:=-\frac{27}2$. The corresponding germ of parametrized cubic   $\Gamma_{d,-\frac13,2}$ is clearly regular and 
tangent to the conic $\{ w+8z^2=0\}$ with at least triple contact. This follows by straightforward calculation of its 2-jet at 
$\tau_0$, see (\ref{wstar}).  Therefore the polynomial $\mcp$ can be normalized by constant factor to have the lower $(2,1)$-quasihomogeneous part equal to $w+8z^2$.  The cubic $\Gamma_{d,-\frac13,2}$ passes through the point $E$, and the corresponding parameter value  is clearly $\tau=0$. Its germ at $E$ is clearly regular and tangent to the conic $\{ w+8z^2=0\}$ with at least triple contact. 
This follows again by (\ref{wstar}), passing to the affine coordinates $(\wt z,\wt w)=(\frac zw, \frac1w)$ centered at $E$. This implies that rewriting the polynomial $\mcp$ in the  coordinates 
$(\wt z,\wt w)$ and multiplying by $\wt w^3$ yields a polynomial in 
$(\wt z,\wt w)$ with lower (2,1)-quasihomogeneous part proportional to $\wt w+8\wt z^2$. Returning back to 
the old coordinates $(z,w)$ this yields that $\mcp$ takes the form
\begin{equation}\mcp(z,w)=w+8z^2+aw(w+8z^2)+bz^3+cwz,  \ \ \ a,b,c\in\cc.\label{mcpd2}\end{equation}
Let us find the unknown coefficients. The cubic $\Gamma_{d,-\frac13,2}$ intersects the infinity line at the point $E$ and at the point $[8:7:0]$. The latter point corresponds to the parameter $\tau=-\frac{72}5$ and is found by substituting this value to 
$$\frac{z_2(\tau)}{w_2(\tau)}=\frac{\tau(\tau+9)}{(2\tau+27)(2\tau-9)}|_{\tau=-\frac{72}5}=\frac{(-72)(-72+45)}{(-144+5*27)(-144-5*9)}=\frac87.$$
Therefore, the cubic homogeneous part $8awz^2+bz^3$ vanishes along the line $\{ 8w=7z\}$ and $b=-7a$. 

The union $\Gamma_{d,-\frac13,1}\cup\Gamma_{d,-\frac13,2}$ of distinct parametrized cubics coincides with the ambient 
level curve $\Gamma_{d,-\frac13}$, since it is a sextic. Its germ at $(1,1)$ should have multiplicity at least three at $(1,1)$, i.e., its intersection index at $(1,1)$ with a generic line should be equal to at least three. This follows from the fact that this is true for 
the zero and pole divisors of the integral $R_d$. Indeed, for the zero divisor, the triple conic $\{ w-z^2=0\}$, this is obvious. 
The pole divisor consists of the cubic $\mcs$ having a cusp at $(1,1)$, see \cite[subsection 7.6]{grat},  the line $\{ z=1\}$ 
and the conic $C$. The cubic $\Gamma_{d,-\frac13,1}$ 
has regular germ at $(1,1)$. Therefore, the cubic $\Gamma_{d,-\frac13,2}$ should have multiplicity at least two at $(1,1)$, 
i.e., the polynomial $\mcp$ should vanish there together with its both derivatives. This yields the following system of equations on the coefficients $a$, $c$, $d$: 
\begin{equation}\begin{cases} 1+8+9a-7a+c=0\\
1+(2+8)a+c=0\\
16+16a-21a+c=0.\end{cases}\end{equation}
The solution is $a=1$, $c=-11$. This yields (\ref{gd213}). 

Subcase 1.2): the factor $t+3$ cancels with a factor in the numerator of some of $z(t)$, $w(t)$ that is different from $\la t-1$. The polynomial $\la t^2+12\la t-9$ vanishes at $t=-3$ for only one $\la$; this holds for $\la=-\frac13$, as was shown above. 
Similarly the polynomial $\la t^2+4\la t-1$ vanishes at $t=-3$ only for $\la=-\frac13$. Therefore, any cancellation of the factor 
$t+3$ happens only for $\la=-\frac13$. 

Case 2): the factor $t+4$ cancels with a factor in the numerator of $w(t)$. 

Subcase 2.1): it cancels with a factor of the quadratic polynomial $\la t^2+12\la t-9$. This holds if and only if  the latter polynomial vanishes at $t=-4$, i.e., $\la=-\frac9{32}$. Let us calculate the  factors in the denominators and numerators for this $\la$. One has 
$$\la t-1=-\frac1{32}(9t+32), \ \  \la t^2+12\la t-9=-\frac9{32}(t^2+12 t+32)=-\frac9{32}(t+4)(t+8),$$
$$(8\la+1)t-5=-\frac54(t+4), \ \la t^2+4\la t-1=-\frac1{32}(9t^2+36t+32)=-\frac1{32}(3t+4)(3t+8).$$
Substituting these formulas to (\ref{forztd}) and (\ref{wformd}) yields the parametrized rational quartic $\Gamma_{d,-\frac9{32},1}$ given by (\ref{gd91}). Thus, the sextic $\Gamma_{d,-\frac9{32}}$ is a union of the quartic 
$\Gamma_{d,-\frac9{32},1}$ and a conic $\Gamma_{d,-\frac9{32},2}$ that a priori needs not be regular. The latter conic 
$\Gamma_{d,-\frac9{32},2}$ can be found as in Subcase 1.1) by taking a small $\delta\neq0$, setting 
$\la=\la_{\delta}=-\frac9{32}+\delta$ and introducing a new parameter $\tau$, $t+4=\delta\tau$. A calculation similar to that done in Subcase 1.1) shows that after the latter substitutions to the expressions $z(t)$, $w(t)$, thus obtained parametrization 
of the curve $\Gamma_{d,\la_{\delta}}$ by the parameter $\tau$ converges to a parametrization of the conic 
$\Gamma_{d,-\frac9{32},2}$ from Theorem \ref{tdg}, see (\ref{starcon}). 

Let us present another argument, which allows to find  $\Gamma_{d,-\frac9{32},2}$ via intersection indices.  
A generic  curve $\Gamma_{d,\la}$  should have: 

-  intersection index at least 6 with the conic $C=\{ w+8z^2=0\}$ at $(0,0)$, since it has two regular local branches at $(0,0)$ 
tangent to $C$ with contact of order at least three, see Proposition \ref{2-cubic} below;

- intersection index at least three with every line through $E$, see Proposition \ref{45inf} below;

- intersection index at least three with every line through the point  $(1,1)$, since it has three regular local branches there, see Proposition \ref{lgen} below. 

Passing to limit, we get that the critical level curve $\Gamma_{d,-\frac9{32}}$ also satisfies the above intersection index inequalities. The quartic $\Gamma_{d,-\frac9{32},1}$ has 

- intersection index at least 3 with $C$ at $(0,0)$; 
 
- two regular  branches at $E$ tangent to $C$, one with at least cubic contact; 

- intersection index two with a generic line through $(1,1)$. 

This follows from its parametrization (\ref{gd91}): the parameter value corresponding 
to $(0,0)$ is $t=-8$; the ones corresponding to $E$ are $t=0 ,-3$, and there are two 
parameter values corresponding to $(1,1)$ with multiplicities; below we show that the latter values coincide. 

The above intersection index and contact inequalities on $\Gamma_{d,-\frac9{32},1}$ are equalities, by B\'ezout Theorem applied to its  intersection with $C$. 
Together  with the previous inequalities, they imply 
 that the conic  $\Gamma_{d,-\frac9{32},2}$ should 
 
 - have intersection index at least three with the conic $C$ at $(0,0)$; 
  
  - pass through the point $(1,1)$; 
  
  - pass through the point $E$. 
 
 These conditions define  the conic  $\Gamma_{d,-\frac9{32},2}$ uniquely, by B\'ezout Theorem, and it is given by the 
 equation $\{ w+8z^2-9zw=0\}$. 
 
 Subcase 2.2): $t+4$ cancels with a factor of the polynomial $\la t^2+4\la t-1$, which means that $-4$ is its root. This is impossible, since its value at $-4$ is equal to  $\la 4^2-4^2\la-1=-1$. 
 
 Case 3): $(8\la+1)t-5$ cancels with some factor of the numerator in some of the fractions (\ref{forztd}), (\ref{wformd}). 
 
 Subcase 3.1): it cancels with $\la t-1$. This happens exactly when $8\la+1=5\la$, i.e., $\la=-\frac13$. This subcase is covered by Case 1).
 
 Subcase 3.2): it cancels with a factor of  the polynomial $\la t^2+12\la t-9$. This means that $\frac5{8\la+1}$ is its root. This condition is a quadratic equation in $\la$. We already know its two solutions: $\la=-\frac13, -\frac9{32}$, see  Cases 1) and 2). Therefore, this subcase is covered by Cases 1) and 2). 
 
 Subcase 3.3): it cancels with a factor of the polynomial $\la t^2+4\la t-1$, i.e.,  $\frac5{8\la+1}$ is its root, that is, 
 $$\la\frac{25}{(8\la+1)^2}+\frac{20\la}{8\la+1}-1=0, \ \ \text{i.e.,} \ \ 25\la+20\la(8\la+1)-(8\la+1)^2=0.$$
 The latter quadratic equation can be rewritten as $96\la^2+29\la-1=(3\la+1)(32\la-1)=0$, and hence, has two roots:    $\la=-\frac13$, which covered by Case 1), and $\la=\frac1{32}$, which is not covered by Cases 1), 2).  For $\la=\frac1{32}$  the  parametrization (\ref{forztd}), (\ref{wformd}) still has degree six. Indeed, the fraction $w(t)$, see  (\ref{wformd})  has the only cancellation described above and thus, has degree 5. 
 The fraction $z(t)$, see (\ref{forztd}), has no cancellations: all the possible cancellations, which are described above, correspond to $\la=-\frac13, -\frac9{32}$. But the denominator in (\ref{forztd}) contains the factor $(8\la+1)t-5$, 
 which does not enter the denominator of the degree 5 fraction (\ref{wformd}), being cancelled there. Thus, a generic linear combination of  
 fractions (\ref{forztd}), (\ref{wformd}) has degree six. 
 
 Parametrization (\ref{forztd}), (\ref{wformd}) is proved. For the proof of Statement 1) of Theorem \ref{tdg} it remains to show that if it has no cancellation  reducing degree, i.e., $\la\neq0,-\frac13, -\frac9{32}, \infty$, then it is a one-to-one parametrization of the sextic $\Gamma_{d,\la}$ up to self-intersections. Namely, we have to forbid the case, when it is a covering of degree $k\geq2$  
 over a curve of degree $\frac6k$. The  parameter value $t=-4$ is the unique value that 
 corresponds to a  local branch of its image at $E$ that is transversal to the infinity line: the other values $t=0,-3$, which are sent to $E$,  correspond to   branches tangent to it. The parametrization has nonzero derivative at $t=-4$, since $w(t)$ has a simple pole there. Hence, it is a covering of degree one over the normalization of the  sextic $\Gamma_{d,\la}$.  Statements 2)--4) of Theorem \ref{tdg} are proved above.  Theorem \ref{tdg} is proved.
 \end{proof}
 
 \begin{proof} {\bf of Theorems \ref{tdcr} and \ref{locind}.}  First we study local behavior of level curves at indeterminacies and prove Theorem \ref{locind} and 
 Statement 6) of Theorem \ref{tdcr}.  Then we find the true critical points and prove Statements 1)--5) of Theorem \ref{tdcr}. 
 
 \begin{proposition} \label{procusp} The quartic $\Gamma_{d,-\frac9{32},1}$ has a cusp at $(1,1)$. It corresponds to the 
 parameter value $t=-\frac{16}7$.
 \end{proposition}
 \begin{proof}  The quartic $\Gamma_{d,-\frac9{32},1}$ is parametrized as  
$$(z(t),w(t)), \ \ z(t)=-\frac{(9t+32)(t+8)}{40t(3+t)}, \ w(t)=-z^2(t)\frac{40(3t+8)(3t+4)}{(9t+32)^2}.$$
The equation $z(t)=1$ is equivalent to $49t^2+224t+16^2=(7t+16)^2=0$, which has unique solution $t=-\frac{16}7$. 
Equating to 1 the multiplier at $z^2$ in the expression for $w(t)$ yields the same quadratic equation. 
Therefore,  the quartic $\Gamma_{d,-\frac9{32},1}$ has unique local branch at $(1,1)$: it corresponds to the parameter 
value $t=-\frac{16}7$. It is a singular local branch, since the latter $t$ is a double root, and hence the curve has intersection index at least two with a generic line through $(1,1)$. Let us show that it is a cusp. The quartic $\Gamma=\Gamma_{d,-\frac9{32},1}$ has zero genus, being rational. 
Therefore,
\begin{equation}\frac{(4-1)(4-2)}2-\sum_{p\in Sing(\Gamma)}\delta(p)=0, \ \ \sum_p\delta(p)=3,\label{hir2}\end{equation}
by Hironaka Genus Formula (\ref{hiron}). The point $E$ is a non-transversal self-intersection, being tangency point of two regular  branches corresponding to $t=0,-3$. Hence,  
its $\delta$-invariant is no less than  two, by Remark \ref{rdelta}. 
Therefore, the $\delta$-invariant of the singular point $(1,1)$ is 
no greater than one, by (\ref{hir2}). Hence, it is equal to one, and thus, it is a cusp, by Remark \ref{rdelta} and since it is obviously not a self-intersection.  Proposition \ref{procusp} is proved.
\end{proof}
 \begin{proposition} \label{lgen} The germ of the curve $\Gamma_{d,\la}$ at $(1,1)$ consists of three regular local branches that are transversal to the conic $\gamma$, if and only if $\la\neq0,-\frac9{32},\infty$. In this case they are transversal to each other. 
  For $\la=-\frac9{32}$ the germ at $(1,1)$ of the curve $\Gamma_{d,\la}$ consists of a regular branch lying in the 
  conic $\Gamma_{d,-\frac9{32},2}$  and of a branch with cusp lying in the quartic $\Gamma_{d,-\frac9{32},2}$. Their tangent lines at $(1,1)$ are transversal. For $\la=\infty$ the germ at $(1,1)$ consists of the line $\{ z=1\}$ and the cusp of 
  the cubic $\mcs$.  The tangent line to the both above cusps are transversal to the conic $\gamma$ and to the line 
  $\{ z=1\}$.
 \end{proposition}
 \begin{proof} Let us first prove the latter statements on cusps. 
 The curve 
 $$\Gamma_{d,\infty}=\{(w+8z^2)(z-1)(w+8z^2+4w^2+5wz^2-14zw-4z^3)=0\}$$
 intersects $\gamma$ at three points: $(0,0)$, $(1,1)$  and its point $E$ at infinity. The irreducible components 
 of its germ at $(1,1)$ are the line $\{ z=1\}$ transversal to $\gamma$ and a cusp germ of the rational cubic 
 $$\mathcal S:=\{ w+8z^2+4w^2+5wz^2-14zw-4z^3=0\}.$$
 The fact that this cubic is rational and has cusp at $(1,1)$ was proved in \cite[subsection 7.6]{grat}. The cubic $\mathcal S$ cannot be tangent to $\gamma$ at its cusp $(1,1)$. Indeed, otherwise their intersection index at $(1,1)$ would be three, 
 while their intersection indices at $(0,0)$ and $E$ are at least two each \cite[subsection  7.6]{grat}, which  yields in total at least 
 $3+2+2=7>2\times 3$, -- a contradiction to B\'ezout Theorem. Similarly $\mathcal S$ cannot be tangent to the line 
 $\{ z=1\}$ at its cusp $(1,1)$, since then they would have intersection index at least three there while they also intersect at $E$, -- a contradiction to B\'ezout Theorem. Analogously the  quartic $\Gamma_{d,-\frac9{32},1}$ is also transversal to $\gamma$ and 
 $\{ z=1\}$ at its cusp $(1,1)$. Indeed, it has a local branch  at $(0,0)$ and two ones at $E$, all the three tangent to $\gamma$. 
 Hence, its intersection index with $\gamma$ at $(1,1)$ cannot be greater than two, by B\'ezout Theorem. Its transversality with the line $\{ z=1\}$ follows analogously: its two  local branches  at $E$ intersect the line with total index at least two, and thus, their  intersection index at $(1,1)$ cannot be bigger than two, since $\Gamma_{d,-\frac9{32},1}$ is a quartic. 
 
   Let us prove the first statement of the proposition. To this end, we choose local coordinates 
 $(u,v)$ centered at $(1,1)$ so that $v=w-z^2$, i.e., $\gamma=\{ v=0\}$, and  the line tangent to the cusp of the cubic 
 $\mcs\subset\Gamma_{d,\infty}$ be the $v$-axis. 
 In addition we can achieve that   the line $\{ z=1\}$ be the  line $\{ u=v\}$. Indeed, in the above coordinates $(u,v)$ 
 the line $\{ z=1\}$ is locally the graph $\{ v=h(u)\}$ of a function $h(u)$ with non-zero derivative at 0, by transversality. 
 Making further coordinate change $(u,v)\mapsto (h(u), v)$ transforms the above graph to the diagonal. 
Then for $\la\neq0$ the germ at $(1,1)$ of the curve $\Gamma_{d,\la}$ is given in the new coordinates by an equation 
 $$\Psi(u,v)=v^3-\la g(u,v)(u-v)(a u^2-b v^3+\dots)=0, \ \ a,b\neq0, \ g(0,0)\neq0.$$
Here the dots denote a term containing only monomials of bidegrees lying above the edge $[(2,0), (0,3)]$. 

The lower homogeneous part of the above expression is $G(u,v)=v^3-c(u-v)u^2$, $c=\la a g(0,0)$. We consider that 
$\la\neq0,\infty$, which is equivalent to the inequality $c\neq0,\infty$. Then each zero line of the polynomial $G$ is transversal to $\gamma=\{ v=0\}$. Let us consider two possible cases.

Case 1): the homogeneous cubic polynomial $G(u,v)$ has three distinct zero lines. Then they are simple zero lines. Hence, 
the germ at $(0,0)$ of the holomorphic function $\Psi(u,v)$ is a product of three holomorphic functions $\Psi_1$, $\Psi_2$, $\Psi_3$ vanishing at $(0,0)$, their differentials at $(0,0)$ are non-zero and their kernels are the above zero lines, by Proposition \ref{proii}, Statement 3). Thus, their zero loci are pairwise transversal.

Case 2): the polynomial $G(u,v)$ has a multiple zero line. This is equivalent to the statement that the polynomial $W(x)=x^3+c(x-1)$ has a multiple zero, i.e., $W$ and $W'$ have a common zero. This happens if and only if $c=-\frac{27}4$. 
Indeed,  their common zero  is also a zero of the polynomial $3W(x)-xW'(x)=2cx-3c=c(2x-3)$. Hence, the common zero is $x=\frac32$, and thus, $W'(x)=3x^2+c=0$,  $c=-3x^2=-\frac{27}4$. The above arguments imply that for this $c$ the latter $x$ is indeed a common zero of the polynomials $W$ and $W'$. The value $\la\neq0,\infty$ for which this takes place is unique, since $c$ is a linear function of $\la$. Let us show that it is equal to $-\frac9{32}$. Indeed, we have already shown that for $\la=-\frac9{32}$ the germ of the function $\Psi(u,v)$ 
has two local branches at $(1,1)$: the cusp of the quartic $\Gamma_{d,-\frac9{32},1}$ and a regular branch of the conic 
$\Gamma_{d,-\frac9{32},2}$, see Proposition \ref{procusp}. Hence,  the corresponding polynomial $G(u,v)$ has a multiple zero line tangent to the cusp. It has also a zero line tangent to the conic $\Gamma_{d,-\frac9{32},2}$. The latter line  
correspond to the remaining and simple root of the polynomial $W(x)=x^3-\frac{27}4(x-1)$ different from $\frac32$. 
Therefore, it does not coincide with the above multiple zero line tangent to the cusp. 
 Proposition   \ref{lgen} is proved.
 \end{proof}
\begin{proposition} \label{45inf} For $\la\in\oc\setminus\{0\}$ the germ at $E$ of the curve $\Gamma_{d,\la}$ consists of three regular local branches: 
one transversal to the infinity line; two branches tangent respectively to the conics  $C$ and $\mcc_{4,5}:=\{4w+5z^2=0\}$, both with at least cubic contact. 
\end{proposition}
\begin{proof} The curve $\Gamma_{d,\infty}$ consists of the line $\{ z=1\}$, the conic $C$ and the cubic $\mathcal S$, which  
is tangent at $E$ to the conic $\mcc_{4,5}$ with at least cubic contact, see \cite[subsection 7.6, claim 8]{grat}. This proves the 
statement of the proposition for $\la=\infty$. For $\la\neq0,\infty$ its statement can be deduced directly from parametrization 
 (\ref{forztd}), (\ref{wformd}) and description of reducible curves $\Gamma_{d,\la}$ given by Theorem \ref{tdg}. 
Another way of proof is to write down the  equation $R_d(z,w)=0$ as a polynomial equation in the affine coordinates $(\wt z,\wt w)=(\frac zw, \frac1w)$ centered at $E$. The latter polynomial equation takes the form $(\wt w-\wt z^2)^3-\la(\wt z-\wt w)(\wt w+8\wt z^2)(4\wt w+5\wt z^2+\wt w^2-4\wt z^3+O(\wt z\wt w))=0$.  The Newton diagram of the left-hand side  has vertices $(5,0)$, $(1,2)$, $(0,3)$ for $\la\neq0,-\frac14,\infty$ and $(5,0)$, $(1,2)$, $(0,4)$ for $\la=-\frac14$. In both cases   its germ at $E$  is a product of an analytic function $g(\wt z,\wt w)$, $g(0,0)=0$, with non-zero derivative in $\wt z$ at  $(0,0)$, 
and two analytic functions of the type $\wt w+8\wt z^2+h.o.t$, $4\wt w+5\wt z^2+h.o.t$, where h.o.t. means a linear combination of monomials of $(2,1)$-quasihomogeneous degrees greater than the quasihomogeneous degree 2 of the monomial $\wt w$. Their zero loci satisfy the statements of  Proposition \ref{45inf}: transversality follows from non-vanishing of the  
derivative $\frac{\partial g}{\partial\wt z}(0,0)$, since the $\wt z$-axis is the infinity line. The proposition is proved. 
\end{proof}

\begin{proposition} \label{2-cubic} 1) For every $\la\in\oc\setminus\{0,-\frac14\}$ the germ of the curve $\Gamma_{d,\la}$ at $(0,0)$ consists of two regular local branches tangent to  the conic $C=\{ w+8z^2=0\}$ with at least cubic contact and having cubic  contact between themselves.

2) For $\la=-\frac14$ the above germ is irreducible, singular and tangent to $C$. 
\end{proposition}
\begin{proof} The germ at $(0,0)$ of the curve 
$\Gamma_{d,\infty}$ is the union of the conic $C$ and the cubic 
$$\{ G_0(z,w)=0\}, \ \ G_0(z,w)=w+8z^2+4w^2+5wz^2-14zw-4z^3.$$
Substituting $w=-8z^2$ to $G_0$ yields that along the conic $C$ one has $G_0=108z^3(1+o(1))$. This proves the statement of 
the proposition for $\la=\infty$. 

Let $\la=-\frac13$. Then the cubics from Theorem \ref{tdg} forming $\Gamma_{d,\la}$ both pass through $(0,0)$ and are tangent to $C$ there with 
cubic contact, thus with at least cubic contact between themselves. They have exactly cubic contact between themselves, since higher contact is impossible by B\'ezout Theorem. 
Indeed, their 
 intersection indices  at the two other indeterminacies $(1,1)$, $E$  are  equal  respectively to two and three, by Proposition \ref{lgen} and parametrization. But they have yet another intersection point: the true critical point $(\frac25, \frac85)$ found below. Thus, their total intersection index at the points diffferent from $(0,0)$ is no less than six.  Therefore their intersection index at $(0,0)$ is at most three, and thus, their contact 
 at $(0,0)$ is at most cubic, and hence, exactly cubic. 
  The case, when $\la=-\frac9{32}$ is treated analogously. 

Let now $\la\neq0,-\frac13,-\frac9{32},\infty$. 
Consider the parametrization of the curve $\Gamma_{d,\la}$ from Theorem \ref{tdg}. The point $(0,0)$ corresponds to those  parameter values $t$ that are roots of the polynomial $Q_\la(t)=\la t^2+12\la t-9$. The exceptional case, when $t$ is a common root 
of the polynomials $\la t-1$, $\la t^2+4\la t-1$ in the numerators of the parametrization, is forbidden, since then 
$t=3$ and $\la=-\frac13$. The curve $\Gamma_{d,\la}$ has two local branches at $(0,0)$, if and only if the polynomial 
$Q_\la(t)$ has two distinct roots. This happens if and only if $\la\neq-\frac14$. For $\la=-\frac14$ only one parameter value, 
namely the double root $t=-6$ of the polynomial $Q_\la(t)$ corresponds to $(0,0)$, and hence, $\Gamma_{d,\la}$ has 
irreducible germ at $(0,0)$.  
 
The fact that for $\la\neq-\frac14$ the two above branches have exactly cubic contact with the conic $C$ and between themselves  
can be proved by a direct calculation. It also follows by the next Newton diagram argument. 
Writing the equation $R_d(z,w)=\la$ in the new variables $(z,u)$, $u:=w+8z^2$, yields 
$$(u-9z^2)^3-\la u(z-1)(u-14 z(u-8z^2)+4(u-8z^2)^2+5(u-8z^2)z^2-4z^3)$$
$$=\la u^2+108\la uz^3-(3z)^6+h.o.t.$$
Here h.o.t means monomials of higher $(3,1)$-quasihomogeneous degrees. The lower quasihomogeneous 
part of the latter polynomial is $\la u^2+108\la uz^3-(3z)^6$. It is $\la$ times a product of two prime quasihomogeneous polynomials of the type $u+c_jz^3$, $j=1,2$, $c_j\neq0$. They coincide, i.e., $c_1=c_2$, if and only if the discriminant $108^2\la^2+4*3^6\la$ vanishes. 
This happens exactly for $\la=-\frac14$, since we have already excluded $\la=0$. 
This together with Statement (ii) from 
Subsection 1.4 implies that in the former case, when $\la\neq-\frac14$, i.e., $c_1\neq c_2$, 
 the curve $\Gamma_{d,\la}$ has two local branches at $(0,0)$ that are zeros of functions of type $u+c_jz^3+h.o.t$, 
 hence, having exactly cubic contact between themselves and with $C=\{ u=0\}$.
\end{proof}

Proposition \ref{2-cubic} immediately implies that the indeterminacy point $(0,0)$ is critical exactly for 
$\la=0, -\frac14$. 

Theorem \ref{locind} and Statement 6) of Theorem \ref{tdcr} follow from Propositions \ref{procusp}--\ref{2-cubic}.
 
 Let us now find true critical points. Each of them either lies in a multiple level curve, or is an intersections of irreducible components of one and the same 
 level curve, by Proposition \ref{procritic}. Therefore, the corresponding values may be only  $\la=0,-\frac13, -\frac9{32}, \infty$, 
 by Theorem \ref{tdg}, Statement 1). Each reducible level curve except for $\Gamma_{d,\infty}$ may contain only one true critical point, since it consists of two components, see Theorem \ref{tdg}, and  by Proposition \ref{procritic}, Statement 5). 
 
 Let $\la=\infty$. Then the corresponding level curve consists of the conic $C=\{ w+8z^2=0\}$, the line $\{ z=1\}$ and 
 the cubic $\mathcal S=\{ w+8z^2+4w^2+5wz^2-14zw-4z^3=0\}$. The only point of intersection $C\cap\{ z=1\}$ different form $E$ 
 is $(1,-8)$. One has $\mathcal S\cap\{ z=1\}=\{(1,1),E\}$, by B\'ezout Theorem and since the intersection point $(1,1)$ has 
 multiplicity two being a cusp of $\mcs$. The  intersection $C\cap\mcs$ contains only one point different from the indeterminacies $(0,0)$ and $E$, by B\'ezout Theorem and since their  intersection indices are equal to 3 at $(0,0)$ and 2 at 
 $E$, see Propositions \ref{2-cubic} and \ref{45inf}. The latter intersection point is found by substituting $w=-8z^2$ to the 
 equation of $\mathcal S$, which yields
 $$4*64z^4-40z^4+8*14z^3-4z^3=4*27z^3(2z+1)=0, \ \ \ z=-\frac12, \ w=-8z^2=-2.$$ 
The critical points found above are transversal intersections of the curves in question, by B\'ezout Theorem. 
 
Let $\la=-\frac13$.  Substituting the equation $w=-\frac{z^2(5z-8)}{4z-1}$ of the cubic $\Gamma_{d,-\frac13,1}$ instead of $w$ to the equation $\mcp=0$, see (\ref{gd213}), of the cubic $\Gamma_{d,-\frac13,2}$ and cancelling $z^2$ yields
  $$-\frac{5z-8}{4z-1}\left(1+8z^2-11z-\frac{z^2(5z-8)}{4z-1}\right)+8-7z=0,$$
  \begin{equation}(5z-8)((1+8z^2-11z)(4z-1)-z^2(5z-8))+(7z-8)(4z-1)^2=0.\label{pol58}\end{equation}
  The latter polynomial is divizible by $z(z-1)^2$. Indeed, it should vanish at $0$: the polynomial $\mcp$ with 
  $w$ replaced by $-\frac{z^2(5z-8)}{4z-1}$ should have a triple zero at $0$, since the cubics $\Gamma_{d,-\frac13,1}$, 
  $\Gamma_{d,-\frac13,2}$ are tangent with order three at 0. It should have a double zero $z=1$, since the polynomial 
  $\mcp$ 
  vanishes at $(1,1)$ and has a critical point there. Indeed, the cubic $\Gamma_{d,-\frac13,2}=\{\mcp=0\}$ has two local branches at $(1,1)$. This follows by Proposition \ref{lgen}, since the other component $\Gamma_{d,-\frac13,1}$ has only one branch there. 
  Dividing polynomial (\ref{pol58}) 
  by $z(z-1)^2$ yields the polynomial $27(5z-2)$. In more detail, grouping the terms  in (\ref{pol58}) with common factor $4z-1$ transforms (\ref{pol58}) to 
  $$(4z-1)\left((5z-8)(1+8z^2-11z)+(7z-8)(4z-1)\right)-z^2(5z-8)^2.$$ 
  The multiplier at $4z-1$ given by the  big brackets is equal to 
  $$40z^3-119z^2+93z-8+28z^2-39z+8=40z^3-91z^2+54z.$$
  Thus, (\ref{pol58})  is equal to 
  $$z\left((4z-1)(40z^2-91z+54)-25z^3+80z^2-64z\right)$$
  $$=z(135z^3-324z^2+243z-54)=27z(z-1)^2(5z-2).$$
  Hence, its non-trivial root is $z=\frac25$, and the corresponding value $w$ 
  is equal to $w=-\frac{z^2(5z-8)}{4z-1}|_{z=\frac25}=\frac85$. Thus,  $(\frac25, \frac85)$ is the unique true critical point with 
  value $\la=-\frac13$. It is again Morse, by B\'ezout Theorem: the intersection indices of the cubics $\Gamma_{d,-\frac13,1}$ and $\Gamma_{d,-\frac13,2}$ at $(0,0)$, $E$ and $(1,1)$ are at least three, three and two respectively, by the above discussion, and hence, their intersection point found above has index one.

 \begin{proposition} \label{procr32} The only intersection point of the curves $\Gamma_{d,-\frac9{32},1}$ and $\Gamma_{d,-\frac9{32},2}$ different from the base points $(0,0)$, $(1,1)$, $E$ is the point $(\frac1{10},-\frac45)$ of their transversal intersection, i.e., with index 1. It corresponds to the parameter 
$t=-4$ of the quartic $\Gamma_{d,-\frac9{32},1}$. 
\end{proposition}
\begin{proof} The total intersection index of the quartic $\Gamma_{d,-\frac9{32},1}$ and the conic $\Gamma_{d,-\frac9{32},2}$ 
is equal to 8, by B\'ezout Theorem. Their intersection index at $(0,0)$ is  three, since their germs there are regular and 
tangent to each other with cubic contact, see Proposition \ref{2-cubic}. The index at $E$ is two, since the corresponding germ of the quartic 
consists of two regular branches tangent to the infinity line while the germ of the conic is regular and transversal to it. The index 
at $(1,1)$ is at least two, since $(1,1)$ lies in both curves and $\Gamma_{d,-\frac9{32},1}$ has a cusp there. Hence, 
there is at most one extra intersection point, and then it has intersection index one. Let us find it. To this end, let us substitute the parametrization (\ref{gd91}) 
of the quartic to the equation $w+8z^2-9zw=0$ of the conic. 
Cancelling the common multiplier $(t+8)^2$ and multiplying by 
the common denominator $40^2t^3(t+3)^3$ of the expression thus obtained yields the quartic equation
$$t(3+t)(-40(3t+8)(3t+4)+8(9t+32)^2)-9(t+8)(9t+32)(3t+8)(3t+4)=0$$
Opening the brackets,  dividing by 9  and changing sign yields
\begin{equation}49t^4+28*29t^3+32*3*47t^2+32*8*40t+2^{13}=0.\label{49t}\end{equation}
The polynomial (\ref{49t}) should have the root $t=-8$. Indeed, it corresponds to the point $(0,0)$ where the 
quartic is tangent to the conic with at least cubic contact, and we have already divided our equation by $(t+8)^2$. 
It has the double root 
$t=-\frac{16}7$ corresponding to the cusp $(1,1)$ of the quartic, by  Proposition \ref{procusp}. Thus, the polynomial (\ref{49t}) 
is divisible by the polynomial 
$$(t+8)(7t+16)^2=49t^3+8*77t^2+2^{11}t+2^{11}.$$
Doing the division yields that their quotient is $t+4$. Thus, $t=-4$ is a root. Substituting $t=-4$ to the parametrization (\ref{gd91}) yields $(z,w)=(\frac1{10}, -\frac45)$. This together with the above discussion imply  that the 
latter  is a  transversal intersection point of regular germs of the quartic and the conic. 
 Proposition \ref{procr32} is proved.
 \end{proof}

 Statements 1)--5) of Theorem \ref{tdcr} follow from Proposition \ref{procr32} and the above discussion. Theorem \ref{tdcr} is proved.
\end{proof}

 \begin{proof} {\bf of Theorem \ref{td}.} 
 For $\la\neq0,-\frac13, -\frac9{32},\infty$ consider the base points in the normalization of the curve $\Gamma_{d,\la}$ 
  of its following local branches: the  transversal local branch  at $E$ from Proposition \ref{45inf} and 
  the three local branches at $(1,1)$ from Proposition \ref{lgen}. The corresponding parameter values are given by (\ref{brd}), 
  which follows from (\ref{forztd}) and (\ref{z-1form}).  They are branching points of the projection $p_M:S_{d,\la}\to\Gamma_{d,\la}$, by transversality to $\gamma$ and 
  Proposition \ref{ctype}. 
 Thus, the projection has at least four branching points. This together with 
 Proposition \ref{ctype}, Statement 5) and Riemann--Hurwitz Theorem implies that  $S_{d,\la}$ is elliptic,  (\ref{brd}) are the only branching points, $S_{d,\la}$ is conformally equivalent to the curve (\ref{ctyped}), and (\ref{diffd}) is its holomorphic differential. The map 
 $F|_{S_{d,\la}}$ is a complex torus translation, by Proposition \ref{elltr}. Statement 1) of Theorem \ref{td} is proved. 
 
 Let now $\la=-\frac13$. The cubic $\Gamma_{d,-\frac13,1}$ intersects $\gamma$ at $(0,0)$, $(1,1)$ and $E$. 
 It is quadratically tangent  to $\gamma$ at its  point $(0,0)$, which is regular. It has two regular local branches at $E$: 
 one quadratically tangent to $\gamma$, one transversal to $\gamma$.  Its germ at $(1,1)$ is regular and transversal to $\gamma$ at $(1,1)$, by B\'ezout Theorem and since its intersection indices with $\gamma$ at $(0,0)$ and $E$ are equal to 
 2 and 3 respectively.  
  The only branching points of the projection $p_M^{-1}(\Gamma_{d,-\frac13,1})\to\Gamma_{d,-\frac13,1}$ correspond to the two above transversal local branches, and 
 thus, the projection has two branching points, by Proposition \ref{ctype}. Therefore the projection preimage is a double covering 
 over our rational cubic $\Gamma_{d,-\frac13,1}$ with two branching points and thus, is a rational curve. 
 
 Similarly the cubic $\Gamma_{d,-\frac13,2}$ is quadratically tangent to $\gamma$ at $(0,0)$ and $E$, which are its regular points. 
 It has two regular local branches at $(1,1)$ that are transversal to $\gamma$, by Proposition \ref{lgen}. Thus, they correspond to the only branching points of the projection, by Proposition \ref{ctype}. Analogously we get that its preimage is 
 a rational curve that is its double covering. Statement 3) of Theorem \ref{td} is proved.
 
 Let now $\la=-\frac9{32}$. The quartic $\Gamma_{d,-\frac9{32},1}$ has one local branch 
 at $(0,0)$ corresponding to $t=-8$ and  two local branches at $E$ corresponding to $t=0,-3$. All the three latter local branches are regular and quadratically  tangent to $\gamma$. It has a unique local branch at $(1,1)$, and it is a cusp. Therefore, each one of the four above local branches has intersection index two with $\gamma$, by B\'ezout Theorem. This implies that the double covering $p_M:S_{-\frac9{32}}\to\Gamma_{d,-\frac9{32},1}$ is not branched at the 
 base points of the three former local branches, by Proposition \ref{ctype}, Statement 4), and its only a priori potential branching point is $(1,1)$. But a double covering is either non-ramified, or has at least two distinct branching points. Thus, the above  covering is a regular double covering over the Riemann sphere parametrizing the quartic $\Gamma_{d,-\frac9{32},1}$. Hence, the preimage $p_M^{-1}(\Gamma_{d,-\frac9{32},1})$ 
consists of two rational components, each bijectively projected onto the above Riemann sphere. 

 The conic $\Gamma_{d,-\frac13,2}$ is clearly 
 tangent to $\gamma$ at $(0,0)$ and their other intersection points $E$ and $(1,1)$ are transversal, by B\'ezout Theorem. 
 Therefore the tangency is quadratic, and the preimage $p_M^{-1}(\Gamma_{d,-\frac13,2})$ is its double covering, branched over 
 $E$ and $(1,1)$, by Proposition \ref{ctype}, Statement 4).  Thus, the latter preimage is a rational curve. 
 
 Let  $\la=\infty$. The curve $\Gamma_{d,\infty}$ consists of the three following components:
 
 a) The conic $\{ w+8z^2=0\}$. It intersects $\gamma$ only at $(0,0)$ and $E$, and it is quadratically 
 tangent to $\gamma$ there. Hence, its $p_M$-preimage is not branched over the latter points, by Proposition 
 \ref{ctype}, Statement 4). Hence, it is a 
 union of two rational curves, each being bijectively projected onto the conic. 
 
 b) The line $L=\{ z=1\}$. It intersects $\gamma$ transversally at the points $(1,1)$ and $E$. Hence, its 
 $p_M$-preimage is a double covering over $L$ branched at the latter points, by Proposition \ref{ctype}, 
 Statement 4). Thus, it is a rational curve. 
 
 c) The rational cubic  $\mathcal S:=\{ w+8z^2+4w^2+5wz^2-14zw-4z^3=0\}$. It intersects $\gamma$ only 
 at $(0,0)$, $E$ and $(1,1)$. It has regular local branches at $(0,0)$, $E$, which are  tangent to 
 $\gamma$: in its parametrization given by \cite[formula (7.14)]{grat} they correspond to  
  $t=0,\infty$. Its local branch at $(1,1)$ is unique,  it is a cusp, see \cite[subsection 7.6]{grat}. 
 Hence, the above tangencies are quadratic, by B\'ezout Theorem. This together with Proposition \ref{ctype}, Statement 4) implies that the preimage $p_M^{-1}(\mathcal S)$ is a double covering over the Riemann sphere parametrizing $\mathcal S$ 
 that is not branched over $(0,0)$, $E$. Hence, it may have at most one branching point, which may be only the 
 cusp $(1,1)$. Thus, the  covering is not branched, as in the item a) above. Hence, the preimage is a union of two rational curves, each being bijectively projected onto  $\mathcal S$ up to self-intersections. 
 Theorem \ref{td} is proved.
 \end{proof}

 \subsection{Invariant forms.  Proofs of Theorems \ref{parea},  \ref{parea2}, \ref{twoell}  and Corollary \ref{cinvfib}}
 Theorem \ref{parea} follows from Theorem \ref{parea2}, since projective transformations between conics respect their   canonical 2-forms, see Remark \ref{remforms}.  
 
 \begin{proof} {\bf of Theorem \ref{parea2}.} We use the following three observations. 
 \begin{proposition} Fix an arbitrary point $P_0=(z_0,z_0^2)\in\gamma=\{ w=z^2\}$ and an arbitrary point $Q_0\in L_{P_0}$, set 
 $z_1:=z(Q_0)$. Set 
 $$\mcp_2(z,w):=z^2-w: \ \ \ \gamma=\{\mcp_2(z,w)=0\}.$$
 Then 
 \begin{equation}\mcp_2(Q_0)=(z_1-z_0)^2.\label{mcpz}\end{equation}
 \end{proposition}
 \begin{proof} One has 
 $$w_1:=w(Q_0)=z_0^2+2z_0(z_1-z_0)=2z_0z_1-z_0^2.$$
 Therefore, $\mcp_2(Q_0)=z_1^2-w_1=z_1^2+z_0^2-2z_0z_1=(z_1-z_0)^2$. 
 \end{proof}
 \begin{proposition} Let $\sigma:\oc_z\to\oc_z$ be a projective involution that fixes the point $z_0$ and permutes the points 
 $z_1$ and $z_1^*$. In the chart $z$ one has 
 \begin{equation}\sigma'(z_1)=-\left(\frac{z_1^*-z_0}{z_1-z_0}\right)^2.\label{sigma'}\end{equation}
 \end{proposition}
 \begin{proof} Without loss of generality we consider that $z_0=0$. An involution $\sigma$ fixes 0, if and only if it takes the form 
 $$\sigma(z)=-\frac z{1+uz}, \ \ u\in\cc.$$
 Writing the equation $\sigma(z_1)=z_1^*$ we find $u$: 
 $$u=-\frac{z_1+z_1^*}{z_1z_1^*}.$$
 One has $\sigma(z)=\frac1u(\frac1{1+uz}-1)$. Hence, 
 $$\sigma'(z)=-\frac1{(1+uz)^2}, \ \ \sigma'(z_1)=-\frac1{(1-\frac{z_1+z_1^*}{z_1^*})^2}=-\left(\frac{z_1^*}{z_1}\right)^2.$$
 This proves (\ref{sigma'}).
 \end{proof}
 \begin{proposition} \label{pjac} Consider the map $F:M\to M$ written in the local coordinates $(z,w)=(z(Q),w(Q))\in\cc^2\subset\cp^2$ given by the projection $p_M:(Q,P)\to Q$. For every $(Q,P)\in M$ set 
 \begin{equation}(Q^*,P^*)=F(Q,P),  \ \ \ z_0:=z(P), \ \ z_1:=z(Q), \ \ z_1^*:=z(Q^*).\label{q*p*}\end{equation}
  The Jacobian of the map $J:(Q,P)\mapsto(Q^*,P)$ at $(Q,P)$, also written in the coordinates $(z,w)$, 
  is equal to $-\left(\frac{z_1^*-z_0}{z_1-z_0}\right)^3$.
  \end{proposition}
 \begin{proof} In the coordinates $(z,w)$ the differential $dF(z_1,w_1)$ has eigenline $QP$ with eigenvalue 
 $-\left(\frac{z_1^*-z_0}{z_1-z_0}\right)^2$, see (\ref{sigma'}). Its action on the quotient of the latter eigenline 
 is multiplication by $\frac{z_1^*-z_0}{z_1-z_0}$. In more detail, a line $L'$ tangent to $\gamma$ at a point 
 $P'$ close to $P$ intersects the line $QP$ at a point $P''$ close to $P$ and $P'$. Fix a line $\La$ through $Q$ transversal to the line $QP$, and let $\La^*$ denote the line through $Q^*$ parallel to $\La$. The involution $\sigma_{P'}$ acts on the line $L'$, fixes the point $P'$ close to $P''$ and sends the point $Q'=L'\cap\La$ close to $Q$  to  a point in $L'$ close to
 $\wt Q:=L'\cap\La^*$. The segments $Q'Q$ and $\wt Q Q^*$ are  homothetic with
 center $P''$ and coefficient tending to $\frac{z_1^*-z_0}{z_1-z_0}$, as $P'\to P$. This implies the above statement on the 
 differential eigenvalue in the quotient, which in its turn implies the Jacobian formula from 
 Proposition \ref{pjac}. 
 \end{proof} 
 
 Let us now return back to the proof of Theorem \ref{parea2}. One has $\omega_\gamma=\omega_{par}$ up to sign, 
 by (\ref{mcpz}). It remains to prove $F$-invariance of the form $\omega_{par}$. Fix a $(Q,P)\in M$ with $Q\notin\gamma$.  Let 
 $Q^*$, $P^*$, $z_0$, $z_1$, $z_1^*$ be the same, as in (\ref{q*p*}). Set $z_0^*=z(P^*)$. Consider the restriction of the form  $(z_1^*-z_0^*)^{-3}p_M^*(dz\wedge dw)$ to $T_{(Q^*,P^*)}M$. We have to show that its pullback under the differential 
 $dF(Q,P)$ is the form $(z_1-z_0)^{-3}p_M^*(dz\wedge dw)$ on $T_{(Q,P)}M$. Indeed, one has $F=I\circ J$, where $I$ is the map $I:(Q^*,P)\mapsto(Q^*,P^*)$ and $J$ is the same, as in the above proposition. One has $(z_0^*-z_1^*)^2=(z_0-z_1^*)^2$, by (\ref{mcpz}). Therefore, $z_0^*-z_1^*=-(z_0-z_1^*)$, since $z_0^*\neq z_0$. 
 The map $I$  fixes the form $p_M^*(dz\wedge dw)$. Hence, it changes the sign of the form 
 $\omega_{par}$, by the above equality. The pullback of the form $\omega_{par}$ under the differential $dJ(Q,P)$ is 
 equal to -$\omega_{par}$, by Proposition \ref{pjac}. Finally, the differential of the map $F=I\circ J$ fixes the form 
 $\omega_{par}$. Theorem \ref{parea2}, and hence Theorem \ref{parea} is proved.
 \end{proof}

 \begin{proof} {\bf of Corollary \ref{cinvfib}.} Let $\Sigma_{ic}$ denote the union of critical and indeterminacy points of the integral  $R$, set $\wt\Sigma_{ic}:=p_M^{-1}(\Sigma_{ic})$.   The equality  
 $$d(R\circ p_M)\wedge\omega_{fib}=\omega_{par}$$
 defines the form $\omega_{fib}$ at each point $x\in M\setminus\wt\Sigma_{ic}$ uniquely up to a multiple 
 of the differential $d(R\circ p_M)(x)$. Therefore, its restriction to the tangent space at $x$ of the invariant fiber 
 through $x$ is 
 uniquely defined. Finally, $\omega_{fib}$ is a well-defined meromorphic 1-form on each fiber. 
 It is $F$-invariant, as  are its defining 2-form $\omega_{par}$ and the function $R\circ p_M$. 
 
 Let now at least one fiber $S_\la$ contain an elliptic component. Then we are in some of Cases b)--d), and $\la$ is a regular value, since each critical fiber consists of rational curves. Therefore, either we are in Cases b) or d) and $S_\la$ is one elliptic component, or we are in Case c) and $S_\la$ is a union of two elliptic curves.  Let us show that then $\omega_{fib}$ is 
 a global holomorphic differential on each component of $S_\la$.  Indeed, suppose the contrary. Then for a generic $\la$ 
  its restriction to some component $S_\la'$ of $S_\la$ has a non-empty divisor. Its degree is bounded by some number $m$, since $\omega_{fib}$ can be taken to be a global meromorphic 1-form. The divisor is $F$-invariant, as is $\omega_{fib}$. Hence, the map 
  $F:S_\la'\to S_\la'$ is a translation of elliptic curve that permutes a non-empty collection of at most $m$ points. 
Thus, $F^k=Id$ on $S_\la'$, $k=m!$.  But for every $N\in\nn$, in particular, for $N>k$,  and every $\la\neq0$ 
 small enough depending on $N$, some orbit of the map $F:S_\la'\to S_\la'$ contains at least $N$ distinct points (Lemma \ref{palt}). 
 The contradiction thus obtained shows that   the restriction 
 of the form $\omega_{fib}$ to each component of a generic fiber is a global holomorphic differential. 
 It is thus invariant under translations of 
 each component of the fiber considered as a complex torus. Corollary \ref{cinvfib} is proved.
 \end{proof}
 
 \begin{proof} {\bf of Theorem \ref{twoell}.} Let  $\Gamma_\la$ be elliptic. This happens only in Case c), and then $\la$ is a regular value. 
 Then $S_\la$ is a union of  two elliptic curves, each 
  covers $\Gamma_\la$ biholomorphically, by Theorem \ref{tc2}. Thus, on each component the projection pushforward  of the form $\omega_{fib}$ is  a well-defined holomorphic differential on  $\Gamma_\la$. Theorem \ref{twoell}  is proved. \end{proof}

 \section{Acknowledgements}
 
 I am grateful to D.V.Treschev for statement of the problem on dynamics and conservation laws in exotic rationally integrable dual  billiards. I with to thank him and  S.M.Gussein-Zade and S.K.Lando for helpful discussions.

 \end{document}